\documentclass[11pt,reqno]{amsart}
\usepackage{amsfonts,amsthm,amssymb}
\usepackage{graphicx,psfrag,overpic}
\usepackage{bm,enumerate}
\usepackage{cases}
\usepackage[all,cmtip,line]{xy}
\usepackage{color}
\usepackage{enumitem}
\usepackage{tikz}
\usepackage{cite}
\usepackage{amsmath}
\usepackage{mathrsfs}
\usepackage{appendix}
\usepackage[mathlines]{lineno}

\numberwithin{equation}{section}
\numberwithin{figure}{section}

\usepackage{graphics}
\usepackage{epsfig}
\newtheorem{theorem}{Theorem}[section]
\newtheorem{corollary}[theorem]{Corollary}
\newtheorem{lemma}{Lemma}[section]
\newtheorem{proposition}[theorem]{Proposition}
\newtheorem{remark}{Remark}[section]
\newtheorem{definition}{Definition}[section]

\newcounter{Step}[section]
\renewcommand{\theStep}{\textit{\arabic{Step}}}
\newcommand{\Step}[1][]{\refstepcounter{Step}\textit{Step~\theStep. }}

\begin{document}
\title[Convergence of Splitting Schemes for Nonlinear Dirac Equation]
{$\mathrm{L}^{2}$--convergence of the time-splitting scheme for nonlinear Dirac equation in 1+1 dimensions}%
\author[N. li]{Ningning Li}
\address{Ningning Li:\ School of Mathematics and Statistics, Northwestern Polytechnical University, Xi'an 710129, China; School of Mathematical Sciences, Fudan University, Shanghai 200433, China}
\email{\tt ningningl@nwpu.edu.cn}

\author[Y. Zhang]{Yongqian Zhang}
\address{Yongqian Zhang:\ School of Mathematical Sciences, Fudan University, Shanghai 200433, China}
\email{\tt yongqianz@fudan.edu.cn}

\author[Q. Zhao]{Qin Zhao}
\address{Qin Zhao:\ School of Mathematics and Statistics, Wuhan University of Technology, Wuhan 430070, China}
\email{\tt qzhao@whut.edu.cn}

\begin{abstract}
\,\,We study the time-splitting scheme for approximating solutions to the Cauchy problem of the nonlinear Dirac equation in 1+1 dimensions. Under the assumption that the initial data for the scheme are convergent in $\mathrm{L}^{2}(\mathbb{R})$, we prove that the approximate solutions constructed by the corresponding time-splitting scheme are strongly convergent in $\mathrm{L}^{2}(\mathbb{R}\times[0,T])$ to the global strong solution of the nonlinear Dirac equation for any $T>0$. To achieve this, we first establish the pointwise estimates for time-splitting solutions. Based on these estimates, a modified Glimm-type functional is carefully designed to show that it is uniformly bounded in time, which yields $\mathrm{L}^2$ stability estimates for the scheme. Furthermore, we prove that the set of time-splitting solutions is relatively compact in $\mathrm{C}([0,T];\mathrm{L}^{2}(\mathbb{R}))$ for any $T>0$. Finally, we show that the limit of any convergent subsequence of the time-splitting solutions is the strong solution to the Cauchy problem of the nonlinear Dirac equation.
\end{abstract}
\maketitle
\section{Introduction}
We are concerned with the $\mathrm{L}^{2}$-convergence of a time-splitting scheme for nonlinear Dirac equation in 1+1 dimensions, involving Thirring type \cite{Thirring} and Gross-Neveu type \cite{Gross} nonlinear self-interactions. Such models are widely used in quantum field theory and nonlinear wave dynamics. The nonlinear Dirac equation (NLDE) is governed by
\begin{equation}\label{eq:NLDE}
\left\{
\begin{aligned}
\partial_tu+\partial_xu=imv+i\alpha u|v|^{2}+i2\beta(\overline{u}v+u\overline{v})v,\\
\partial_tv-\partial_xv=imu+i\alpha v|u|^{2}+i2\beta(\overline{u}v+u\overline{v})u,
\end{aligned}\right.
\end{equation}
with initial data
\begin{equation}\label{eq:NLDE-initail-data}
	\left\{
\begin{aligned}
	u|_{t=0}=u_{0}(x),\\
	v|_{t=0}=v_{0}(x),
\end{aligned}\right.
\end{equation}
where $(u,v)\in \mathbb{C}^{2}$ is a complex vector, $(x,t)\in \mathbb{R}^{2}$, the variables  $t$ and $x$ represent the time and spatial coordinates, respectively. Here $m$ is a nonnegative constant, $i=\sqrt{-1}$ is the imaginary unit, and the constants $\alpha$ and $\beta$ are real. $\overline{u}$ and $\overline{v}$ denote the complex conjugate of $u$ and $v$. The terms $\alpha u|v|^{2}$ and $2\beta(\overline{u}v+u\overline{v})v$ model the Thirring type and Gross-Neveu type interactions, respectively.

The well-posedness of NLDE Cauchy problem \eqref{eq:NLDE}--\eqref{eq:NLDE-initail-data} has been extensively studied in various Sobolev spaces;
see, for example, Delgado \cite{Delgado} in $\mathrm{H}^s$ with $s \ge 1$, Selberg-Tesfahun \cite{S.SelbergandA.Tesfahun} and Huh \cite{H.Huh3} in $\mathrm{H}^s$ with $s > 1/2$, Candy \cite{T.Candy} in $\mathrm{H}^s$ with $s\ge 0$, and Huh \cite{H.Huh1}, Zhang \cite{Y.Zhang}, and Zhang-Zhao \cite{Y.ZhangandQ.Zhao1} in $\mathrm{L}^{2}$. In particular, Zhang and Zhao \cite{Y.ZhangandQ.Zhao1} established the global existence and uniqueness of strong solutions in $\mathrm{L}^{2}$ for the NLDE Cauchy problem \eqref{eq:NLDE}--\eqref{eq:NLDE-initail-data}, under the initial condition $(u_{0}(x),v_{0}(x)) \in \mathrm{L}^{2}(\mathbb{R})$; see also Lemma \ref{lem:A3}.

From the numerical perspective, various numerical methods have been proposed for both linear and nonlinear Dirac equations, including finite difference methods \cite{N.BournaveasandG.E.Zouraris,R.Hammer,FengYin}, ultra low-regularity integrators \cite{SchratzKatharinaandWangYanandZhaoXiaofei}, quantum Lattice Boltzmann scheme \cite{SAUROSUCCI1,SAUROSUCCI2} and quantum walks\cite{Masaya,Lee}.
Recently, time-splitting methods for Dirac equations have attracted growing interest, for instance, see
\cite{F.FillionandS.Succi,BaoWeiZhuCaiYongYongJiaXiaoWeiYinJia2016,W.Bao2024-1,W.Bao,W.Bao2020,BaoWeizhuCaiYongyongYinJia2021,
BaoWeizhuandYinJia,N.BournaveasandG.E.Zouraris,ZHuangSJinPA,
SchratzandX.Zhao,SLiXLiFShi2017,HeWangYang2024} and references therein. Splitting methods decompose the original system into several subproblems that are easier to solve, thereby enabling efficient and accurate numerical approximations \cite{Blanes2024}. Moreover, the time-splitting schemes can be regarded as discrete quantum walks. M. Maeda et al.\cite{Masaya} considered the continuous limit of nonlinear quantum walks and rigorously proved that the discrete solutions of nonlinear quantum walks converge to the solutions of the nonlinear Dirac equation in $\mathrm{H}^s(s\geq1)$ for a fixed time interval. In addition, for a special type of nonlinear quantum walk, the quantum Lattice Boltzmann scheme, the corresponding discrete solutions were rigorously proved to converge in the function space to the  strong solution of the nonlinear Dirac equation in \cite{N.Li}.

To the best of our knowledge, there is no rigorous proof of the strong convergence in $\mathrm{L}^{2}$-norm of splitting scheme for the NLDE Cauchy problem \eqref{eq:NLDE}--\eqref{eq:NLDE-initail-data}.
We remark that the convergence of the corresponding splitting methods
has been rigorously proved for various types of partial differential equations, for instance, Schr\"odinger equations
\cite{W.Bao2024-2,C.NeuhauserandM.Thalhammer,M.Thalhammer}, KdV equations
\cite{Holden2011MC}, Maxwell's equations \cite{JiaxiangCai},
convection--diffusion equations \cite{M.JunkandZ.Yang,Holden2010Book},
and Allen--Cahn equations \cite{D.Li}.
This paper aims to rigorously prove that the solutions of our time-splitting scheme are strongly convergent to the strong solution of the NLDE Cauchy problem \eqref{eq:NLDE}--\eqref{eq:NLDE-initail-data}.


\subsection{Time splitting schemes and main result}
We now introduce the time-splitting scheme to construct approximate solutions to the NLDE Cauchy problem \eqref{eq:NLDE}--\eqref{eq:NLDE-initail-data} with initial data $(u_0,v_0)\in \mathrm{L}^{2}(\mathbb{R})$.
The key idea is to split the system \eqref{eq:NLDE}--\eqref{eq:NLDE-initail-data} into two subproblems. The first subproblem consists of linear transport equations
\begin{equation}\label{5.2.1}
	\left\{
\begin{aligned}
	&\partial_{t}u+\partial_{x}u=0,\\
	&\partial_{t}v-\partial_{x}v=0,\\
	&(u,v)|_{t=0}=(u_{1,0},v_{1,0}),
\end{aligned}\right.
\end{equation}
while the second subproblem corresponds to solving the nonlinear differential equations
\begin{equation}\label{5.2.2}
	\left\{
\begin{aligned}
	&\frac{\mathrm{d}u}{\mathrm{d}t}=imv+i\alpha u|v|^{2}+i2\beta(\overline{u}v+u\overline{v})v,\\
	&\frac{\mathrm{d}v}{\mathrm{d}t}=imu+i\alpha v|u|^{2}+i2\beta(\overline{u}v+u\overline{v})u,\\
	&(u,v)|_{t=0}=(u_{2,0},v_{2,0}).
\end{aligned}\right.
\end{equation}

Let $S^1_t$ and $S^2_t$ be the semigroup in $\mathrm{L}^{2}(\mathbb{R})$ generated by (\ref{5.2.1}) and (\ref{5.2.2}), respectively, that is, $S^1_t(u_{1,0},v_{1,0})$ and $S^2_t(u_{2,0},v_{2,0})$ is the strong solutions to (\ref{5.2.1}) and (\ref{5.2.2}) with the initial data $(u_{i,0},v_{i,0})\in \mathrm{L}^{2}(\mathbb{R})$, $i=1,2$, respectively.

Then the construction of approximate solutions is based on the solutions of \eqref{5.2.1}--\eqref{5.2.2} with suitable initial data, namely, the semigroups $S^1_t$ and $S^2_t$.
First, given a spatial mesh size $\tau:=\Delta x>0$, which will serve as a parameter, the spatial grid points are $x_{j}=j\tau$ for $j=0,\pm1,\pm2,\ldots$. Then we choose the time step size equal to the spatial one, i.e., $\Delta t:=\Delta x=\tau>0$, and define the time steps
$t_n=n\tau$ for $n=0,1,2,\ldots$. Thus, we build the grid of mesh-points
$(x_j,t_n)$ with  $j\in\mathbb{Z}$ and $n\in\mathbb{N}$.

At the initial time $t=0$, the approximate solution $\big(u^{(\tau)},v^{(\tau)}\big)$ is defined by
\begin{equation}\label{eq:piecewise-constant-initial-data}
\big( u^{(\tau)},v^{(\tau)}\big)(x,0) =\big( u^{(\tau)},v^{(\tau)}\big)(x_{j},0),\quad x\in [j\tau,(j+1)\tau),	
\end{equation}
which is a piecewise constant approximation of the initial data $\big(u_{0},v_{0}\big)$ in the sense that
\begin{equation}\label{eq:initial-data-convergent}
\lim_{\tau\rightarrow0}\big(\|u^{(\tau)}(x,0)-u_{0}(x)\|_{\mathrm{L}^{2}(\mathbb{R})}
+\|v^{(\tau)}(x,0)-v_{0}(x)\|_{\mathrm{L}^{2}(\mathbb{R})}\big)=0.
\end{equation}

Then, on the domain $\{(x,t): x\in\mathbb{R},t\geq 0\}$, the approximate solution $\big(u^{(\tau)},v^{(\tau)}\big)$ to the NLDE Cauchy problem \eqref{eq:NLDE}--\eqref{eq:NLDE-initail-data}, constructed via the time-splitting scheme, is defined by
\begin{equation}\label{time-split}
\big(u^{(\tau)},v^{(\tau)}\big)(x,t)
= S^1_{t-n\tau}\Big( S^2_\tau S^1_\tau\Big)^n
\left(\big(u^{(\tau)},v^{(\tau)}\big)(\cdot,0)\right)
\end{equation}
for $t\in [n\tau, (n+1)\tau)$, $n=0,1,2,\ldots$.

More precisely, we present a detailed description of the time-splitting scheme in an inductive manner.
Assume that the approximate solution $\big(u^{(\tau)},v^{(\tau)}\big)$ has been defined on the domain $\{(x,t): x\in\mathbb{R}, 0\leq t\leq n\tau\}$ for some $n\geq0$, with $\big(u^{(\tau)},v^{\tau}\big)(\cdot,n\tau)\in \mathrm{L}^{2}(\mathbb{R})$. Then we determine $\big(u^{(\tau)},v^{(\tau)}\big)$ on the strip $\{(x,t):  x\in\mathbb{R}, n\tau< t\leq (n+1)\tau\}$ by the following two steps, $\mathbf{Step(n,1)}$ and $\mathbf{Step(n,2)}$.

$\mathbf{Step(n,1)}:$\ \ for $\{(x,t): x\in\mathbb{R},\ n\tau<t<(n+1)\tau,\ n\in\mathbb{N}\}$, the approximate solution $\big( u^{(\tau)},v^{(\tau)}\big) $ is defined as
\begin{equation*}
\big(u^{(\tau)},v^{(\tau)}\big)(x,t)
:=\left( u^{(\tau)}_{n,1},v^{(\tau)}_{n,1}\right)(x,t)
=S^1_{t-n\tau}\left(\big(u^{(\tau)},v^{(\tau)} \big)(\cdot,n\tau) \right),
\end{equation*}
which satisfies
\begin{align}\label{eq:subproblem1}
\begin{cases}
	\partial_{t}u_{n,1}^{(\tau)}+\partial_{x}u_{n,1}^{(\tau)}=0,\\
	\partial_{t}v_{n,1}^{(\tau)}-\partial_{x}v_{n,1}^{(\tau)}=0.
\end{cases}
\end{align}
The initial data of the system \eqref{eq:subproblem1} at the time step $t_{n}=n\tau$ is taken as
\begin{equation}\label{eq:subproblem1-initial-data}
\big( u_{n,1}^{(\tau)},v_{n,1}^{(\tau)}\big) (x,t)\Big|_{t=n\tau}=\big( u^{(\tau)},v^{(\tau)}\big) (x,n\tau).
\end{equation}

$\mathbf{Step(n,2)}:$\ \ to define the approximate solution $\big( u^{(\tau)},v^{(\tau)}\big) $ at the time step $t=(n+1)\tau$, we solve the following Cauchy problem for the system of nonlinear differential equations
\begin{equation}\label{eq:subproblem2}
	\left\{
\begin{aligned}
	\frac{\mathrm{d}u^{(\tau)}_{n,2}}{\mathrm{d}s}=imv^{(\tau)}_{n,2}+i\alpha u^{(\tau)}_{n,2}|v^{(\tau)}_{n,2}|^{2}+i2\beta\Big(\overline{u^{(\tau)}_{n,2}}v^{(\tau)}_{n,2}
+u^{(\tau)}_{n,2}\overline{v^{(\tau)}_{n,2}}\Big)v^{(\tau)}_{n,2},\\
	\frac{\mathrm{d}v^{(\tau)}_{n,2}}{\mathrm{d}s}=imu^{(\tau)}_{n,2}+i\alpha v^{(\tau)}_{n,2}|u^{(\tau)}_{n,2}|^{2}+i2\beta\Big(\overline{u^{(\tau)}_{n,2}}v^{(\tau)}_{n,2}
+u^{(\tau)}_{n,2}\overline{v^{(\tau)}_{n,2}}\Big)u^{(\tau)}_{n,2},
\end{aligned}\right.
\end{equation}
 for $(x,s)\in \mathbb{R}\times[0,\tau]$, with the initial condition
\begin{equation}\label{eq:subproblem2-initial-data}
\big(u^{(\tau)}_{n,2},v^{(\tau)}_{n,2}\big)(x,s=0)
=\lim_{t\rightarrow(n+1)\tau-}\big( u^{(\tau)},v^{(\tau)}\big)(x,t)\\
:=\big( u^{(\tau)},v^{(\tau)}\big)(x,(n+1)\tau-).
\end{equation}
We remark that, as shown in Section \ref{sec:conservation},
the approximate solution $\big(u^{(\tau)}, v^{(\tau)}\big)$
can be defined globally, and we define the approximate solution
at time $t= (n+1)\tau$ by
\begin{equation*}
\big(u^{(\tau)},v^{(\tau)}\big) (x,t=(n+1)\tau)
:=\big( u^{(\tau)}_{n,2},v^{(\tau)}_{n,2}\big)(x,s=\tau)
 =S^2_{\tau}\left(\big( u^{(\tau)},v^{(\tau)}\big)(x,(n+1)\tau-)\right).
\end{equation*}
In this way, we inductively complete the definition of our time-splitting scheme.

We call these solutions $\big(u^{(\tau)},v^{(\tau)}\big)$ \textit{time-splitting solutions} of the NLDE Cauchy problem \eqref{eq:NLDE}--\eqref{eq:NLDE-initail-data}.
Here and in the sequel, let
\begin{align*}
	\big( u_{j}^{n},v_{j}^{n}\big)
	=&\big( u^{(\tau)}(j\tau,n\tau),v^{(\tau)}(j\tau,n\tau)\big),\\
	\big( u_{j}^{n+1-},v_{j}^{n+1-}\big)
	=&\big( u^{(\tau)}(j\tau,(n+1)\tau-),v^{(\tau)}(j\tau,(n+1)\tau-)\big),
\end{align*}
and let
\begin{equation*}\big( u_{j}^{n,2}(s),v_{j}^{n,2}(s)\big)=\big( u_{n,2}^{(\tau)}(j\tau,s),v_{n,2}^{(\tau)}(j\tau,s)\big) \end{equation*}
for $j=0,\pm1,\pm2,\ldots,n=0,1,2,\ldots$, and $s\in[0,\tau]$.
Then, from the construction, we conclude that
\begin{equation}\label{eq:construction}
\big( u^{(\tau)},v^{(\tau)}\big)(x,n\tau)
=\big( u^{(\tau)}(j\tau,n\tau),v^{(\tau)}(j\tau,n\tau)\big)
=\big( u_{j}^{n},v_{j}^{n}\big),\qquad  x\in[j\tau,(j+1)\tau).
\end{equation}

The aim of this paper is to establish the convergence of the time-splitting solutions constructed above to the strong solution of the NLDE Cauchy problem \eqref{eq:NLDE}--\eqref{eq:NLDE-initail-data}. The definition of strong solutions to the NLDE Cauchy problem \eqref{eq:NLDE}--\eqref{eq:NLDE-initail-data} can be stated as follows,  in the same manner as in \cite{Schecter1977}. The corresponding definition of strong solutions for linear hyperbolic equations can be found in \cite{Lax2002}.

\begin{definition}[Strong solutions]\label{definition1.1}
	A pair of functions $(u,v)\in \mathrm{C}([0,\infty);\mathrm{L}^{2}(\mathbb{R}))$ is called a strong solution of the NLDE Cauchy problem
	\eqref{eq:NLDE}--\eqref{eq:NLDE-initail-data} on $\mathbb{R}\times[0,\infty)$ if there exists a sequence of smooth solutions $(u_{\mathrm{k}},v_{\mathrm{k}})$ such that \begin{equation*}\lim_{\mathrm{k}\rightarrow\infty}\big(\|u_{\mathrm{k}}(x,0)-u_{0}(x)\|_{\mathrm{L}^{2}(\mathbb{R})}
	+\|v_{\mathrm{k}}(x,0)-v_{0}(x)\|_{\mathrm{L}^{2}(\mathbb{R})}\big)=0,\end{equation*}
	and for any $T>0$,
	\begin{equation*}\lim_{\mathrm{k}\rightarrow\infty}\big(\|u_{\mathrm{k}}-u\|_{\mathrm{L}^{2}(\mathbb{R}\times[0,T])}
	+\|v_{\mathrm{k}}-v\|_{\mathrm{L}^{2}(\mathbb{R}\times[0,T])}\big)=0.\end{equation*}
	In particular, $(u,v)(\cdot,0)=(u_{0},v_{0})\in\mathrm{L}^{2}(\mathbb{R})$.
\end{definition}

Now, we state our main theorem.
\begin{theorem}\label{mainresult}
Suppose that $(u_{0},v_{0})\in \mathrm{L}^{2}(\mathbb{R})$, and suppose that
\begin{equation*}
	\lim_{\tau\rightarrow0+}
	\big(\|u^{(\tau)}(x,0)-u_{0}(x)\|_{\mathrm{L}^{2}(\mathbb{R})}
	+\|v^{(\tau)}(x,0)-v_{0}(x)\|_{\mathrm{L}^{2}(\mathbb{R})}\big)=0.
\end{equation*}
Let $(\hat{u},\hat{v})$ be the strong solution to the NLDE Cauchy problem \eqref{eq:NLDE}--\eqref{eq:NLDE-initail-data}, and let $\big(u^{(\tau)},v^{(\tau)}\big)$ be the time-splitting solution
to \eqref{eq:NLDE}--\eqref{eq:NLDE-initail-data} with initial data $(u^{(\tau)}(x,0),v^{(\tau)}(x,0))$, defined by the scheme \eqref{time-split} for $\tau>0$. Then, for any $T>0$,
\begin{equation*}
\lim_{\tau\rightarrow0+}
\Big(\|u^{(\tau)}-\hat{u}\|_{\mathrm{L}^{2}(\mathbb{R}\times[0,T])}
+\|v^{(\tau)}-\hat{v}\|_{\mathrm{L}^{2}(\mathbb{R}\times[0,T])}\Big)=0.
\end{equation*}
\end{theorem}

\subsection{Main difficulties and key ideas of the proof}
There are two main difficulties to study the strong convergence in $\mathrm{L}^2$ of time-splitting solutions to the NLDE Cauchy problem \eqref{eq:NLDE}--\eqref{eq:NLDE-initail-data}.
First, unlike previous implicit quantum lattice Boltzmann schemes, the time-splitting scheme decomposes the full problem into a linear transport equation with potential and a nonlinear subproblem, each of which can be solved exactly in time.
As a result, it is necessary to exploit the structure of the linear part to identify quantities that control the growth of the nonlinear terms.
Thus, the Bony-type functional and the Glimm-type functional previously used to establish compactness of solutions in the $\mathrm{L}^2$ norm, introduced in \cite{N.Li}, could not be applied directly.
The second difficulty lies in proving the uniqueness of the limit of any convergent subsequence of time-splitting solutions.
Notice that for smooth solutions the nonlinear effects propagate uniformly along characteristic directions, for time-splitting solutions the nonlinear interactions are concentrated at discrete time steps. Therefore, the fundamental challenge arises from comparing classical solutions and time-splitting solutions whose nonlinear growth mechanisms are inherently different.

To overcome these difficulties, we first observe that the time-splitting solutions conserve the $\mathrm{L}^2$ norm for each subproblem.
Using the discrete characteristic triangle $\Delta(j_{1},n_{1};n_{0})$, as shown in see Fig.~ \ref{fig1}, we establish a priori bounds for the time-splitting solutions to the NLDE Cauchy problem
\eqref{eq:NLDE}--\eqref{eq:NLDE-initail-data}.
We then introduce a modified Glimm-type functional $\mathcal{F}$ defined on discrete characteristic triangles.
By choosing a suitable weight, we show that $\mathcal{F}$ uniformly bounded in time direction, thereby obtaining local $\mathrm{L}^2$ stability bounds for the difference of two time-splitting solutions. A crucial ingredient in its proof is the introduction of a new Bony-type functional $\mathcal{Q}$ into $\mathcal{F}$, which acts as a potential to control the growth induced by the nonlinear terms.
As a consequence, by decomposing the strip $\mathbb{R}\times[0,T]$ into two subdomains, we obtain $\mathrm{L}^2$ stability estimates of time-splitting solutions for any $T>0$.
Finally, we establish uniform estimates in both space and time, ensuring that the set of time-splitting solutions is relatively compact in $\mathrm{C}([0,T];\mathrm{L}^{2}(\mathbb{R}))$ for any $T>0$.

As for the uniqueness of the limit of any convergent subsequence of time-splitting solutions, we first consider the difference between a smooth solution and a time-splitting solution in a characteristic triangle $\Lambda((j+3/2)\tau,(n+3/2)\tau;n\tau)$, see Fig.~\ref{fig1-4}. Between two consecutive time levels, we employ the method of characteristics to establish uniform estimates for the difference between the smooth solution and the solution of the subproblem \eqref{5.2.1}. At the temporal grid points, we derive pointwise estimates for the difference between the smooth solution and the subproblem solution \eqref{5.2.2}. Based on these estimates, we introduce a new Glimm-type functional $\tilde{\mathcal{F}}$. Moreover, uniform
boundedness of $\tilde{\mathcal{F}}$ is proved.
Consequently, the limit of any convergent subsequence converges strongly in $\mathrm{L}^2$ to the strong solution of \eqref{eq:NLDE}--\eqref{eq:NLDE-initail-data} as the mesh size tends to zero. It is worth pointing out that our error bounds and convergence result are valid globally in time, while the rigorous error estimates established in \cite{BaoWeiZhuCaiYongYongJiaXiaoWeiYinJia2016,BaoWeizhuCaiYongyongYinJia2021} are only local in time. All these theoretical results are collected in Theorem \ref{mainresult}. Finally, we remark that the Glimm-type functional has been used in the study of hyperbolic conservation laws; for example, see Bressan \cite{Bressan2000Book}, Dafermos \cite{Dafermos2016Book}, and Glimm \cite{Glimm1965CPAM}; while the Bony-type functional has been used in the study of discrete Boltzmann equations; for example, see Bony \cite{Bony1987} and Ha--Tzavaras \cite{Ha2003CMP}.
\subsection{Organization of the paper}
The rest of the paper is organized as follows. In Section \ref{sec:conservation}, we establish pointwise estimates for time-splitting solutions. Based on these estimates, we introduce a Glimm-type functional $\mathcal{F}$ and prove that an appropriate weight can be chosen so that $\mathcal{F}$ uniformly bounded, which yields the $\mathrm{L}^2$ stability estimates for time-splitting solutions in a characteristic triangle. Then, in Section \ref{sec:compactness}, we use these local $\mathrm{L}^2$ stability estimates to establish $\mathrm{L}^2$ stability on the strip $\mathbb{R}\times[0,T]$ by decomposing it into characteristic triangular domains. Moreover, we establish uniform estimates for time splitting solutions along characteristics and prove the relative compactness of the set of time-splitting solutions in $\mathrm{C}([0,T];\mathrm{L}^{2}(\mathbb{R}))$ for any $T>0$. The compactness result used in the proof is also stated in the appendix.
Finally, in Section \ref{sec:strong}, we show that the limit of any convergent subsequence of time-splitting solutions is coincide the strong solution to the NLDE Cauchy problem \eqref{eq:NLDE}--\eqref{eq:NLDE-initail-data} in $\mathrm{L}^2(\mathbb{R}\times[0,T])$, which implies the uniqueness of the limit.

\section{A priori estimates}
\label{sec:conservation}
In this section, we first establish estimates for time-splitting solutions along characteristic directions. We then introduce a discrete characteristic triangle $\Delta(j_{1},n_{1};n_{0})$ to derive pointwise bounds, see Fig.~\ref{fig1}. Moreover, by constructing a nonlinear Glimm-type functional $\mathcal{F}$ in $\Delta(j_{1},n_{1};n_{0})$ and
proving that the functional is uniformly bounded with respect to the time step, we obtain local $\mathrm{L}^2$ stability estimates for time-splitting solutions constructed by \eqref{time-split}.
\subsection{Estimates on the time-splitting solutions along characteristic directions}
\label{sub:characteristic-Estimates}
Since $(u_0,v_0)\in \mathrm{L}^{2}(\mathbb{R})$, it follows from \eqref{eq:piecewise-constant-initial-data} and
\eqref{eq:initial-data-convergent} that there exist constants $\bar{\tau}\in(0,1)$ and $c_0>0$ such that, for any $\tau\in(0,\bar{\tau}]$,
\begin{equation}\label{eq:c_0}
	\sum_{j=-\infty}^{+\infty}\big( |u_{j}^{0}|^{2}+|v_{j}^{0}|^{2}\big) \tau\leq c_0.
\end{equation}
Here $\big(u_j^0,v_j^0\big)=\big( u^{(\tau)},v^{(\tau)}\big)(x_{j},0)$, and \eqref{eq:c_0} provides a uniform $\mathrm{L}^2$ bound for the approximate initial data.

Given the piecewise constant initial data $\big(u_j^{0}, v_j^{0}\big)$, we first establish several properties of the solutions to the linear subproblem~\eqref{eq:subproblem1}--\eqref{eq:subproblem1-initial-data} and the nonlinear subproblem~\eqref{eq:subproblem2}--\eqref{eq:subproblem2-initial-data} in the following lemma, which ensures that the time-splitting solution $\big( u^{(\tau)},v^{(\tau)}\big)$ can be defined globally.
\begin{lemma}\label{lemma2.2}
For any $j=0,\pm1,\pm2,\ldots,n=0,1,2,\ldots$ and $\tau>0$, the time-splitting solution $\big( u^{(\tau)},v^{(\tau)}\big)$ can be defined globally. Moreover, the following properties hold:
\begin{itemize}
\item [\rm (i)]
Any solution $\big(u^{(\tau)},v^{(\tau)}\big)$ of the linear subproblem \eqref{eq:subproblem1}--\eqref{eq:subproblem1-initial-data} satisfies
\begin{equation}\label{eq:characteristic-estimate-uv}
	u_{j+1}^{n+1-}=u_{j}^{n},\ \ v_{j-1}^{n+1-}=v_{j}^{n}.
\end{equation}
\item[\rm (ii)]For $s\in[0,\tau]$, any solution $\big( u^{(\tau)}_{n,2},v^{(\tau)}_{n,2}\big)$ of the nonlinear subproblem \eqref{eq:subproblem2}--\eqref{eq:subproblem2-initial-data} satisfies
\begin{align}
&\frac{\mathrm{d}|u_{j}^{n,2}(s)|^{2}}{\mathrm{d}s}\leq
m\big( |u_{j}^{n,2}(s)|^{2}+|v_{j}^{n,2}(s)|^{2}\big) +4|\beta||u_{j}^{n,2}(s)|^{2}|v_{j}^{n,2}(s)|^{2},\label{2.2}\\
&\frac{\mathrm{d}|v_{j}^{n,2}(s)|^{2}}{\mathrm{d}s}\leq m\big( |u_{j}^{n,2}(s)|^{2}+|v_{j}^{n,2}(s)|^{2}\big) +4|\beta||u_{j}^{n,2}(s)|^{2}|v_{j}^{n,2}(s)|^{2},\label{2.3}
\end{align}
	and
\begin{equation}\label{2.1}
	\frac{\mathrm{d}|u_{j}^{n,2}(s)|^{2}}{\mathrm{d}s}+\frac{\mathrm{d}|v_{j}^{n,2}(s)|^{2}}{\mathrm{d}s}=0.
\end{equation}

\end{itemize}
\end{lemma}
\begin{proof}
The proof is carried out by induction on $n$. Suppose that $\big(u^{(\tau)},v^{(\tau)}\big)$ has been defined for $x\in \mathbb{R}$ and $0\leq t\leq n\tau$, then we consider the case $t\in(n\tau,(n+1)\tau]$.

For $x\in [j\tau,(j+1)\tau)$ and $n\tau\leq t<(n+1)\tau$, the linear transport equations \eqref{eq:subproblem1} with the constant initial data $\big( u_{j}^{n},v_{j}^{n}\big)$ can be solved by the method of characteristics, namely,
\begin{equation*}
	 u^{(\tau)}(x,n\tau)=u^{(\tau)}(x+t-n\tau,t),\quad  v^{(\tau)}(x,n\tau)=v^{(\tau)}(x-t+n\tau,t),
\end{equation*}
which also gives (\ref{eq:characteristic-estimate-uv}).
	
Then we turn to the proof of the property (ii). For any $z\in \mathbb{C}$, we denote by $\mathrm{Re}\,{z}$ the real part of $z$.
Evaluating \eqref{eq:subproblem2} at $x=j\tau$, multiplying the first and second equations by $\overline{u_{j}^{n,2}(s)}$ and $\overline{v_{j}^{n,2}(s)}$, respectively,and then taking the real part of the resulting equations, we obtain
\begin{equation}\label{eq1}
\frac{\mathrm{d}|u_{j}^{n,2}(s)|^{2}}{\mathrm{d}s}=2m\mathrm{Re}\left\lbrace iv_{j}^{n,2}(s)\overline{u_{j}^{n,2}(s)}\right\rbrace +4\beta \mathrm{Re}\left\lbrace i\big( \overline{u_{j}^{n,2}(s)}\big)  ^{2}\big( v_{j}^{n,2}(s)\big) ^{2}\right\rbrace
\end{equation}
and
\begin{equation}\label{eq2}
\frac{\mathrm{d}|v_{j}^{n,2}(s)|^{2}}{\mathrm{d}s}=2m\mathrm{Re}\left\lbrace iu_{j}^{n,2}(s)\overline{v_{j}^{n,2}(s)}\right\rbrace +4\beta \mathrm{Re}\left\lbrace i\big( u_{j}^{n,2}(s)\big) ^{2}\big( \overline{v_{j}^{n,2}(s)}\big) ^{2}\right\rbrace.
\end{equation}
Since $m\geq 0$, by Young's inequality we get \eqref{2.2} and \eqref{2.3}, and get the estimate \eqref{2.1} by summing \eqref{eq1} and \eqref{eq2}.

Moreover, the estimate \eqref{2.1} implies that the solution $\big(u_{j}^{n,2}(s),v_{j}^{n,2}(s)\big)$ is bounded for $s\in[0,\tau]$. Therefore, we can define the time-splitting solution $\big(u^{(\tau)},v^{(\tau)}\big)$ for $t=(n+1)\tau$ by $\mathbf{Step(n,2)}$. Thus, the proof of Lemma \ref{lemma2.2} is complete.
\end{proof}

Next, we aim to establish the pointwise estimates for the time-splitting solutions $\big( u^{(\tau)},v^{(\tau)}\big)$ constructed via the scheme \eqref{time-split}. To this end, by Lemma~\ref{lemma2.2}, we observe that the scheme \eqref{time-split} during the $\mathbf{Step(n,1)}$ preserves the solution values along the discrete characteristic directions.
Thus, we introduce the discrete characteristic triangles associated with $\big(u^{(\tau)}, v^{(\tau)}\big)$, defined by
\begin{equation*}
\Delta(j_{1},n_{1};n_{0})
=\left\{(j,n)\::\: j_{1}-n_{1}+n\leq j\leq j_{1}+n_{1}-n,\;
n_0 \le n \le n_1\right\}
\end{equation*}
for $j,j_1=0,\pm1,\pm2,\ldots,$ and $n,n_0,n_1=0,1,2,\ldots$, see Fig.~\ref{fig1}.

\begin{figure}[h]
	\centering
	\begin{tikzpicture}[scale=0.8]
	\draw[dashed] (-1,0) -- (9,0) node[right] {\small$n=n_0 $};
        \draw[dashed] (-1,2) -- (9,2) node[right] {\small $n=n_1$ };
	\draw[thick] (2,0) -- (4,2) -- (6,0) -- cycle;
	 (2,0)
	\node[below] at (1,0) {\small($j_{1}-n_{1}+n_{0},\,n_{0}$)};
	\fill (2,0) circle (1pt);
	\fill (4,2) circle (1pt);
	\fill (6,0) circle (1pt);
	\node[below] at (7,0) {\small($j_{1}+n_{1}-n_{0},\,n_{0}$)};
	\node[above] at (4,2) {\small($j_{1},\,n_{1}$)};
	\node at (4,0.7) {\scriptsize$\Delta(j_{1},n_{1};n_{0})$};
	\end{tikzpicture}
	\caption{\label{fig1} The discrete characteristic triangle $\Delta(j_{1},n_{1};n_{0})$}
\end{figure}

The following lemma establishes the estimates on time-splitting solutions along characteristic directions in the discrete characteristic triangle $\Delta(j_{1},n_{1};n_{0})$.
\begin{lemma}\label{lemma2.3}
Let $\big( u^{(\tau)},v^{(\tau)}\big)$ be time-splitting solutions to the $\mathrm{NLDE}$ Cauchy problem \eqref{eq:NLDE}--\eqref{eq:NLDE-initail-data}. Then for any discrete characteristic triangle $\Delta(j_{1},n_{1};n_{0})$, it holds that
\begin{equation}\label{2.4}
\sum_{j=j_{1}-n_{1}+n}^{j_{1}+n_{1}-n}\big(|u_{j}^{n}|^{2}+|v_{j}^{n}|^{2}\big)
\leq\sum_{j=j_{1}-n_{1}+n_{0}}^{j_{1}+n_{1}-n_{0}}\big( |u_{j}^{n_{0}}|^{2}+|v_{j}^{n_{0}}|^{2}\big),
	\end{equation}
and for $n_0+1\leq n\leq n_1$,
\begin{equation}\label{eq:two-side}
    \sum_{k=n_1-n+1}^{n_1-n_0}\big|u_{j_1+k}^{n_1-k}\big|
    +\sum_{k=n_1-n+1}^{n_1-n_0}\big|v_{j_1-k}^{n_1-k}\big|
    \leq \sum_{j=j_{1}-n_{1}+n_{0}}^{j_{1}+n_{1}-n_{0}}\big( |u_{j}^{n_{0}}|^{2}+|v_{j}^{n_{0}}|^{2}\big).
\end{equation}
In particular, setting $n=n_1$ in \eqref{eq:two-side} yields
\begin{equation}\label{2.5}
\sum_{k=1}^{n_{1}-n_{0}}|u_{j_{1}+k}^{n_{1}-k}|^{2}
+\sum_{k=1}^{n_{1}-n_{0}}|v_{j_{1}-k}^{n_{1}-k}|^{2}
\leq\sum_{j=j_{1}-n_{1}+n_{0}}^{j_{1}+n_{1}-n_{0}}\big( |u_{j}^{n_{0}}|^{2}+|v_{j}^{n_{0}}|^{2}\big) .\\
\end{equation}
\end{lemma}
\begin{proof}

By \eqref{2.1}, we have
\begin{equation*}
    |u_{j}^{n+1}|^{2}+|v_{j}^{n+1}|^{2}=|u_{j}^{n+1-}|^{2}+|v_{j}^{n+1-}|^{2}.
\end{equation*}
This together with \eqref{eq:characteristic-estimate-uv} yields
	\begin{equation}\label{2.7}
		|u_{j}^{n+1}|^{2}+|v_{j}^{n+1}|^{2}
        =|u_{j-1}^{n}|^{2}+|v_{j+1}^{n}|^{2}.
	\end{equation}
 For fixed $n$, taking the summation of \eqref{2.7} over $j = j_1 - n_1 + n + 1, \ldots, j_1 + n_1 - n - 1 $ yields
\begin{align}
    \sum_{j=j_1 - n_1 + n + 1}^{j_1 + n_1 - n - 1 }
    \big(|u_{j}^{n+1}|^{2}+|v_{j}^{n+1}|^{2}\big)
    =&\sum_{j=j_1 - n_1 + n + 1}^{j_1 + n_1 - n - 1 }
    \big(|u_{j-1}^{n}|^{2}+|v_{j+1}^{n}|^{2}\big)\nonumber\\
    =&\sum_{j=j_1 - n_1 + n }^{j_1 + n_1 - n - 2 }|u_{j}^{n}|^{2}
    +\sum_{j=j_1 - n_1 + n + 2}^{j_1 + n_1 - n }|v_{j}^{n}|^{2}.\label{eq:sum-uv-2}
\end{align}
Taking the summation of \eqref{eq:sum-uv-2} over $n=n_0,\ldots,\tilde{n}-1$
with $n_0+1\leq\tilde{n}\leq n_1$, we get
\begin{equation*}
   \sum_{n=n_0}^{\tilde{n}-1} \sum_{j=j_1 - n_1 + n + 1}^{j_1 + n_1 - n - 1 }
    \big(|u_{j}^{n+1}|^{2}+|v_{j}^{n+1}|^{2}\big)
    =\sum_{n=n_0}^{\tilde{n}-1}\Big(\sum_{j=j_1 - n_1 + n }^{j_1 + n_1 - n - 2 }|u_{j}^{n}|^{2}
    +\sum_{j=j_1 - n_1 + n + 2}^{j_1 + n_1 - n }|v_{j}^{n}|^{2}\Big),
\end{equation*}
which implies that
\begin{equation*}
   \sum_{n=n_0+1}^{\tilde{n}} \sum_{j=j_1 - n_1 + n }^{j_1 + n_1 - n }
    \big(|u_{j}^{n}|^{2}+|v_{j}^{n}|^{2}\big)
    \leq \sum_{n=n_0}^{\tilde{n}-1}
    \Big(
    \sum_{k=j_1 - n_1 + n }^{j_1 + n_1 - n }\big(|u_{j}^{n}|^{2}+|v_{j}^{n}|^{2}\big)
    -|u_{j_1 + n_1 - n}^{n}|^{2}
    -|v_{j_1 - n_1 + n}^{n}|^{2}\Big).
\end{equation*}

Therefore, we deduce that
\begin{equation*}
    \sum_{j=j_1 - n_1 + \tilde{n} }^{j_1 + n_1 - \tilde{n}}
    \big(|u_{j}^{\tilde{n}}|^{2}+|v_{j}^{\tilde{n}}|^{2}\big)
   \leq
   \sum_{j=j_1 - n_1 + n_0 }^{j_1 + n_1 - n_0}
   \big(|u_{j}^{n_0}|^{2}+|v_{j}^{n_0}|^{2}\big)
   -\sum_{n=n_0}^{\tilde{n}-1}|u_{j_1 + n_1 - n}^{n}|^{2}
   -\sum_{n=n_0}^{\tilde{n}-1}|v_{j_1 - n_1 + n}^{n}|^{2},
\end{equation*}
which gives estimates \eqref{2.4} and \eqref{eq:two-side} by taking $k=n_1-n$. This completes the proof.
\end{proof}
\subsection{Pointwise estimates on the time-splitting solutions}
\label{sub:pointwise-estimate}
Based on the estimates established in Lemmas~\ref{lemma2.2}--\ref{lemma2.3},
this subsection is devoted to establishing pointwise estimates for the
time-splitting solutions $\big(u^{(\tau)},v^{(\tau)}\big)$ constructed for
the NLDE Cauchy problem \eqref{eq:NLDE}--\eqref{eq:NLDE-initail-data}.

To begin with, we define a characteristic triangle $\Lambda(x_1,t_{1};t_{0})$ by
\begin{equation}\label{eq:characteristic-triangle}
	\Lambda(x_1,t_{1};t_{0})=\{ (x,t)\in\mathbb{R}\times [0,+\infty)\: : \: x_1-t_{1}+t\leq x\leq x_1+t_{1}-t,\ t_{0}\leq t\leq t_{1}\},
\end{equation}
where $x_1\in \mathbb{R}$ and $0\leq t_0\leq t_1$.

Let $T>0$. Then we consider the trapezoidal region in the characteristic triangle $\Lambda(j_{1}\tau,n_{1}\tau;n_{0}\tau)$ for an integer $0\leq n_0\leq n$ and $n\leq (T+1)/\tau$, as shown in Fig.~\ref{fig-n}.

\begin{figure}[h]
	\centering
	\begin{tikzpicture}[scale=0.9]
		\draw[densely dashed] (0,0) -- (8,0) node[right] {\small$t=n_0 \tau$};
		\draw[densely dashed] (0,2) -- (8,2) node[right] {\small$t=n_1 \tau$};
		\draw[densely dashed] (0,1.4) -- (8,1.4) node[right] {\small$t=T+1$};
		\draw[densely dashed] (0,1) -- (8,1) node[right] {\small$t=T$};
		\draw[densely dashed] (0,0.7) -- (8,0.7) node[right] {\small$t=n \tau$};
		\draw[thick] (2.7,0.7)--(5.3,0.7);
		\draw[thick] (2,0)--(2.7,0.7)--(5.3,0.7)--(6,0)--cycle;;
		\draw[thin] (3,1)--(5,1);
		\draw[thin] (3.4,1.4)--(4.6,1.4);
		\draw[] (2,0) -- (4,2) -- (6,0) -- cycle;
		\node[below] at (2,0) {\small($(j_{1}-n_{1}+n_{0})\tau,\,n_{0}\tau$)};
		\node[below] at (6,0) {\small($(j_{1}+n_{1}-n_{0})\tau,\,n_{0}\tau$)};
		\node[above] at (4,2) {\small($j_{1}\tau,\,n_{1}\tau$)};
	\end{tikzpicture}
	\caption{\label{fig-n} The characteristic triangle $\Lambda(j_{1}\tau,n_{1}\tau;n_{0}\tau)$}
\end{figure}

The following lemma provides uniform pointwise bounds for time-splitting solutions.
\begin{lemma}\label{lem:pointwise-estimates}
Let $\tau\in (0,\bar{\tau}]$ and let
\begin{equation}\label{guodu1}
\mathfrak{s}_{j}^{n,2}(s)=|u_{j}^{n,2}(s)|^{2}+|v_{j}^{n,2}(s)|^{2},\quad s\in[0,\tau].
\end{equation}
Then for any $j=0,\pm1,\pm2,\ldots$, $n=0,1,2,\ldots$ and $s\in[0,\tau]$, there exist constants $c_{1},m_1>0$, depending only on $c_0$ and the system \eqref{eq:NLDE}, such that
\begin{align} |u_{j}^{n,2}(s)|^{2}\leq|u_{j}^{n,2}(0)|^{2}+m_{1}\mathfrak{s}_{j}^{n,2}(0)\tau
+c_{1}|u_{j}^{n,2}(0)|^{2}|v_{j}^{n,2}(0)|^{2}\tau,\label{2.10}\\
	|v_{j}^{n,2}(s)|^{2}\leq|v_{j}^{n,2}(0)|^{2}+m_{1}\mathfrak{s}_{j}^{n,2}(0)\tau
+c_{1}|u_{j}^{n,2}(0)|^{2}|v_{j}^{n,2}(0)|^{2}\tau.\label{2.11}
\end{align}

Furthermore, let $T>0$ and let $n$ be an integer satisfying $0\leq n_0\leq n$ and $n\leq (T+1)/\tau$, see Fig.~\ref{fig-n}. Then there exists a constant $c_2(T)>0$, depending only on $T,c_0$ and the system \eqref{eq:NLDE}, such that
	\begin{align}
		|u_{j}^{n}|^{2}\leq&
        c_{2}(T)\big( |u_{j-n}^{0}|^{2}+m_{1}c_0\big), \label{2.8}\\
		|v_{j}^{n}|^{2}\leq &
        c_{2}(T)\big( |v_{j+n}^{0}|^{2}+m_{1}c_0\big).\label{2.9}
	\end{align}
Here $\bar{\tau}\in(0,1)$ and $c_0>0$ are given by \eqref{eq:c_0}.
\end{lemma}
\begin{proof}
	Recalling the notation $\mathfrak{s}_{j}^{n,2}(s)$ in \eqref{guodu1} and using \eqref{2.1}, we have
	\begin{equation*}\mathfrak{s}_{j}^{n,2}(s)
    =|u_{j}^{n,2}(s)|^{2}+|v_{j}^{n,2}(s)|^{2}
    =|u_{j}^{n,2}(0)|^{2}+|v_{j}^{n,2}(0)|^{2}
    =\mathfrak{s}_{j}^{n,2}(0),\quad s\in[0,\tau].\end{equation*}
	This together with \eqref{2.2} and \eqref{2.3} yields
	\begin{align*}
		\frac{\mathrm{d}\big(|u_{j}^{n,2}(s)|^{2}|v_{j}^{n,2}(s)|^{2}\big)}{\mathrm{d}s}
        =&\frac{\mathrm{d}|u_{j}^{n,2}(s)|^{2}}{\mathrm{d}s}|v_{j}^{n,2}(s)|^{2}
        +\frac{\mathrm{d}|v_{j}^{n,2}(s)|^{2}}{\mathrm{d}s}|u_{j}^{n,2}(s)|^{2}
        \\
		\leq
        & \big(m\mathfrak{s}_{j}^{n,2}(0)+4|\beta||u_{j}^{n,2}(s)|^{2}|v_{j}^{n,2}(s)|^{2}\big)\mathfrak{s}_{j}^{n,2}(0).
	\end{align*}
	Then, using Gronwall's inequality and the fact that $s\in[0,\tau]$, we have
	\begin{equation}\label{2.14}
		\begin{aligned}
			|u_{j}^{n,2}(s)|^{2}|v_{j}^{n,2}(s)|^{2}
			\leq\mathrm{e}^{4|\beta|\mathfrak{s}_{j}^{n,2}(0)\tau}\left[|u_{j}^{n,2}(0)|^{2}|v_{j}^{n,2}(0)|^{2}
+m\big(\mathfrak{s}_{j}^{n,2}(0)\big)^{2} \tau\right].
		\end{aligned}
	\end{equation}
Since $\tau\in(0,\bar{\tau}]$, it follows from \eqref{eq:characteristic-estimate-uv} and \eqref{2.4} that
\begin{equation}\label{3.16.5}
	\begin{aligned}
		\mathfrak{s}_{j}^{n,2}(0) \tau
		=&(|u_{j}^{n,2}(0)|^2+|v_{j}^{n,2}(0)|^2) \tau
        =(|u_{j}^{n+1-}|^2+|v_{j}^{n+1-}|^2)\tau\\
		=&(|u_{j-1}^{n}|^2+|v_{j+1}^{n}|^2)\tau
		\leq
        \sum_{j=-\infty}^{+\infty}\big( |u_{j}^{0}|^{2}+|v_{j}^{0}|^{2}\big)  \tau
        \leq c_0,
    \end{aligned}
	\end{equation}
where the last inequality is a consequence of \eqref{eq:c_0}. Plugging  \eqref{3.16.5} into \eqref{2.14}, we have
	\begin{equation}\label{2.14-1}
		\begin{aligned}
			|u_{j}^{n,2}(s)|^{2}|v_{j}^{n,2}(s)|^{2}
			\leq\mathrm{e}^{4|\beta|c_0}\left[|u_{j}^{n,2}(0)|^{2}|v_{j}^{n,2}(0)|^{2}+mc_0\mathfrak{s}_{j}^{n,2}(0)\right].
		\end{aligned}
	\end{equation}

Substituting \eqref{2.14-1} into \eqref{2.2}, we deduce that
	\begin{equation}\label{2.15.777}
		\frac{\mathrm{d}|u_{j}^{n,2}(s)|^{2}}{\mathrm{d}s}
		\leq m\mathfrak{s}_{j}^{n,2}(0)	+4|\beta|\mathrm{e}^{4|\beta|c_0}\big(|u_{j}^{n,2}(0)|^{2}|v_{j}^{n,2}(0)|^{2}
+mc_0\mathfrak{s}_{j}^{n,2}(0)\big).
	\end{equation}
Since $s\in[0,\tau]$, integrating \eqref{2.15.777} over $[0,s]$ yields
\begin{equation}\label{eq:u-estimate}
	\begin{aligned}
		|u_{j}^{n,2}(s)|^{2}
        \leq&
        |u_{j}^{n,2}(0)|^{2}+
        m\big(|u_{j}^{n,2}(0)|^{2}+|v_{j}^{n,2}(0)|^{2}\big)\tau	\\
        &+4|\beta|\mathrm{e}^{4|\beta|c_0}\Big(|u_{j}^{n,2}(0)|^{2}|v_{j}^{n,2}(0)|^{2}+mc_0\big(|u_{j}^{n,2}(0)|^{2}+|v_{j}^{n,2}(0)|^{2}\big)\Big)\tau\\
        \leq&\big( 1+m_{1}\tau+c_{1}|v_{j}^{n,2}(0)|^{2}\tau\big) |u_{j}^{n,2}(0)|^{2}
		+m_{1} |v_{j}^{n,2}(0)|^{2}\tau,
	\end{aligned}
    \end{equation}
where $c_{1}=4|\beta|\mathrm{e}^{4|\beta|c_0}$ and $m_{1}=m(c_{1}c_{0}+1)$.
This completes the proof of \eqref{2.10}.
	
Next, we prove \eqref{2.8}. Since $u_{j}^{n,2}(\tau)=u_j^{n+1}$ and $(u_{j}^{n,2}(0),v_{j}^{n,2}(0))=(u_{j-1}^n,v_{j+1}^n)$, setting $s=\tau$ and applying an induction argument to \eqref{eq:u-estimate} with respect to $n$, we obtain
\begin{align*}
|u_{j}^{n+1}|^{2}
=&\big( 1+m_{1}\tau+c_{1}|v_{j+1}^n|^{2}\tau\big) |u_{j-1}^n|^{2}
		+m_{1} |v_{j+1}^n|^{2}\tau\\
\leq&
\mathrm{e}^{m_{1}\tau+c_{1}|v_{j+1}^n|^{2}\tau} |u_{j-1}^n|^{2}
+m_{1} |v_{j+1}^n|^{2}\tau\\
\leq&
\mathrm{e}^{m_{1}\tau+c_{1}|v_{j+1}^n|^{2}\tau}
\Big(
\mathrm{e}^{m_{1}\tau+c_{1}|v_{j}^{n-1}|^{2}\tau} |u_{j-2}^{n-1}|^{2}
+m_{1} |v_{j}^{n-1}|^{2}\tau
\Big)
+m_{1} |v_{j+1}^n|^{2}\tau\\
\leq&
\mathrm{e}^{m_{1}(n+1)\tau+c_{1}\sum_{p=0}^{n}|v_{j+1-p}^{n-p}|^{2}\tau}
\Big( |u_{j-n-1}^{0}|^{2}+m_{1}\sum_{p=0}^{n}|v_{j+1-p}^{n-p}|^{2}\tau\Big).
\end{align*}
Then, by setting $j_1=j+2,n_1=n+1,n_0=0,k=p+1$ in \eqref{2.5}, we deduce from \eqref{eq:c_0} that
\begin{equation*}
    \sum_{p=0}^{n}|v_{j+1-p}^{n-p}|^{2}\tau
=\sum_{k=1}^{n_{1}}|v_{j_{1}-k}^{n_{1}-k}|^{2}\tau
\leq
\sum_{j=j_{1}-n_{1}}^{j_{1}+n_{1}}\big( |u_{j}^{0}|^{2}+|v_{j}^{0}|^{2}\big) \tau
\leq
\sum_{j=-\infty}^{+\infty}\big( |u_{j}^{0}|^{2}+|v_{j}^{0}|^{2}\big) \tau
\leq c_0.
\end{equation*}
Therefore, since $n\leq (T+1)/\tau$ and $\tau<1$, we obtain
\begin{equation*}
    |u_{j}^{n+1}|^{2}\leq
\mathrm{e}^{m_{1}(T+2)+c_{1}c_0}\big( |u_{j-n-1}^{0}|^{2}+m_{1}c_0\big),
\end{equation*}
which gives \eqref{2.8} by taking $c_{2}(T)=\mathrm{e}^{m_{1}(T+2)+c_{1}c_0}$.

For $v_{j}^{n,2}$, the estimates \eqref{2.11} and \eqref{2.9} follow similarly as in \eqref{2.10} and \eqref{2.8}, by using \eqref{2.3} in place of \eqref{2.2}. The proof of Lemma \ref{lem:pointwise-estimates} is completed.
\end{proof}

As a corollary, the above lemma yields the following result on the nonlinear terms of the time-splitting scheme~\eqref{time-split}, which will be used to establish the $\mathrm{L}^{2}$ stability estimates.
\begin{lemma}\label{lem:nonlinear-estimates}
Let $\bar{\tau}\in(0,1)$ be the constant given by \eqref{eq:c_0}. For any $T>0$, there exists a constant $c(T)>0$, depending only on $T$, $c_0$ and the system \eqref{eq:NLDE}, such that for all $\tau\in(0,\bar{\tau}]$ and all integers $n$ satisfying
$0\leq n_{0} \leq n \leq n_{1}$, $n\leq (T+1)/\tau$, with the range of $n$ shown in Fig.~\ref{fig-n}, it holds that
	\begin{equation*}
		\sum_{p=n_{0}}^{n}\sum_{j=-\infty}^{+\infty}|u_{j}^{p}|^{2}|v_{j}^{p}|^{2}\tau^{2}\leq c(T).
	\end{equation*}
\end{lemma}
\begin{proof}
Using the estimate \eqref{2.8} in Lemma \ref{lem:pointwise-estimates}, we have
\begin{equation*}
\sum_{p=n_{0}}^{n}\sum_{j=-\infty}^{+\infty}
|u_{j}^{p}|^{2}|v_{j}^{p}|^{2}\tau^{2}
\leq
\sum_{p=n_{0}}^{n}\sum_{j=-\infty}^{+\infty}
\Big(c_{2}(T)|u_{j-p}^{0}|^{2}|v_{j}^{p}|^{2}\tau^{2}
+c_{2}(T)m_{1}c_0|v_{j}^{p}|^{2}\tau^{2}\Big)
.
\end{equation*}
Then by \eqref{2.9} in Lemma \ref{lem:pointwise-estimates}, we get
\begin{align}\label{eq:u2v2}
&\sum_{p=n_{0}}^{n}\sum_{j=-\infty}^{+\infty}
|u_{j}^{p}|^{2}|v_{j}^{p}|^{2}\tau^{2}\nonumber\\
\leq&
\sum_{p=n_{0}}^{n}\sum_{j=-\infty}^{+\infty}
\Big(c_{2}(T)|u_{j-p}^{0}|^{2}c_2(T)
\big( |v_{j+p}^{0}|^{2}+m_{1}c_0\big)
+c_{2}(T)m_{1}c_0|v_{j}^{p}|^{2}\Big)\tau^{2}
\nonumber\\
\leq&
\sum_{p=n_{0}}^{n}\sum_{j=-\infty}^{+\infty}
\Big(c_{2}^2(T)\Big(|u_{j-p}^{0}|^{2}|v_{j+p}^{0}|^{2}
+m_{1}c_0|u_{j-p}^{0}|^{2}\Big)
+c_{2}(T)m_{1}c_0|v_{j}^{p}|^{2}\Big)\tau^{2}
\\
\leq&
\sum_{p=n_{0}}^{n}\sum_{j=-\infty}^{+\infty}
\Big(c_{2}^2(T)\Big(|u_{j-p}^{0}|^{2}|v_{j+p}^{0}|^{2}
+m_{1}c_0\big(|u_{j-p}^{0}|^{2}+|v_{j-p}^{0}|^{2}\big)\Big)
+c_{2}(T)m_{1}c_0\big(|u_{j}^{p}|^{2}+|v_{j}^{p}|^{2}\big)\Big)\tau^{2}.\nonumber
\end{align}

We next estimate the right-hand side of
\eqref{eq:u2v2} one by one.
First, setting $\tilde{p}=j-p$ and using \eqref{eq:c_0}, we conclude that
\begin{align}\label{eq:u02-v02}	
\sum_{p=n_{0}}^{n}\sum_{j=-\infty}^{+\infty}|u_{j-p}^{0}|^{2}|v_{j+p}^{0}|^{2}\tau^{2}
=&\sum_{p=n_{0}}^{n}\sum_{\tilde{p}=-\infty}^{+\infty}|u_{\tilde{p}}^{0}|^{2}|v_{\tilde{p}+2p}^{0}|^{2}\tau^{2}
=\sum_{\tilde{p}=-\infty}^{+\infty}\Big(|u_{\tilde{p}}^{0}|^{2}\tau
\big(\sum_{p=n_{0}}^{n}|v_{\tilde{p}+2p}^{0}|^{2}\tau\big)\Big)\nonumber\\
\leq&
\sum_{\tilde{p}=-\infty}^{+\infty}\Big(|u_{\tilde{p}}^{0}|^{2}\tau
\big(\sum_{k=-\infty}^{+\infty}|v_{k}^{0}|^{2}\tau\big)\Big)
\leq c_0^2.
\end{align}
Then, it follows from \eqref{eq:c_0} that
\begin{equation}\label{eq:0}
   \sum_{j=-\infty}^{+\infty}\big(|u_{j-p}^{0}|^{2}+|v_{j-p}^{0}|^{2}\big)\tau
    =\sum_{j=-\infty}^{+\infty}\bigl(|u_{j}^{0}|^{2}+|v_{j}^{0}|^{2}\bigr)\tau\leq c_0.
\end{equation}
Moreover, in view of \eqref{2.7}, we obtain
\begin{equation*}
\begin{aligned}
\sum_{j=-\infty}^{+\infty}\bigl(|u_{j}^{p+1}|^{2}+|v_{j}^{p+1}|^{2}\bigr)
=\sum_{j=-\infty}^{+\infty}|u_{j-1}^{p}|^{2}
+\sum_{j=-\infty}^{+\infty}|v_{j+1}^{p}|^{2}
=\sum_{j=-\infty}^{+\infty}\bigl(|u_{j}^{p}|^{2}+|v_{j}^{p}|^{2}\bigr).
\end{aligned}
\end{equation*}
Consequently, by induction on $p$ and by \eqref{eq:c_0}, we conclude that
\begin{equation}\label{eq:l2-conserve-i}
\sum_{j=-\infty}^{+\infty}\bigl(|u_{j}^{p}|^{2}+|v_{j}^{p}|^{2}\bigr)\tau
=
\sum_{j=-\infty}^{+\infty}\bigl(|u_{j}^{0}|^{2}+|v_{j}^{0}|^{2}\bigr)\tau\leq c_0,
\qquad p\ge0.
\end{equation}

Plugging  \eqref{eq:u02-v02}, \eqref{eq:0} and \eqref{eq:l2-conserve-i}	into \eqref{eq:u2v2}, we obtain
\begin{align*}
\sum_{p=n_{0}}^{n}\sum_{j=-\infty}^{+\infty}
|u_{j}^{p}|^{2}|v_{j}^{p}|^{2}\tau^{2}
\leq&
c_{2}^2(T)\Big(c_0^2
+m_{1}c_0^2(n-n_0)\tau\Big)
+c_{2}(T)m_{1}c_0^2(n-n_0)\tau
\\
\leq&
c_{2}^2(T)c_{0}^2+c_{2}(T)m_{1}c_0^2(c_{2}(T)+1)(T+1),
\end{align*}
where we have used the fact that $(n-n_0)\tau\leq T+1$.
Therefore, Lemma \ref{lem:nonlinear-estimates} follows by taking $c(T)=c_0 ^{2}c_{2}(T)\big(c_{2}(T)+m_{1}(c_{2}(T)+1)(T+1)\big)$.
\end{proof}
\subsection{$\mathrm{L}^2$ stability estimates on the time-splitting solutions in a characteristic triangle}
\label{subsec:Glimm-functional}
This subsection is devoted to establishing the $\mathrm{L}^2$ stability estimates for the time-splitting solutions under perturbations of the initial data in a discrete characteristic triangle $\Delta(j_{1},n_{1};n_{0})$, as shown in Fig.~\ref{fig1}. To begin with, we consider pointwise estimates for the difference between two time-splitting solutions.

Let $(u^{(\tau)},v^{(\tau)})$ and $(\tilde{u}^{(\tau)}, \tilde{v}^{(\tau)})$ denote the time-splitting solutions to \eqref{eq:NLDE} with initial data $(u_0,v_0)\in{\mathrm{L}^{2}(\mathbb{R})}$ and $(\tilde u_0,\tilde v_0)\in{\mathrm{L}^{2}(\mathbb{R})}$, respectively.
Note that the piecewise constant approximations of the initial data
$(u_0,v_0)$, denoted by
$(u^{(\tau)}(x,0),v^{(\tau)}(x,0))$, are given by
\eqref{eq:piecewise-constant-initial-data} and
\eqref{eq:initial-data-convergent} for \(x\in\mathbb{R}\).

The piecewise constant approximations $
(\tilde{u}^{(\tau)}(\cdot,0),\tilde{v}^{(\tau)}(\cdot,0))
$ of the initial data $(\tilde u_0,\tilde v_0)$
are defined to be piecewise constant on the partition
\(\{[j\tau,(j+1)\tau)\}_{j\in\mathbb Z}\),
namely,
\begin{equation*}
\big( \tilde{u}^{(\tau)},\tilde{v}^{(\tau)}\big)(x,0)
=\big( \tilde{u}^{(\tau)},\tilde{v}^{(\tau)}\big)(j\tau,0),
\quad x\in [j\tau,(j+1)\tau), \; j\in\mathbb{Z},
\end{equation*}
and satisfy
\begin{equation*}
\lim_{\tau\to0}
\left(
\|\tilde{u}^{(\tau)}(\cdot,0)-\tilde{u}_{0}\|_{\mathrm{L}^{2}(\mathbb{R})}
+\|\tilde{v}^{(\tau)}(\cdot,0)-\tilde{v}_{0}\|_{\mathrm{L}^{2}(\mathbb{R})}
\right)=0.
\end{equation*}
Since the proofs of Lemmas \ref{lemma2.2}--\ref{lem:nonlinear-estimates} depend only on the time-splitting scheme and on the initial data, their conclusions remain valid for the solution $(\tilde{u}^{(\tau)},\tilde{v}^{(\tau)})$.

To simplify the notation, we drop the index $(\tau)$ whenever no confusion arises from this subsection, namely, we write $(u,v)$ and $(\tilde{u},\tilde{v})$ in place
of $(u^{(\tau)},v^{(\tau)})$ and $(\tilde{u}^{(\tau)}, \tilde{v}^{(\tau)})$. Therefore, the difference between these two time-splitting solutions is denoted by
\begin{equation*}
(U,V)=(u-\tilde{u},v-\tilde{v}),
\end{equation*}
with
\begin{equation*}
(U_{n,1},V_{n,1})(x,t)=(u_{n,1}-\tilde{u}_{n,1},v_{n,1}-\tilde{v}_{n,1})(x,t),\quad
(x,t)\in\mathbb{R}\times[n\tau,(n+1)\tau),
\end{equation*}
and
\begin{equation*}
(U_{n,2},V_{n,2})(x,s)=(u_{n,2}-\tilde{u}_{n,2},v_{n,2}-\tilde{v}_{n,2})(x,s), \quad (x,s)\in\mathbb{R}\times[0,\tau].
\end{equation*}

Since $(u_{n,1},v_{n,1})$ and $(\tilde{u}_{n,1},\tilde{v}_{n,1})$ are determined by system~\eqref{eq:subproblem1}, then
\begin{align}\label{eq:U-linear}
\begin{cases}
	\partial_{t}U_{n,1}+\partial_{x}U_{n,1}=0,\\
	\partial_{t}V_{n,1}-\partial_{x}V_{n,1}=0.
\end{cases}
\end{align}
In addition, $(u_{n,2},v_{n,2})$ and $(\tilde{u}_{n,2},\tilde{v}_{n,2})$ are determined by system~\eqref{eq:subproblem2}, then
\begin{equation}\label{3.1}	\frac{\mathrm{d}U_{n,2}}{\mathrm{d}s}=imV_{n,2}+i(\alpha+2\beta)\Big(u_{n,2}\big|v_{n,2}\big|^{2}
-\tilde{u}_{n,2}\big|\tilde{v}_{n,2}\big|^{2}\Big)	+i2\beta\Big(\overline{u_{n,2}}\big(v_{n,2}\big)^{2}-\overline{\tilde{u}_{n,2}}\big(\tilde{v}_{n,2}\big)^{2}\Big)
\end{equation}
and
\begin{equation}\label{3.2} \frac{\mathrm{d}V_{n,2}}{\mathrm{d}s}=imU_{n,2}+i(\alpha+2\beta)\Big(v_{n,2}\big|u_{n,2}\big|^{2}
-\tilde{v}_{n,2}\big|\tilde{u}_{n,2}\big|^{2}\Big)\\ +i2\beta\Big(\overline{v_{n,2}}\big(u_{n,2}\big)^{2}
-\overline{\tilde{v}_{n,2}}\big(\tilde{u}_{n,2}\big)^{2}\Big).
\end{equation}
For $j=0,\pm1,\pm2,\ldots,n=0,1,2,\ldots,$ and $s\in[0,\tau]$, we set
\begin{align*}
	(U_{j}^{n},V_{j}^{n})=(U(j\tau,n\tau),V(j\tau,n\tau)),\quad
	(U_{j}^{n+1-}, V_{j}^{n+1-})=(U(j\tau,(n+1)\tau-),V(j\tau,(n+1)\tau-)),
\end{align*}
and
\begin{equation*}
(U_{j}^{n,2}(s),V_{j}^{n,2}(s))=(U_{n,2}(j\tau,s),V_{n,2}(j\tau,s)).
\end{equation*}

We introduce the following notation associated with the nonlinear subproblem \eqref{eq:subproblem2}--\eqref{eq:subproblem2-initial-data} and \eqref{3.1}--\eqref{3.2}.
\begin{definition}\label{def:functional-uv}
Let $(j,n)$ be a pair of integers with $ n\geq 0$, and let $s\in[0,\tau]$. For solutions $(u_{n,2},v_{n,2})$ and $(\tilde{u}_{n,2}, \tilde{v}_{n,2})$ to the nonlinear subproblem \eqref{eq:subproblem2}--\eqref{eq:subproblem2-initial-data}, we define
\begin{equation*}\mathfrak{s}_{j}^{n,2}(s)=|{u}_{j}^{n,2}(s)|^{2}+|{v}_{j}^{n,2}(s)|^{2},\quad \tilde{\mathfrak{s}}_{j}^{n,2}(s)=|\tilde{u}_{j}^{n,2}(s)|^{2}+|\tilde{v}_{j}^{n,2}(s)|^{2}.\end{equation*}
For solutions $(U_{n,2},V_{n,2})$ to the problem \eqref{3.1}--\eqref{3.2}, we define
\begin{equation*}\mathcal{L}_{j}^{n,2}(s)=|U_{j}^{n,2}(s)|^{2}+|V_{j}^{n,2}(s)|^{2}\end{equation*}
and
\begin{equation*}\mathcal{D}_{j}^{n,2}(s)
=|U_{j}^{n,2}(s)|^{2}
\big(|v_{j}^{n,2}(s)|^{2}+|\tilde{v}_{j}^{n,2}(s)|^{2}\big)
+|V_{j}^{n,2}(s)|^{2}
\big(|u_{j}^{n,2}(s)|^{2}+|\tilde{u}_{j}^{n,2}(s)|^{2}\big).\end{equation*}
\end{definition}
With the above notations, the following lemma establishes pointwise estimates for the difference between two time-splitting solutions.
\begin{lemma}\label{lem:pointwise-estimates-UV}
	Let $\bar{\tau}\in(0,1)$ be the constant given by \eqref{eq:c_0}.
	Then for $j=0,\pm1,\ldots,n=0,1,\ldots$, $s\in[0,\tau]$ and $\tau\in(0,\bar{\tau}]$, there exist constants $C_1,C_2>0$, depending only on $c_0$ and the system \eqref{eq:NLDE}, such that the following estimates hold.
 \begin{itemize}
 	\item [\textup{(i)}] For linear subproblem \eqref{eq:U-linear},
 	\begin{equation}\label{eq:UV-sub1}
 		U_{j}^{n+1-}=U_{j-1}^{n},\quad  V_{j}^{n+1-}=V_{j+1}^{n}.
 	\end{equation}
 		\item [\textup{(ii)}] For nonlinear subproblem \eqref{3.1} and \eqref{3.2},
 	\begin{equation}\label{3.9}
 		\mathcal{L}_{j}^{n,2}(s)+\mathcal{D}_{j}^{n,2}(s)\leq\mathrm{e}^{C_2}\big(\mathcal{L}_{j}^{n,2}(0)+\mathcal{D}_{j}^{n,2}(0)\big),
 	\end{equation}
 	and
 		\begin{align}
 		|U_{j}^{n,2}(\tau)|^{2}\leq|U_{j}^{n,2}(0)|^{2}+C_1\big(\mathcal{L}_{j}^{n,2}(0)+\mathcal{D}_{j}^{n,2}(0)\big)\tau,\label{3.7}\\
 		|V_{j}^{n,2}(\tau)|^{2}\leq|V_{j}^{n,2}(0)|^{2}+C_1\big(\mathcal{L}_{j}^{n,2}(0)+\mathcal{D}_{j}^{n,2}(0)\big)\tau.\label{3.8}
 	\end{align}
 \end{itemize}
\end{lemma}

\begin{proof}
First, by \eqref{eq:characteristic-estimate-uv}, we have
\begin{equation*}
	u_{j+1}^{n+1-}=u_{j}^{n},\ \ v_{j-1}^{n+1-}=v_{j}^{n}
	\quad\text{ and }\quad
	\tilde{u}_{j+1}^{n+1-}=\tilde{u}_{j}^{n},\ \ \tilde{v}_{j-1}^{n+1-}=\tilde{v}_{j}^{n}.
\end{equation*}
Thus, from the definition $(U,V)=(u-\tilde{u},v-\tilde{v})$, we obtain \eqref{eq:UV-sub1}.

To prove (ii), recall that $\mathrm{Re}\, z$ denotes the real part of $z \in \mathbb{C}$.
Multiplying the equation \eqref{3.1}, evaluated at $x=j\tau$, by $\overline{U_{j}^{n,2}(s)}$, we get
\begin{align*}
		\frac{\mathrm{d}|U_{j}^{n,2}(s)|^{2}}{\mathrm{d}s}
        =&2\mathrm{Re}\left\{imV_{j}^{n,2}(s)\overline{U_{j}^{n,2}(s)}\right\}\\
        &+2\mathrm{Re}\left\{i(\alpha+2\beta)
\big(u_{j}^{n,2}(s)\big|v_{j}^{n,2}(s)\big|^{2}-\tilde{u}_{j}^{n,2}(s)
\big|\tilde{v}_{j}^{n,2}(s)\big|^{2}\big)\overline{U_{j}^{n,2}(s)}\right\}\\
		&+2\mathrm{Re}\left\{i2\beta\big(\overline{u_{j}^{n,2}(s)}\big(v_{j}^{n,2}(s)\big)^{2}-\overline{\tilde{u}_{j}^{n,2}(s)}
\big(\tilde{v}_{j}^{n,2}(s)\big)^{2}\big)\overline{U_{j}^{n,2}(s)}\right\}.
	\end{align*}
Then, using Young's inequality, it follows that
\begin{equation*}
   2\mathrm{Re}\left\{imV_{j}^{n,2}(s)\overline{U_{j}^{n,2}(s)}\right\}
   \leq m\big(|U_{j}^{n,2}(s)|^2+|V_{j}^{n,2}(s)|\big)=m\mathcal{L}_{j}^{n,2}(s),
\end{equation*}
and
\begin{align*}
    &2\mathrm{Re}\left\{i(\alpha+2\beta)
\big(u_{j}^{n,2}(s)\big|v_{j}^{n,2}(s)\big|^{2}-\tilde{u}_{j}^{n,2}(s)
\big|\tilde{v}_{j}^{n,2}(s)\big|^{2}\big)\overline{U_{j}^{n,2}(s)}\right\}\\
= &
2\mathrm{Re}\left\{i(\alpha+2\beta)
\Big(
U_{j}^{n,2}(s)\big|v_{j}^{n,2}(s)\big|^{2}
+
\tilde{u}_{j}^{n,2}(s)V_{j}^{n,2}(s)\overline{v_{j}^{n,2}(s)}
+
\tilde{u}_{j}^{n,2}(s)\tilde{v}_{j}^{n,2}(s)\overline{V_{j}^{n,2}(s)}
\Big)
\overline{U_{j}^{n,2}(s)}\right\}\\
\leq&
(|\alpha|+2|\beta|)
\Big(
2\big|U_{j}^{n,2}(s)\big|^2\big|v_{j}^{n,2}(s)\big|^{2}
+
2\big|\tilde{u}_{j}^{n,2}(s)\big|^2\big|V_{j}^{n,2}(s)\big|^2
+
\big|\tilde{v}_{j}^{n,2}(s)\big|^2\big|U_{j}^{n,2}(s)\big|^2
\Big)\\
\leq &
4(|\alpha|+2|\beta|)\mathcal{D}_{j}^{n,2}(s),
\end{align*}
and
\begin{align*}
&2\mathrm{Re}\left\{i2\beta\big(\overline{u_{j}^{n,2}(s)}\big(v_{j}^{n,2}(s)\big)^{2}-\overline{\tilde{u}_{j}^{n,2}(s)}
\big(\tilde{v}_{j}^{n,2}(s)\big)^{2}\big)\overline{U_{j}^{n,2}(s)}\right\}\\
=&
2\mathrm{Re}\left\{i2\beta
\Big(
\overline{U_{j}^{n,2}(s)}\big(v_{j}^{n,2}(s)\big)^{2}
+
\overline{\tilde{u}_{j}^{n,2}(s)}V_{j}^{n,2}(s)v_{j}^{n,2}(s)
+
\overline{\tilde{u}_{j}^{n,2}(s)}\tilde{v}_{j}^{n,2}(s)V_{j}^{n,2}(s)
\Big)
\overline{U_{j}^{n,2}(s)}\right\}\\
\leq&
2|\beta|
\Big(
2\big|U_{j}^{n,2}(s)\big|^2\big|v_{j}^{n,2}(s)\big|^{2}
+
2\big|\tilde{u}_{j}^{n,2}(s)\big|^2\big|V_{j}^{n,2}(s)\big|^2
+\big|\tilde{v}_{j}^{n,2}(s)\big|^2\big|U_{j}^{n,2}(s)\big|^2
\Big)\\
\leq& 8|\beta|\mathcal{D}_{j}^{n,2}(s).
\end{align*}
Thus, by setting $C_*=4(|\alpha|+4|\beta|)$, we have
\begin{equation}\label{3.3}
	\frac{\mathrm{d}|U_{j}^{n,2}(s)|^{2}}{\mathrm{d}s}
    \leq m\mathcal{L}_{j}^{n,2}(s)+C_*\mathcal{D}_{j}^{n,2}(s).
\end{equation}

Next, multiplying the equation \eqref{3.2}, evaluated at $x=j\tau$, by $\overline{V_{j}^{n,2}(s)}$ gives
\begin{align*}
		\frac{\mathrm{d}|V_{j}^{n,2}(s)|^{2}}{\mathrm{d}s}
		=&2\mathrm{Re}\left\{imU_{j}^{n,2}(s)\overline{V_{j}^{n,2}(s)}\right\}\\
		&+2\mathrm{Re}\left\{i(\alpha+2\beta)\big(v_{j}^{n,2}(s) \big|u_{j}^{n,2}(s)\big|^{2}-\tilde{v}_{j}^{n,2}(s)\big|\tilde{u}_{j}^{n,2}(s)\big|^{2}\big)\overline{V_{j}^{n,2}(s)}\right\}\\ &+2\mathrm{Re}\left\{i2\beta\big(\overline{v_{j}^{n,2}(s)}\big(u_{j}^{n,2}(s)\big)^{2}-\overline{\tilde{v}_{j}^{n,2}(s)}
	\big(\tilde{u}_{j}^{n,2}(s)\big)^{2}\big)\overline{V_{j}^{n,2}(s)}\right\}.
\end{align*}
By similar argument as in the proof of \eqref{3.3}, we get
\begin{equation}\label{3.4}
    \frac{\mathrm{d}|V_{j}^{n,2}(s)|^{2}}{\mathrm{d}s}\leq m\mathcal{L}_{j}^{n,2}(s)+C_*\mathcal{D}_{j}^{n,2}(s).
\end{equation}

Furthermore, since $\mathrm{Re}\left\{imV_{j}^{n,2}(s)\overline{U_{j}^{n,2}(s)}\right\}+\mathrm{Re}\left\{imU_{j}^{n,2}(s)\overline{V_{j}^{n,2}(s)}\right\}=0$, we deduce
\begin{align}\label{3.5}
	\frac{\mathrm{d}\mathcal{L}_{j}^{n,2}(s)}{\mathrm{d}s}
    =\frac{\mathrm{d}\big(|\mathcal{U}_{j}^{n,2}(s)|^2+|\mathcal{V}_{j}^{n,2}(s)|^2\big)}{\mathrm{d}s}
    \leq 2C_*\mathcal{D}_{j}^{n,2}(s).
\end{align}
For the quartic term $\mathcal{D}_{j}^{n,2}$, we use \eqref{2.2}--\eqref {2.1} and \eqref{3.3}--\eqref{3.4} to deduce that
\begin{align*}
		\frac{\mathrm{d}\mathcal{D}_{j}^{n,2}(s)}
        {\mathrm{d}s}
        =&
      \frac{\mathrm{d}|U_{j}^{n,2}(s)|^{2}}{\mathrm{d}s}
\big(|v_{j}^{n,2}(s)|^{2}+|\tilde{v}_{j}^{n,2}(s)|^{2}\big)
+
|U_{j}^{n,2}(s)|^{2}
\Big(\frac{\mathrm{d}|v_{j}^{n,2}(s)|^{2}}{\mathrm{d}s}
+\frac{\mathrm{d}|\tilde{v}_{j}^{n,2}(s)|^{2}}{\mathrm{d}s}\Big)\\
&+
\frac{\mathrm{d}|V_{j}^{n,2}(s)|^{2}}{\mathrm{d}s}
\big(|u_{j}^{n,2}(s)|^{2}+|\tilde{u}_{j}^{n,2}(s)|^{2}\big)
+
|V_{j}^{n,2}(s)|^{2}
\Big(\frac{\mathrm{d}|u_{j}^{n,2}(s)|^{2}}{\mathrm{d}s}
+
\frac{\mathrm{d}|\tilde{u}_{j}^{n,2}(s)|^{2}}{\mathrm{d}s}\Big)\\
\leq&
m\mathcal{L}_{j}^{n,2}(s)\big(\mathfrak{s}_{j}^{n,2}(s)+\tilde{\mathfrak{s}}_{j}^{n,2}(s)\big)
+C_*\mathcal{D}_{j}^{n,2}(s)\big(\mathfrak{s}_{j}^{n,2}(s)+\tilde{\mathfrak{s}}_{j}^{n,2}(s)\big)
\\
&+
m\big(\mathfrak{s}_{j}^{n,2}(s)+\tilde{\mathfrak{s}}_{j}^{n,2}(s)\big)\big(|U_{j}^{n,2}(s)|^{2}+|V_{j}^{n,2}(s)|^{2}\big)
\\
&+4|\beta|\big(|u_{j}^{n,2}(s)|^{2}|v_{j}^{n,2}(s)|^{2}+|\tilde{u}_{j}^{n,2}(s)|^{2}|\tilde{v}_{j}^{n,2}(s)|^{2}\big)
\big(|U_{j}^{n,2}(s)|^{2}+|V_{j}^{n,2}(s)|^{2}\big)\\
\leq&
2m\big(\mathfrak{s}_{j}^{n,2}(0)+\tilde{\mathfrak{s}}_{j}^{n,2}(0)\big)\mathcal{L}_{j}^{n,2}(s)
+(C_*+4|\beta|)\big(\mathfrak{s}_{j}^{n,2}(0)+\tilde{\mathfrak{s}}_{j}^{n,2}(0)\big)\mathcal{D}_{j}^{n,2}(s),
\end{align*}
which, together with \eqref{3.5}, yields
	\begin{equation*}	\frac{\mathrm{d}\big(\mathcal{L}_{j}^{n,2}(s)+\mathcal{D}_{j}^{n,2}(s)\big)}{\mathrm{d}s}
		\leq (2m+3C_*+4|\beta|)\big(\mathfrak{s}_{j}^{n,2}(0)+\tilde{\mathfrak{s}}_{j}^{n,2}(0)\big)\big(\mathcal{L}_{j}^{n,2}(s)+\mathcal{D}_{j}^{n,2}(s)\big).
	\end{equation*}
Thus, integrating the above inequality over $[0,\tau]$ and using Gronwall's inequality, we have
\begin{align*}
 		\mathcal{L}_{j}^{n,2}(s)+\mathcal{D}_{j}^{n,2}(s)
        \leq&
        \mathrm{e}^{(2m+3C_*+4|\beta|)\big(\mathfrak{s}_{j}^{n,2}(0)+\tilde{\mathfrak{s}}_{j}^{n,2}(0)\big)\tau}\big(\mathcal{L}_{j}^{n,2}(0)+\mathcal{D}_{j}^{n,2}(0)\big)\\
        \leq&
\mathrm{e}^{2(2m+3C_*+4|\beta|)c_0}\big(\mathcal{L}_{j}^{n,2}(0)+\mathcal{D}_{j}^{n,2}(0)\big).
 	\end{align*}
Here we used that $\big(\mathfrak{s}_{j}^{n,2}(0)+\tilde{\mathfrak{s}}_{j}^{n,2}(0)\big)\tau \leq 2c_0$, which follows from \eqref{3.16.5}. Therefore, \eqref{3.9} follows by setting $C_2=2(2m+3C_*+4|\beta|)c_0$.
	
It remains to verify \eqref{3.7} and \eqref{3.8}.
Hence, substituting \eqref{3.9} into \eqref{3.3} and \eqref{3.4}, we obtain
\begin{equation*}\frac{\mathrm{d}|U_{j}^{n,2}(s)|^{2}}{\mathrm{d}s}
   \leq (m+C_*) \big(\mathcal{L}_{j}^{n,2}(s)+\mathcal{D}_{j}^{n,2}(s)\big)
    \leq (m+C_*)\mathrm{e}^{C_2}\big(\mathcal{L}_{j}^{n,2}(0)+\mathcal{D}_{j}^{n,2}(0)\big),\end{equation*}
	and
	\begin{equation*}\frac{\mathrm{d}|V_{j}^{n,2}(s)|^{2}}{\mathrm{d}s}
     \leq (m+C_*) \big(\mathcal{L}_{j}^{n,2}(s)+\mathcal{D}_{j}^{n,2}(s)\big)
    \leq (m+C_*)\mathrm{e}^{C_2}\big(\mathcal{L}_{j}^{n,2}(0)+\mathcal{D}_{j}^{n,2}(0)\big).\end{equation*}
Integrating the above inequalities over $[0,\tau]$ and taking $C_1=(m+C_*)\mathrm{e}^{C_2}$  gives \eqref{3.7} and \eqref{3.8}.	
This completes the proof of Lemma \ref{lem:pointwise-estimates-UV}.
\end{proof}

Next, we turn to study the $\mathrm{L}^2$ stability for $(u,v)$ and $(\tilde{u},\tilde{v})$ in a discrete characteristic triangle $\Delta(j_{1},n_{1};n_{0})$. Again, for simplicity of notation, we let $\Delta$ stand for $\Delta(j_{1},n_{1};n_0)$.
For this purpose, we use the estimates obtained in Subsections \ref{sub:characteristic-Estimates}--\ref{sub:pointwise-estimate} and Lemma \ref{lem:pointwise-estimates-UV} to carefully design a Glimm-type functional $\mathcal{F}$ on $\Delta$ and to establish its uniform boundedness with respect to the time step, thereby deriving uniform stability estimates for the time-splitting solutions defined by \eqref{eq:subproblem1}--\eqref{eq:subproblem2-initial-data}.

Now, the modified nonlinear Glimm-type functional $\mathcal{F}$ on a discrete characteristic triangle $\Delta$ can be defined as follows.
\begin{definition}[Glimm-type functional]\label{def:Glimm-type-functional}
For any discrete characteristic triangle $\Delta(j_{1},n_{1};n_0)$ with $0\leq n_0 \le n \le n_1$ and for any $\tau>0$, the nonlinear Glimm-type functional $\mathcal{F}$ is defined by
\begin{equation*}
\mathcal{F}(n;\Delta)=\mathcal{L}(n;\Delta)\,\tau+\kappa\mathcal{Q}(n;\Delta)\,\tau^{2},
\end{equation*}
where
\begin{align*}
\mathcal{L}(n;\Delta)=&\sum_{j=j_{1}-n_{1}+n}^{j_{1}+n_{1}-n}
\big(|U_{j}^{n}|^{2}+|V_{j}^{n}|^{2}\big),\\
\mathcal{Q}(n;\Delta)=&\sum_{j_{1}-n_{1}+n\leq j<k\leq j_{1}+n_{1}-n }\Big({|U_{j}^{n}|^{2}\big(|v_{k}^{n}|^{2}+|\tilde{v}_{k}^{n}|^{2}\big)+|V_{k}^{n}|^{2}\big(|u_{j}^{n}|^{2}+|\tilde{u}_{j}^{n}|^{2}\big)}\Big),
\end{align*}
and $\kappa$ is a large positive constant that will be determined later.
\end{definition}
Here $\mathcal{Q}(n;\Delta)$ is a Bony-type functional, see also \cite{Bony1987} and \cite{Ha2003CMP}.
To establish the uniform boundedness of the Glimm-type functional $\mathcal{F}(n;\Delta)$, we introduce the following definition.
\begin{definition}\label{def:functional-UV}
Let $(j,n)$ be a pair of integers satisfying $ 0\leq n_0\leq n\leq n_{1}$.
For time-splitting solutions $(u,v)$ and $(\tilde{u}, \tilde{v})$, we define the quantities
\begin{equation*}\mathfrak{s}(n;\Delta)
=\sum_{j=j_{1}-n_{1}+n}^{j_{1}+n_{1}-n}
\big(|u_{j}^{n}|^{2}+|v_{j}^{n}|^{2}\big),
\qquad
\tilde{\mathfrak{s}}(n;\Delta)
=\sum_{j=j_{1}-n_{1}+n}^{j_{1}+n_{1}-n}
\big(|\tilde{u}_{j}^{n}|^{2}+|\tilde{v}_{j}^{n}|^{2}\big),\end{equation*}
and
\begin{equation*}
\mathfrak{q}(n;\Delta)=
\sum_{j=j_{1}-n_{1}+n-1}^{j_{1}+n_{1}-n+1}
|u_{j-1}^{n}|^{2}|v_{j+1}^{n}|^{2},
\qquad
\tilde{\mathfrak{q}}(n;\Delta)
=\sum_{j=j_{1}-n_{1}+n+1}^{j_{1}+n_{1}-n-1}
|\tilde{u}_{j-1}^{n}|^{2}|\tilde{v}_{j+1}^{n}|^{2}.
\end{equation*}
	
For the difference $(U,V)$, we define
\begin{equation*}
\mathcal{D}(n;\Delta)
=\sum_{j=j_{1}-n_{1}+n}^{j_{1}+n_{1}-n} \Big({|U_{j}^{n}|^{2}\big(|v_{j}^{n}|^{2}+|\tilde{v}_{j}^{n}|^{2}\big)+|V_{j}^{n}|^{2}\big(|u_{j}^{n}|^{2}+|\tilde{u}_{j}^{n}|^{2}\big)}\Big).
\end{equation*}
\end{definition}

\begin{remark}\label{remark:L-L_i^j,2}
From Definitions \ref{def:functional-uv}--\ref{def:Glimm-type-functional}, and from the construction of time-splitting solutions, we have
\begin{small}
\begin{equation*}
\mathfrak{s}(n;\Delta)=
\sum_{j=j_{1}-n_{1}+n}^{j_{1}+n_{1}-n} \mathfrak{s}_{j}^{n-1,2}(\tau),\quad
\mathcal{L}(n;\Delta)=
\sum_{j=j_{1}-n_{1}+n}^{j_{1}+n_{1}-n} \mathcal{L}_{j}^{n-1,2}(\tau),\quad
\mathcal{D}(n;\Delta)=
\sum_{j=j_{1}-n_{1}+n}^{j_{1}+n_{1}-n} \mathcal{D}_{j}^{n-1,2}(\tau),
\end{equation*}
\end{small}
and
\begin{small}
\begin{equation*}
\mathfrak{s}(n-;\Delta)=
\sum_{j=j_{1}-n_{1}+n}^{j_{1}+n_{1}-n} \mathfrak{s}_{j}^{n-1,2}(0),\
\mathcal{L}(n-;\Delta)=
\sum_{j=j_{1}-n_{1}+n}^{j_{1}+n_{1}-n}\mathcal{L}_{j}^{n-1,2}(0),\
\mathcal{D}(n-;\Delta)=
\sum_{j=j_{1}-n_{1}+n}^{j_{1}+n_{1}-n} \mathcal{D}_{j}^{n-1,2}(0).
\end{equation*}
\end{small}
\end{remark}
Using the above estimates and definitions, we show that the Glimm-type functional $\mathcal{F}$ is uniformly bounded in the following proposition, when the constant $\kappa$ is properly chosen.
\begin{proposition}\label{pro:functional-decreasing}
Let $T>0$ and let $\bar{\tau}\in(0,1)$ be the constant given by \eqref{eq:c_0}. Then, there exist positive constants $\delta$, $\kappa$, $C_3$ and $C_4$, depending only on $c_0$ and the system \eqref{eq:NLDE}, such that if $\mathfrak{s}(n_0;\Delta)\tau\leq\delta$ and if $\tilde{\mathfrak{s}}(n_0;\Delta)\tau\leq\delta$, then for all $\tau\in(0,\bar{\tau})$ and for all integers $n$ satisfying $\:0\leq n_{0}\leq n\leq n_{1}$, it holds that
\begin{equation}\label{eq:Glimm-functional-decrease}
\begin{aligned}
&\mathcal{F}(n+1;\Delta)-\mathcal{F}(n;\Delta)\\
\leq&
-\mathcal{D}(n+1-;\Delta)\tau^{2}
+ \left[ (C_{3}+\kappa C_{4})\tau
+\kappa C_{4}\big(\mathfrak{q}(n;\Delta)+\tilde{\mathfrak{q}}(n;\Delta)\big)\tau^2\right]  \mathcal{F}(n;\Delta).
\end{aligned}	
\end{equation}
Furthermore, let $\mathcal{K}(n,n_0)=(C_{3}+\kappa C_{4})(n-n_0)\tau
	+\kappa C_{4}\sum_{p=n_0}^{n-1}\big(\mathfrak{q}(p;\Delta)+\tilde{\mathfrak{q}}(p;\Delta)\big)\tau^2$. Then there exists a constant $C_5(T)>0$, depending only on $T,c_0$ and the system \eqref{eq:NLDE}, such that for all integers $n$ satisfying $\:0\leq n_{0}\leq n\leq n_{1}$, $n\leq (T+1)/\tau$, with the range of $n$ illustrated in Fig.~\ref{fig-n}, it holds that
	\begin{equation}\label{3.17}
		\mathcal{F}(n;\Delta)
       \leq
       \mathrm{e}^{\mathcal{K}(n,n_0)}\mathcal{F}(n_{0};\Delta)
	\end{equation}
and
\begin{equation}\label{3.17-1}
		\mathcal{K}(n,n_0)\leq C_5(T).
	\end{equation}
\end{proposition}
\begin{proof}
	We split the proof into three steps.

\Step \label{step:1}
For $\mathbf{Step(n,1)}$, our aim is to show the following claims:
\begin{equation}\label{eq:claim1}
\mathcal{L}(n+1-;\Delta) \tau-\mathcal{L}(n;\Delta)\tau\leq 0
\end{equation}
and
\begin{equation}\label{eq:claim1-2}
\mathcal{Q}(n+1-;\Delta)\tau^2-\mathcal{Q}(n;\Delta)\tau^2
\leq -\mathcal{D}(n+1-;\Delta)\tau^2.
\end{equation}
First, we can combine the estimate \eqref{eq:UV-sub1} with the definition of $\mathcal{L}(n;\Delta)$
to obtain \eqref{eq:claim1}.

Secondly, to prove the claim \eqref{eq:claim1-2}, we deduce from the Definition \ref{def:Glimm-type-functional}, \eqref{eq:characteristic-estimate-uv} and \eqref{eq:UV-sub1} that
\begin{align*}
		&\mathcal{Q}(n+1-;\Delta)\\
		=&
		\sum_{j_{1}-n_{1}+n+1\leq j<k\leq j_{1}+n_{1}-n-1}
\Big({|U_{j}^{n+1-}|^{2}\big(|v_{k}^{n+1-}|^{2}+|\tilde{v}_{k}^{n+1-}|^{2}\big)
+|V_{k}^{n+1-}|^{2}\big(|u_{j}^{n+1-}|^{2}+|\tilde{u}_{j}^{n+1-}|^{2}\big)}\Big)\\
=&
\sum_{j_{1}-n_{1}+n+1\leq j<k\leq j_{1}+n_{1}-n-1}
\Big({|U_{j-1}^{n}|^{2}\big(|v_{k+1}^{n}|^{2}+|\tilde{v}_{k+1}^{n}|^{2}\big)
+|V_{k+1}^{n}|^{2}\big(|u_{j-1}^{n}|^{2}+|\tilde{u}_{j-1}^{n}|^{2}\big)}\Big).
\end{align*}
Thus, a direct calculation shows that
\begin{align}	
		&\mathcal{Q}(n+1-;\Delta)\nonumber\\
\leq &
		\sum_{j_{1}-n_{1}+n\leq j<k\leq j_{1}+n_{1}-n  }
		\Big({|U_{j}^{n}|^{2}\big(|v_{k}^{n}|^{2}+|\tilde{v}_{k}^{n}|^{2}\big)+|V_{k}^{n}|^{2}\big(|u_{j}^{n}|^{2}+|\tilde{u}_{j}^{n}|^{2}\big)}\Big)\nonumber\\
		&-\sum_{j=j_{1}-n_{1}+n}^{j_{1}+n_{1}-n-2}
		\Big({|U_{j}^{n}|^{2}(|v_{j+2}^{n}|^{2}+|\tilde{v}_{j+2}^{n}|^{2})+|V_{j+2}^{n}|^{2}(|u_{j}^{n}|^{2}+|\tilde{u}_{j}^{n}|^{2})}\Big)\nonumber\\
		=&
		\mathcal{Q}(n;\Delta)
		-\sum_{j=j_{1}-n_{1}+n}^{j_{1}+n_{1}-n-2  }
		\Big({|U_{j}^{n}|^{2}(|v_{j+2}^{n}|^{2}+|\tilde{v}_{j+2}^{n}|^{2})+|V_{j+2}^{n}|^{2}(|u_{j}^{n}|^{2}+|\tilde{u}_{j}^{n}|^{2})}\Big).\label{eq:Q-step1}
	\end{align}
Then recalling the Definition \ref{def:functional-UV} of $\mathcal{D}(n;\Delta)$ and again using \eqref{eq:characteristic-estimate-uv} and \eqref{eq:UV-sub1}, we obtain
\begin{align*}
\mathcal{D}(n+1-;\Delta)
=&\sum_{j=j_{1}-n_{1}+n+1}^{j_{1}+n_{1}-n-1} \Big({|U_{j}^{n+1-}|^{2}\big(|v_{j}^{n+1-}|^{2}+|\tilde{v}_{j}^{n+1-}|^{2}\big)+|V_{j}^{n+1-}|^{2}\big(|u_{j}^{n+1-}|^{2}+|\tilde{u}_{j}^{n+1-}|^{2}\big)}\Big)\nonumber\\
=&
\sum_{j=j_{1}-n_{1}+n+1}^{j_{1}+n_{1}-n-1} \Big({|U_{j-1}^{n}|^{2}\big(|v_{j+1}^{n}|^{2}+|\tilde{v}_{j+1}^{n}|^{2}\big)
+|V_{j+1}^{n}|^{2}\big(|u_{j-1}^{n}|^{2}+|\tilde{u}_{j-1}^{n}|^{2}\big)}\Big)
\nonumber\\
= &
\sum_{j=j_{1}-n_{1}+n}^{j_{1}+n_{1}-n-2 }
\Big({|U_{j}^{n}|^{2}\big(|v_{j+2}^{n}|^{2}+|\tilde{v}_{j+2}^{n}|^{2}\big)+|V_{j+2}^{n}|^{2}\big(|u_{j}^{n}|^{2}+|\tilde{u}_{j}^{n}|^{2}\big)}\Big),
\end{align*}
which together with \eqref{eq:Q-step1} implies \eqref{eq:claim1-2}.

\Step \label{step:2}
For $\mathbf{Step(n,2)}$,  our aim is to show the following claim:
There exist positive constants $\delta,\kappa,C_3,C_4$ and $\tilde{C}_4$, depending only on $c_0$ and the system \eqref{eq:NLDE}, such that if $\mathfrak{s}(n_0;\Delta)\tau\leq\delta$ and if $\tilde{\mathfrak{s}}(n_0;\Delta)\tau\leq\delta$, then for all $\tau\in(0,\bar{\tau})$ and for all $\:0\leq n_{0}\leq n\leq n_{1}$, it holds that
\begin{equation}\label{eq:claim2}
\mathcal{L}(n+1;\Delta)\tau-\mathcal{L}(n+1-;\Delta)\tau
\leq
C_{3}\tau  \mathcal{L}(n+1-;\Delta)\tau+C_{3} \mathcal{D}(n+1-;\Delta)\tau^2
\end{equation}
and
\begin{align}\label{eq:claim2-2}
&\mathcal{Q}(n+1;\Delta)\tau^2-\mathcal{Q}(n+1-;\Delta)\tau^2\nonumber\\
\leq&
C_{4}\left[\tau+\big(\mathfrak{q}(n;\Delta)+\tilde{\mathfrak{q}}(n;\Delta)\big)\tau^2\right]\mathcal{L}(n+1-;\Delta)\tau+\tilde{C}_4\delta \mathcal{D}(n+1-;\Delta)\tau^2.
\end{align}

To prove the estimate \eqref{eq:claim2}, by \eqref{3.5} and Remark \ref{remark:L-L_i^j,2}, we have
\begin{equation}\label{eq:L-TRI}
    \frac{\mathrm{d}\big(\sum_{j=j_{1}-n_{1}+n+1}^{j_{1}+n_{1}-n-1} \mathcal{L}_{j}^{n,2}(s)\big)}{\mathrm{d}s}
\leq 2C_*\Big(\sum_{j=j_{1}-n_{1}+n+1}^{j_{1}+n_{1}-n-1} \mathcal{D}_{j}^{n,2}(s)\Big).
\end{equation}
In view of \eqref{3.9}, we get
\begin{equation}\label{corollary3.1}
\sum_{j=j_{1}-n_{1}+n+1}^{j_{1}+n_{1}-n-1}\mathcal{D}_{j}^{n,2}(s)
\leq\mathrm{e}^{C_2}\sum_{j=j_{1}-n_{1}+n+1}^{j_{1}+n_{1}-n-1}
\Big(\mathcal{L}_{j}^{n,2}(0)+\mathcal{D}_{j}^{n,2}(0)\Big).
\end{equation}

Plugging \eqref{corollary3.1} into \eqref{eq:L-TRI}, and integrating the resulting equation over $[0,\tau]$, we deduce that
\begin{equation*}
\sum_{j=j_{1}-n_{1}+n+1}^{j_{1}+n_{1}-n-1}\mathcal{L}_{j}^{n,2}(\tau)-
\sum_{j=j_{1}-n_{1}+n+1}^{j_{1}+n_{1}-n-1}\mathcal{L}_{j}^{n,2}(0)
\leq2C_*\mathrm{e}^{C_2}	\sum_{j=j_{1}-n_{1}+n+1}^{j_{1}+n_{1}-n-1}
\Big(\mathcal{L}_{j}^{n,2}(0)+\mathcal{D}_{j}^{n,2}(0)\Big)\tau,
\end{equation*}
which gives the estimate \eqref{eq:claim2} by taking $C_{3}=2C_*\mathrm{e}^{C_2}$ and using notations in Remark \ref{remark:L-L_i^j,2}.

Next, we want to prove the estimate \eqref{eq:claim2-2}. To do this, from the definition of $\mathcal{Q}(n;\Delta)$ in Definition~\ref{def:Glimm-type-functional}, we obtain
\begin{align*}
&\mathcal{Q}(n+1;\Delta)-\mathcal{Q}(n+1-;\Delta)\\
=&\sum_{j_{1}-n_{1}+n+1\leq j<k\leq j_{1}+n_{1}-n-1}
\Big({|U_{j}^{n+1}|^{2}\big(|v_{k}^{n+1}|^{2}+|\tilde{v}_{k}^{n+1}|^{2}\big)+|V_{k}^{n+1}|^{2}\big(|u_{j}^{n+1}|^{2}+|\tilde{u}_{j}^{n+1}|^{2}\big)}\Big)\\
-&
\sum_{j_{1}-n_{1}+n+1\leq j<k\leq j_{1}+n_{1}-n-1}
\Big({|U_{j}^{n+1-}|^{2}\big(|v_{k}^{n+1-}|^{2}+|\tilde{v}_{k}^{n+1-}|^{2}\big)+|V_{k}^{n+1-}|^{2}\big(|u_{j}^{n+1-}|^{2}+|\tilde{u}_{j}^{n+1-}|^{2}\big)}\Big).
\end{align*}
By \eqref{2.10} and  \eqref{2.11}, we have
\begin{align*}
|u_{j}^{n+1}|^{2}+|\tilde{u}_{j}^{n+1}|^{2}
\leq&
|u_{j}^{n+1-}|^{2}+|\tilde{u}_{j}^{n+1-}|^{2}
+m_1\left( \mathfrak{s}_{j}^{n,2}(0)+\tilde{\mathfrak{s}}_{j}^{n,2}(0)\right) \tau\\
&+c_{1}\left( |u_{j}^{n+1-}|^{2}|v_{j}^{n+1-}|^{2}+|\tilde{u}_{j}^{n+1-}|^{2}|\tilde{v}_{j}^{n+1-}|^{2}\right) \tau,\\
|v_{k}^{n+1}|^{2}+|\tilde{v}_{k}^{n+1}|^{2}
\leq
&|v_{k}^{n+1-}|^{2}+|\tilde{v}_{k}^{n+1-}|^{2}
+m_1\left( \mathfrak{s}_{j}^{n,2}(0)+\tilde{\mathfrak{s}}_{j}^{n,2}(0)\right) \tau\\	
&+c_{1}\left( |u_{k}^{n+1-}|^{2}|v_{k}^{n+1-}|^{2}+|\tilde{u}_{k}^{n+1-}|^{2}|\tilde{v}_{k}^{n+1-}|^{2}\right) \tau.
\end{align*}
Furthermore, it follows from \eqref{3.7} and \eqref{3.8} that
\begin{equation*}
|U_{j}^{n+1}|^{2}
\leq
|U_{j}^{n+1-}|^{2}+C_1\big(\mathcal{L}_{j}^{n,2}(0)+\mathcal{D}_{j}^{n,2}(0)\big)\tau
\end{equation*}
and
\begin{equation*}
|V_{k}^{n+1}|^{2}
\leq
|V_{k}^{n+1-}|^{2}+C_1\big(L_{k}^{n,2}(0)+D_{k}^{n,2}(0)\big)\tau.
\end{equation*}
Since $\mathfrak{s}(n_0;\Delta)\tau\leq\delta$ and $\tilde{\mathfrak{s}}(n_0;\Delta)\tau\leq\delta$,
recalling Remark \ref{remark:L-L_i^j,2} and using \eqref{eq:characteristic-estimate-uv} and \eqref{2.4}, we have
\begin{align*}
&\sum_{j=j_{1}-n_{1}+n+1}^{j_{1}+n_{1}-n-1}
\left( \mathfrak{s}_{j}^{n,2}(0)+\tilde{\mathfrak{s}}_{j}^{n,2}(0)\right) \tau
=\left(\mathfrak{s}(n+1-;\Delta)+\tilde{\mathfrak{s}}(n+1-;\Delta)\right)\tau\nonumber\\
\leq &
\left(\mathfrak{s}(n;\Delta)+\tilde{\mathfrak{s}}(n;\Delta)\right)\tau
\leq
\left(\mathfrak{s}(n_{0};\Delta)+\tilde{\mathfrak{s}}(n_{0};\Delta)\right)\tau
\leq 2\delta,
\end{align*}
which yields
\begin{equation*}
    \sum_{j=j_{1}-n_{1}+n+1}^{j_{1}+n_{1}-n-1}\left( |u_{j}^{n+1-}|^{2}|v_{j}^{n+1-}|^{2}+|\tilde{u}_{j}^{n+1-}|^{2}|\tilde{v}_{j}^{n+1-}|^{2}\right) \tau^2
	\leq
    \big(\mathfrak{s}(n;\Delta)\tau\big)^{2}
    +\big(\tilde{\mathfrak{s}}(n;\Delta)\tau\big)^2
\leq2\delta^{2}.
\end{equation*}

Using the above estimates and recalling the notations introduced in Remark~\ref{remark:L-L_i^j,2}, we obtain
\begin{align*}
&\mathcal{Q}(n+1;\Delta)-\mathcal{Q}(n+1-;\Delta)\\
\leq&
2\delta\big(m_{1}+2C_1m_{1}+C_1+2C_1c_1\delta\big) \mathcal{L}(n+1-;\Delta)
+c_{1}\big(\mathfrak{q}(n;\Delta)\tau+\tilde{\mathfrak{q}}(n;\Delta)\tau\big)\mathcal{L}(n+1-;\Delta)\\
&+2\delta C_1\big(2m_{1}+1+2c_{1}\delta\big)\mathcal{D}(n+1-;\Delta),
\end{align*}
which gives \eqref{eq:claim2-2} by taking
\begin{equation*}
C_{4}=2\delta\big(m_{1}+2C_1m_{1}+C_1+2C_1c_1\delta\big)+c_{1},\quad \tilde{C}_4=2C_1(2m_{1}+1+2c_{1}\delta).
\end{equation*}
			
\Step \label{step:3}
Using the estimates \eqref{eq:claim1}--\eqref{eq:claim1-2} and \eqref{eq:claim2}--\eqref{eq:claim2-2}, obtained in Steps \ref{step:1}--\ref{step:2}, we have
\begin{align}
\mathcal{F}(n+1;\Delta)-\mathcal{F}(n;\Delta)
=&
\mathcal{L}(n+1;\Delta)\tau+\kappa\mathcal{Q}(n+1;\Delta)\tau^{2}
-\mathcal{L}(n;\Delta)\tau-\kappa\mathcal{Q}(n;\Delta)\tau^{2}\nonumber\\
\leq &
\left[(C_{3}+\kappa C_{4})\tau+\kappa C_{4}\big(\mathfrak{q}(n;\Delta)+\tilde{\mathfrak{q}}(n;\Delta)\big)\tau^{2}\right] \mathcal{L}(n;\Delta)\tau\nonumber\\
&+[C_{3}+\kappa(-1+\tilde{C}_4\delta)] \mathcal{D}(n+1-;\Delta)\tau^{2}.\label{eq:step3-F}
\end{align}
We first choose $\delta>0$ sufficiently small such that $\delta<\frac{1}{2\tilde{C}_4}$, and then choose $\kappa>0$ sufficiently large so that $\kappa>2(1+C_3)$.
Consequently,
\begin{equation}\label{eq:-1}
    C_{3}+\kappa(-1+\tilde{C}_4\delta)\leq -1.
\end{equation}

Since $\mathcal{L}(n;\Delta)\tau\leq \mathcal{F}(n;\Delta)$, we infer from \eqref{eq:step3-F} and \eqref{eq:-1} that
\begin{equation*}
    \mathcal{F}(n+1;\Delta)-\mathcal{F}(n;\Delta)
    \leq \left[(C_{3}+\kappa C_{4})\tau+\kappa C_{4}\big(\mathfrak{q}(n;\Delta)+\tilde{\mathfrak{q}}(n;\Delta)\big)\tau^{2}\right] \mathcal{F}(n;\Delta)-\mathcal{D}(n+1-;\Delta)\tau^{2}.
\end{equation*}
This proves \eqref{eq:Glimm-functional-decrease}.

Now, it remains to prove \eqref{3.17}. To do this,
using \eqref{eq:Glimm-functional-decrease}, we have
\begin{align*}
    \mathcal{F}(n;\Delta)
\leq &
\left[1+ (C_{3}+\kappa C_{4})\tau
+\kappa C_{4}\big(\mathfrak{q}(n-1;\Delta)+\tilde{\mathfrak{q}}(n-1;\Delta)\big)\tau^2\right]  \mathcal{F}(n-1;\Delta)-\mathcal{D}(n-;\Delta)\tau^{2}
\\
\leq &
\mathrm{e}^{(C_{3}+\kappa C_{4})\tau
+\kappa C_{4}\big(\mathfrak{q}(n-1;\Delta)+\tilde{\mathfrak{q}}(n-1;\Delta)\big)\tau^2}\mathcal{F}(n-1;\Delta).
\end{align*}
Iterating the above inequality from $n_0$ to $n-1$ yields
	\begin{equation*}
\mathcal{F}(n;\Delta)
\leq
\mathrm{e}^{(C_{3}+\kappa C_{4})(n-n_0)\tau
+\kappa C_{4}\sum_{p=n_{0}}^{n-1}
\big(\mathfrak{q}(p;\Delta)\tau^{2}+\tilde{\mathfrak{q}}(p;\Delta)\tau^{2}\big)}\mathcal{F}(n_0;\Delta),
	\end{equation*}
which gives \eqref{3.17}.

Since $0\leq n_{0}\leq n$ and
	$(-n_0)\tau\leq T+1$, it follows from Lemma \ref{lem:nonlinear-estimates} that
\begin{equation*}
\sum_{p=n_{0}}^{n-1}
\big(\mathfrak{q}(p;\Delta)\tau^{2}+\tilde{\mathfrak{q}}(p;\Delta)\tau^{2}\big)
\leq
\sum_{p=n_{0}}^{n-1}\sum_{j=-\infty}^{+\infty}|u_{j}^{p}|^{2}|v_{j}^{p}|^{2}\tau^{2}
+
\sum_{p=n_{0}}^{n-1}\sum_{j=-\infty}^{+\infty}|\tilde{u}_{j}^{p}|^{2}|\tilde{v}_{j}^{p}|^{2}\tau^{2}
\leq 2c(T).
	\end{equation*}
Consequently, \eqref{3.17-1} follows with
$C_5(T)=(C_{3}+\kappa C_{4} )(T+1)+2\kappa C_{4}c(T)$.
This completes the proof of Proposition~\ref{pro:functional-decreasing}.	
\end{proof}

 \section{Compactness of the set of time-splitting solutions in $\mathrm{C}([0,T]; \mathrm{L}^2(\mathbb{R}))$}
\label{sec:compactness}
In this section, we show that the set of time-splitting solutions is relatively compact in $\mathrm{C}([0,T]; \mathrm{L}^2(\mathbb{R}))$.
The proof of this result is based on the Ascoli characterization of compact sets in $\mathrm{C}([0,T]; \mathrm{L}^2(\mathbb{R}))$, see Lemma \ref{lem:A2}.
To do this, we first establish $\mathrm{L}^2$ stability for the time-splitting solutions for any $t\in[0,T]$. Then, to prove the space-time shift estimates, we establish estimates for $u^{(\tau)}$ and $v^{(\tau)}$ along their respective characteristic lines that are uniform with respect to $x$ and $t$.

\subsection{$\mathrm{L}^2$ stability for time-splitting solutions}
Based on Proposition~\ref{pro:functional-decreasing}, we decompose the strip $\mathbb{R}\times[0,T]$ into two subdomains and establish $\mathrm{L}^2$ stability estimates for time-splitting solutions in each subdomain.

Since $(u_{0},v_{0})\in \mathrm{L}^{2}(\mathbb{R})$, we can choose a constant $\Theta>0$ such that
\begin{equation}\label{eq:initial-edge}
\int_{|x|\geq \Theta/2}
\big(|u_{0}(x)|^{2}+|v_{0}(x)|^{2}\big)\mathrm{d}x\leq\frac{\delta}{4},
\end{equation}
where $\delta>0$ is the constant in Proposition~\ref{pro:functional-decreasing}.
Then, we define a subdomain of $\mathbb{R}\times[0,T]$ as
\begin{equation*}\mathscr{E}(\Theta,T)
=\{(x,t)\in\mathbb{R}\times [0,+\infty)\: :\: |x|\geq \Theta+t,\:0\leq t\leq T\}.\end{equation*}

On the other hand, recall that $\Lambda(x_1,t_{1};t_{0})$ is a characteristic triangle as defined in \eqref{eq:characteristic-triangle}, see Fig.~\ref{fig2}.
\begin{figure}[h]
	\centering
	\begin{tikzpicture}[scale=0.64]
		\draw[dashed](-1,0)--(9,0) node [right]{\small$t=t_{0}$};
		\draw[dashed](-1,2)--(9,2) node [right]{\small$t=t_{1}$};
		\draw[thick](2,0)--(4,2)--(6,0)--(2,0);
		\filldraw[black,opacity=0.0](2,0)--(4,2)--(6,0)--cycle;
		\node at (2.1,-.3){\scriptsize$(x_1-t_{1}+t_{0},t_{0})$};
		\node at (6.5,-.3){\scriptsize$(x_1+t_{1}-t_{0},t_{0})$};
		\node at (4,2.3){\scriptsize$(x_1,t_{1})$};
		\node at (4,0.6){\scriptsize$\Lambda(x_1,t_{1};t_{0})$};
	\end{tikzpicture}
	\caption{\label{fig2}The characteristic triangle $\Lambda(x_1,t_{1};t_{0})$}
\end{figure}

Noting that the $\mathrm{L}^{2}$-difference between two time-splitting solutions is estimated in a characteristic triangle in Proposition~\ref{pro:functional-decreasing}, we divide the strip $\mathbb{R}\times[0,T]$ into the following two subdomains, see Fig.~\ref{fig3}, namely,
\begin{equation*}
	\mathbb{R}\times [0,T]=\mathscr{E}(\Theta,T)\cup\big((\mathbb{R}\times [0,T])\cap\Lambda(0,\Theta+4T;0)\big).
\end{equation*}

\begin{figure}[h]
	\centering
	\begin{tikzpicture}[scale=1]
		\draw[](-6,0)--(6,0) node [right] {\small$t=0$};
		\draw[thick,dashed](-6,0.4)--(6,0.4) node [right] {\small$t=T$};
		\draw[line width=1pt](-2.6,0)--(0,2.6)--(2.6,0)--(-2.6,0) node at (0,0.2) {\scriptsize$\mathbf{\Lambda(0,\Theta+4T;0)}$};
		\fill (-2.6,0) circle (1pt);
		\fill (0,0) circle (1pt);
		\fill (2.6,0) circle (1pt);
		\fill (-1,0) circle (1pt);
		\fill (1,0) circle (1pt);
		\node at (-2.9,-0.3){\small$(-\Theta-4T,0)$};
		\node at (-1,-0.3){\small$(-\Theta,0)$};
		\node at (0,-0.3){\small$(0,0)$};
		\node at (1,-0.3){\small$(\Theta,0)$};
		\node at (2.9,-0.3){\small$(\Theta+4T,0)$};
		\node at (0,2.8){\small$(0,\Theta+4T)$};
		\draw[](-1,0)--(-3.5,2.5);
		\draw[](1,0)--(3.5,2.5);
		\node[rotate=-45]  at (-2.7,1.4) {\small$x=-\Theta-t$};
		\node[rotate=45] at (2.7,1.4){\small$x=\Theta+t$};
		\node at (3.3,0.2) {\scriptsize$\mathscr{E}(\Theta,T)$};
		\node at (-3.3,0.2) {\scriptsize$\mathscr{E}(\Theta,T)$};
	\end{tikzpicture}
	\caption{\label{fig3} Division of the strip $\mathbb{R}\times [0,T]$}
\end{figure}

Next, for any $|\mu|>0$ and any $\tau \in (0,\bar{\tau})$, where $\bar{\tau}\in(0,1)$ is given by \eqref{eq:c_0}, we consider the difference
\begin{equation*}
(u^{(\tau)}(x+\mu,t),\, v^{(\tau)}(x+\mu,t))
- (u^{(\tau)}(x,t),\, v^{(\tau)}(x,t))
\end{equation*}
on the subdomains $\mathscr{E}(\Theta,T)$ and
$(\mathbb{R}\times[0,T])\cap\Lambda(0,\Theta+4T;0)$.

To obtain $\mathrm{L}^2$ estimates for the difference between the two solutions in these subdomains, we establish the following $\mathrm{L}^2$ stability estimate in a characteristic triangle $\Lambda(j_1\tau,n_1\tau;n_0\tau)$ for all $t\in[n_0\tau,T]$, as a corollary of Proposition~\ref{pro:functional-decreasing}.

\begin{proposition}\label{pro:stability-bounds}
Let $T>0$ and let $\bar{\tau}\in(0,1)$ be the constant given by \eqref{eq:c_0}.
Suppose that $\mu=j_0 \tau$ for some integer $j_0=\pm1,\pm2,\ldots$ and that $(\tilde{u}^{(\tau)},\tilde{v}^{(\tau)})(x,t)=\big(u^{(\tau)},v^{(\tau)}\big)(x+\mu,t)$ for $(x,t)\in \Lambda(j_1\tau,n_1\tau;n_0\tau)$ with $0\leq n_0\leq n_1$.
If $\mathfrak{s}(n_0;\Delta)\tau\leq\delta$ and $\tilde{\mathfrak{s}}(n_0;\Delta)\tau\leq\delta$, then for all $\tau\in(0,\bar{\tau})$ and $t\in[n_0\tau,T]$, it holds that
	\begin{align*}
	 &\int_{j_1\tau-(n_1\tau-t)}^{j_1\tau+(n_1\tau-t)}
	\left(|u^{(\tau)}(x,t)-\tilde{u}^{(\tau)}(x,t)|^2
	+|v^{(\tau)}(x,t)-\tilde{v}^{(\tau)}(x,t)|^2\right)\mathrm{d}x\\
	\leq&
	\mathrm{e}^{\mathcal{K}(\left[\frac{t}{\tau}\right],n_0)}\mathcal{F}(n_{0};\Delta)
	\leq
	\mathrm{e}^{C_5(T)}\mathcal{F}(n_{0};\Delta),
\end{align*}
where the positive constants $\delta$ and $C_5(T)$ are given by Proposition \ref{pro:functional-decreasing}, and
{\small\begin{equation*}
\mathcal{K}\Big(\Big[\frac{t}{\tau}\Big],n_0\Big)=(C_{3}+\kappa C_{4})\Big(\Big[\frac{t}{\tau}\Big]-n_0\Big)\tau
+\kappa C_{4}\sum_{p=n_0}^{{\tiny\left[\frac{t}{\tau}\right]}-1}\big(\mathfrak{q}(p;\Delta)
+\tilde{\mathfrak{q}}(p;\Delta)\big)\tau^2.\end{equation*}
}
Here, the notation $\left[\frac{t}{\tau}\right]$ denotes the greatest integer less than or equal to $\frac{t}{\tau}$.
\end{proposition}
\begin{proof}
For $(x,t)\in \Lambda(j_1\tau,n_1\tau;n_0\tau)$ with $t\in[n_0\tau,T]$, we set
\begin{equation*}
U^{(\tau)}(x,t)=u^{(\tau)}(x,t)-\tilde{u}^{(\tau)}(x,t),\quad
V^{(\tau)}(x,t)=v^{(\tau)}(x,t)-\tilde{v}^{(\tau)}(x,t).
\end{equation*}
Then, we have
\begin{equation*}
   U_j^n=u_j^n-\tilde{u}_j^n=u_j^n-u_{j+j_0}^n,\quad
    V_j^n=v_j^n-\tilde{v}_j^n=v_j^n-v_{j+j_0}^n.
\end{equation*}

If $t=n\tau$ for $n=0,1,2,\ldots$ with $0\leq n_{0}\leq n$ and  $n\leq (T+1)/\tau$,
then  Proposition~\ref{pro:functional-decreasing} gives
    \begin{align}
		&\int_{(j_1-n_1+n)\tau}^{(j_1+n_1-n)\tau}
		\left(
        |U^{(\tau)}(x,n\tau)|^2
        +|V^{(\tau)}(x,n\tau)|^2\right)\mathrm{d}x\nonumber\\
        =&
        \sum_{j=j_1-n_1+n}^{j_1+n_1-n}
        \left(|U_j^n|^2+|V_j^n|^2\right)\tau=\mathcal{L}(n;\Delta)\tau\nonumber\\
\leq&
\mathcal{F}(n;\Delta)
\leq  \mathrm{e}^{(C_{3}+\kappa C_{4})(n-n_0)\tau
+\kappa C_{4}\sum_{p=n_0}^{n-1}\big(\mathfrak{q}(p;\Delta)+\tilde{\mathfrak{q}}(p;\Delta)\big)\tau^2}\mathcal{F}(n_{0};\Delta)
\leq \mathrm{e}^{C_5(T)}\mathcal{F}(n_{0};\Delta).\label{eq:decrease-j}
	\end{align}
	
It remains to examine the case $t\in(n\tau,(n+1)\tau)\cap[n_0\tau,T]$ for integers $n$ satisfying $0\leq n_{0}\leq n\leq \left[\frac{T}{\tau}\right]$.
Recall that $\left[\frac{t}{\tau}\right]$ denotes the greatest integer less than or equal to $\frac{t}{\tau}$.
Thus, in this case $\left[\frac{t}{\tau}\right]=n$.
Then, by \eqref{eq:UV-sub1}, we have
\begin{equation*}
	U^{(\tau)}(x,t)=U^{(\tau)}\left(x-\left(t-n\tau\right),n\tau\right),
	\quad
	V^{(\tau)}(x,t)=V^{(\tau)}\left(x+\left(t-n\tau\right),n\tau\right),
\end{equation*}
which together with \eqref{eq:decrease-j} gives
\begin{align*}
		&\int_{j_1\tau-n_1\tau+t}^{j_1\tau+n_1\tau-t}
		\left(|U^{(\tau)}(x,t)|^2
        +|V^{(\tau)}(x,t)|^2\right)\mathrm{d}x\\
        =&
        \int_{j_1\tau-n_1\tau+t}^{j_1\tau+n_1\tau-t}
		\left(\left|U^{(\tau)}\left(x-\left(t-n\tau\right),n\tau\right)\right|^2
		+
        \left|V^{(\tau)}\left(x+\left(t-n\tau\right),n\tau\right)\right|^2\right)\mathrm{d}x\\
		\leq &
		\int_{(j_1-n_1+n)\tau}^{(j_1+n_1-n)\tau}
		\left(\left|U^{(\tau)}\left(x,n\tau\right)\right|^2
		+
        \left|V^{(\tau)}\left(x,n\tau\right)\right|^2\right)\mathrm{d}x
		\leq  \mathcal{F}\left(n;\Delta\right),
\end{align*}
which together with \eqref {eq:decrease-j} completes the proof.
\end{proof}

Next, we establish the $\mathrm{L}^2$ stability estimates for time-splitting solutions in the following two cases, depending on whether $(x,t) \in \mathscr{E}(\Theta,T)$ or
$(x,t)\in (\mathbb{R}\times[0,T]) \cap \Lambda(0, \Theta + 4T;0)$.

\textbf{Case 1. In the subdomain $\mathscr{E}(\Theta,T)$.}
For $(x,t)\in\mathscr{E}(\Theta,T)$, the following lemma provides the $\mathrm{L}^2$ stability estimate for time-splitting solutions to \eqref{eq:NLDE}--\eqref{eq:NLDE-initail-data}.
\begin{lemma}\label{lem:subcase1}
For any $T>0$, there exists a constant $\tau_{\scriptscriptstyle\mathscr{E}}>0$ and a constant $C_{0}(T)>0$, depending on $T$, $c_0$ and the system \eqref{eq:NLDE}, such that, for all  $\tau\in(0,\tau_{\scriptscriptstyle\mathscr{E}})$ and $|\mu|=|j_0\tau|\leq \frac{\Theta}{4}$ with some integer $j_0\neq0$, it holds that
\begin{align*}
&\sup_{0\leq t\leq T}\int_{|x|\geq \Theta+t}\big(|u^{(\tau)}(x+\mu,t)-u^{(\tau)}(x,t)|^{2}+|v^{(\tau)}(x+\mu,t)-v^{(\tau)}(x,t)|^{2}\big)\mathrm{d}x\\
\leq
&C_{0}(T)\int_{|x|\ge	\Theta}\big(|u^{(\tau)}(x+\mu,0)-u^{(\tau)}(x,0)|^{2}+|v^{(\tau)}(x+\mu,0)-v^{(\tau)}(x,0)|^{2}\big)\mathrm{d}x, \end{align*}			
 where $\Theta>0$ is given by \eqref{eq:initial-edge}.
\end{lemma}	
\begin{proof}
Recall that $\bar{\tau}\in(0,1)$ is given by \eqref{eq:c_0} and that $\Theta>0$ is given by \eqref{eq:initial-edge}. Let $\delta>0$ be the constant given by Proposition~\ref{pro:functional-decreasing}. By \eqref{eq:initial-data-convergent}, we can choose a $ \tau_{\scriptscriptstyle\mathscr{E}}\in(0,\bar{\tau}]$ small enough such that $\tau_{\scriptscriptstyle\mathscr{E}}<\frac{\Theta}{36}$ and for any $\tau\in(0,\tau_{\scriptscriptstyle\mathscr{E}})$,
\begin{equation}\label{eq:approximate-initial-edge}
	\|u^{(\tau)}(x,0)-u_{0}(x)\|_{\mathrm{L}^{2}(|x|\geq\Theta/2)}^{2}	+\|v^{(\tau)}(x,0)-v_{0}(x)\|_{\mathrm{L}^{2}(|x|\geq\Theta/2)}^{2}
	\leq \frac{\delta}{4}.
\end{equation}
Then, using \eqref{eq:initial-edge} and \eqref{eq:approximate-initial-edge}, we have
\begin{equation}\label{eq:0-small-1}	
\|u^{(\tau)}(x,0)\|_{\mathrm{L}^{2}(|x|\geq\Theta/2)}^{2}	
+\|v^{(\tau)}(x,0)\|_{\mathrm{L}^{2}(|x|\geq\Theta/2)}^{2}
\leq \delta.
\end{equation}
Since $\tau\in (0,\tau_{\scriptscriptstyle\mathscr{E}})$, $\tau_{\scriptscriptstyle\mathscr{E}}<\frac{\Theta}{36}$ and $|\mu|=|j_0\tau|\leq \frac{\Theta}{4}$, we have $[\frac{\Theta}{\tau}]\tau-\mu\geq \Theta/2$. Thus, \eqref{eq:0-small-1} implies
\begin{equation}\label{eq:0-small-2}	
\|u^{(\tau)}(x+\mu,0)\|_{\mathrm{L}^{2}(|x|\geq[\frac{\Theta}{\tau}]\tau)}^{2}	
+
\|v^{(\tau)}(x+\mu,0)\|_{\mathrm{L}^{2}(|x|\geq[\frac{\Theta}{\tau}]\tau)}^{2}
\leq \delta,
\end{equation}
where $[\frac{\Theta}{\tau}]$ denotes the greatest integer less than or equal to  $\frac{\Theta}{\tau}$.

Then for $(x,t)\in\mathscr{E}(\Theta,T)$, we consider $\Lambda([\frac{\Theta}{\tau}]\tau+n_1\tau,n_1\tau;0)$ and $\Lambda(-[\frac{\Theta}{\tau}]\tau-n_1\tau,n_1\tau;0)$, as shown in
Fig. \ref{fig-subdomain-side}.

\begin{figure}[ht]
	\centering
	\begin{tikzpicture}[scale=0.95]
		\draw[](-6,0)--(-1,0);
		\draw[](1,0)--(6,0) node [right] {\small$t=0$};
		\draw[densely dashed](-1,0)--(1,0);
		\draw[densely dashed](-6,1.2)--(6,1.2) node [right] {\small$t=T$};
		\draw[densely dashed](-6,2)--(6,2) node [right] {\small$t=n_1\tau$};
		\draw[densely dashed](-6,0.8)--(6,0.8) node [right] {\small$t$};
		\fill (0,0) circle (1pt);
		\draw[thin](-1-0.5,0)--(-3.5-0.5,2.5);
		\fill (-1.5,0) circle (1pt);
		\node at (-1.9,-0.3) {\tiny$(-\Theta,0)$};
		\draw[thin](1+0.5,0)--(3.5+0.5,2.5);
		\fill (1.5,0) circle (1pt);
		\node at (1.8,-0.3) {\tiny$(\Theta,0)$};
		\node at (3,0.2) {\small$\mathscr{E}(\Theta,T)$};
		\node at (-3,0.2) {\small$\mathscr{E}(\Theta,T)$};
		\draw[thick](0.8,0)--(2.8,2);
		\draw[](2.8,2)--(3.3,2.5);
		\draw[thick](0.8,0)--(4.8,0);
		\fill (0.8,0) circle (1pt);
		\node at (0.7,-0.3) {\tiny$([\frac{\Theta}{\tau}]\tau,0)$};
\node at (-4.8,2.2) {\small$x=-\Theta-t$};
\node at (4.6,2.2){\small$x=\Theta+t$};
\node at (4,2.8){\small$x=[\frac{\Theta}{\tau}]\tau+t$};
\node at (-4,2.8){\small$x=-[\frac{\Theta}{\tau}]\tau-t$};
		\draw[thick](-0.8,0)--(-2.8,2);
		\draw[](-2.8,2)--(-3.3,2.5);
		\draw[thick](-0.8,0)--(-4.8,0);
		\fill (-0.8,0) circle (1pt);
		\node at (-0.6,-0.3) {\tiny$(-[\frac{\Theta}{\tau}]\tau,0)$};
		\draw[thick] (3-0.2,2)--(5-0.2,0);
		\draw[thick] (-3+0.2,2)--(-5+0.2,0);
		\draw[thick] (1.8-0.2,0.8)--(4.2-0.2,0.8);
		\draw[] (2.2-0.2,1.2)--(3.8-0.2,1.2);
		\draw[thick] (-1.8+0.2,0.8)--(-4.2+0.2,0.8);
		\draw[] (-2.2+0.2,1.2)--(-3.8+0.2,1.2);
		\fill (5-0.2,0) circle (1pt);
		\node at (5,-0.3) {\scriptsize$([\frac{\Theta}{\tau}]\tau+2n_1\tau,0)$};
		\fill (-5+0.2,0) circle (1pt);
		\node at (-5+0.2,-0.3) {\scriptsize$(-[\frac{\Theta}{\tau}]\tau-2n_1\tau,0)$};
	\end{tikzpicture}
	\caption{\label{fig-subdomain-side} The subdomain $\mathscr{E}(\Theta,T)$}
\end{figure}

From \eqref{eq:c_0} and the Definition \ref{def:Glimm-type-functional}, we deduce
\begin{align*}
&\mathcal{Q}(0;\Delta)\tau^2
=
\sum_{j_{1}-n_{1}\leq j<k\leq j_{1}+n_{1} }\Big({|U_{j}^{0}|^{2}\big(|v_{k}^{n}|^{2}+|\tilde{v}_{k}^{0}|^{2}\big)+|V_{k}^{0}|^{2}\big(|u_{j}^{0}|^{2}+|\tilde{u}_{j}^{0}|^{2}\big)}\Big)\tau^2\\
\leq&
\Big( \sum_{j=j_{1}-n_{1}}^{j_{1}+n_{1}}|U_{j}^{0}|^{2}\tau\Big)
\Big(\sum_{j=-\infty}^{\infty} \big(|v_{j}^{0}|^{2}+|\tilde{v}_{j}^{0}|^{2}\big)\tau \Big)
+\Big(\sum_{j=j_{1}-n_{1}}^{j_{1}+n_{1}}|{V}_{j}^{0}|^2\tau \Big)
\Big(\sum_{j=-\infty}^{\infty} \big(|u_{j}^{0}|^{2}+|\tilde{u}_{j}^{0}|^{2}\big)\tau\Big)\\
\leq &
2c_0\Big( \sum_{j=j_{1}-n_{1}}^{j_{1}+n_{1}}\big(|U_{j}^{0}|^{2}+|{V}_{j}^{0}|^2\big)\tau\Big)
=2c_0\mathcal{L}(0;\Delta)\tau.
\end{align*}
Therefore,
\begin{equation}\label{eq:F-L}
	\mathcal{F}(0;\Delta)=\mathcal{L}(0;\Delta)\,\tau+\kappa\mathcal{Q}(0;\Delta)\,\tau^{2}\leq (2c_0\kappa+1)\mathcal{L}(0;\Delta)\tau.
\end{equation}

Next, by \eqref{eq:0-small-1}, \eqref{eq:0-small-2}, \eqref{eq:F-L} and by Proposition~\ref{pro:stability-bounds}, we find that	
    \begin{align*}			
    &
    \sup_{0\leq t\leq T}
    \int_{[\frac{\Theta}{\tau}]\tau+t}^{[\frac{\Theta}{\tau}]\tau+2n_{1}\tau-t}\big(|u^{(\tau)}(x+\mu,t)-u^{(\tau)}(x,t)|^{2}+|v^{(\tau)}(x+\mu,t)-v^{(\tau)}(x,t)|^{2}\big)\mathrm{d}x\\		
    \leq &
  \mathrm{e}^{C_5(T)}\mathcal{F}(0;\Delta)\leq \mathrm{e}^{C_5(T)}(2c_0+1)\mathcal{L}(0;\Delta)\tau\\
    \leq&
    \mathrm{e}^{C_5(T)}(2c_0\kappa+1)\int_{[\frac{\Theta}{\tau}]\tau}^{[\frac{\Theta}{\tau}]\tau+2n_{1}\tau}\big(|u^{(\tau)}(x+\mu,0)-u^{(\tau)}(x,0)|^{2}+|v^{(\tau)}(x+\mu,0)-v^{(\tau)}(x,0)|^{2}\big)\mathrm{d}x	
    \end{align*}
   and
   \begin{align*}
	&\sup_{0\leq t\leq T}
    \int^{-[\frac{\Theta}{\tau}]\tau-t}_{-[\frac{\Theta}{\tau}]\tau-2n_{1}\tau+t}\big(|u^{(\tau)}(x+\mu,t)-u^{(\tau)}(x,t)|^{2}+|v^{(\tau)}(x+\mu,t)-v^{(\tau)}(x,t)|^{2}\big) \mathrm{d}x\\
	\leq&
\mathrm{e}^{C_5(T)}(2c_0\kappa+1)\int^{-[\frac{\Theta}{\tau}]\tau}_{-[\frac{\Theta}{\tau}]\tau-2n_{1}\tau}\big( |u^{(\tau)}(x+\mu,0)-u^{(\tau)}(x,0)|^{2}+|v^{(\tau)}(x+\mu,0)-v^{(\tau)}(x,0)|^{2}\big) \mathrm{d}x.
	\end{align*}
Consequently, by letting $n_{1} \to \infty$, see Fig.~\ref{fig-subdomain-side}, and setting $C_{0}(T) =\mathrm{e}^{C_5(T)}(2c_0\kappa+1)$, we complete the proof of Lemma \ref{lem:subcase1}.
\end{proof}

\textbf{Case 2. In the subdomain $(\mathbb{R}\times [0,T])\cap\Lambda(0,\Theta+4T;0)$.}
In the following lemma, we establish the $\mathrm{L}^2$ stability estimates for time-splitting solutions on the subdomain $(\mathbb{R} \times [0,T]) \cap \Lambda(0,\Theta+4T;0)$.
	\begin{lemma}\label{lem:subcase2}
For any $T>0$, there exists a constant $\tau_{\scriptscriptstyle\Lambda}>0$ and a constant $\tilde{C}_{0}(T)>0$, depending on $T$, $c_0$ and the system \eqref{eq:NLDE}, such that, for all $\tau \in (0,\tau_{\scriptscriptstyle\Lambda})$ and $|\mu|=|j_0\tau|\leq \frac{\Theta}{4}$ with some integer $j_0\neq0$, it holds that
	\begin{align}\label{eq:stability-subdomain2}
	&\sup_{0\leq t\leq T}\int_{-\Theta-4T+t}^{\Theta+4T-t}
	\big(|u^{(\tau)}(x+\mu,t)-u^{(\tau)}(x,t)|^{2}
     +
     |v^{(\tau)}(x+\mu,t)-v^{(\tau)}(x,t)|^{2}\big)\mathrm{d}x\nonumber\\
	\leq &
    \tilde{C}_{0}(T)
	\int_{-\Theta-4T}^{\Theta+4T}\big(|u^{(\tau)}(x+\mu,0)-u^{(\tau)}(x,0)|^{2}
+|v^{(\tau)}(x+\mu,0)-v^{(\tau)}(x,0)|^{2}\big)\mathrm{d}x.
\end{align}
\end{lemma}
\begin{proof}
We will use Proposition \ref{pro:stability-bounds} to get the desired result. First, for any interval $[x_1,x_2]\subset(-\infty,\infty)$ at time $t=n\tau\in[0,T]$, we consider two cases:
\begin{equation*}
[x_1,x_2] \subset (-\infty, -\Theta-n\tau) \cup (\Theta+n\tau, +\infty)
\quad \text{and} \quad
[x_1,x_2] \subset (-\Theta-4T+n\tau, \Theta+4T-n\tau).
\end{equation*}

For the former case, it follows from \eqref{eq:0-small-1}, \eqref{eq:0-small-2} and Lemma \ref{lemma2.3} that
   \begin{equation}\label{eq:former-case-2}
   \int_{x_1}^{x_2}
\big(|u^{(\tau)}|^{2}+|v^{(\tau)}|^{2}\big)(x,n\tau)\mathrm{d}x
\leq\delta
\:\: \text{ and } \:\:
\int_{x_1}^{x_2}
\big(|u^{(\tau)}|^{2}+|v^{(\tau)}|^{2}\big)(x+\mu,n\tau)\mathrm{d}x
\leq\delta
   \end{equation}
for $[x_1,x_2]\subset(-\infty,-\Theta-n\tau)\cup(\Theta+n\tau,+\infty)$ and for $|\mu|<\frac{\Theta}{4}$.

To deal with the latter case, by \eqref{2.8} and \eqref{2.9}, we deduce that for any interval $[x_1,x_2]\subset(-2\Theta-4T,2\Theta+4T)$,
   \begin{align}\label{eq:interval-any}
   	&\int_{x_1}^{x_2}\big(|u^{(\tau)}(x,n\tau)|^{2}+|v^{(\tau)}(x,n\tau)|^{2}\big)\mathrm{d}x\nonumber\\
   	\leq
   	&c_{2}(T)\int_{x_1}^{x_2}\Big(|u^{(\tau)}(x-n\tau,0)|^{2}+|v^{(\tau)}(x+n\tau,0)|^{2}\Big)\mathrm{d}x
   	+2m_{1}c_0c_{2}(T)(x_2-x_1)\nonumber\\
   	\leq &c_{2}(T)\Big(\int_{x_1-n\tau}^{x_2-n\tau}|u^{(\tau)}(x,0)|^{2}\mathrm{d}x
   	+\int_{x_1+n\tau}^{x_2+n\tau}|v^{(\tau)}(x,0)|^{2}\mathrm{d}x\Big)
   	+2m_{1}c_0c_{2}(T)(x_2-x_1).
   \end{align}
 We now estimate the right-hand side of \eqref{eq:interval-any}. Since $(u_{0},v_{0})\in \mathrm{L}^{2}(\mathbb{R})$, there exists a constant $\sigma\in(0,\min\{\Theta/8,\bar{\tau}\})$ such that for any $x_2-x_1\in(0,4\sigma]$,
\begin{equation}\label{4.2}	
c_{2}(T)\Big(
\|u_{0}(x)\|_{\mathrm{L}^{2}([x_1,x_2])}^{2}
+\|v_{0}(x)\|_{\mathrm{L}^{2}([x_1,x_2])}^{2}\Big)\leq \frac{\delta}{8},\qquad 2m_1c_0c_{2}(T)(x_2-x_1)\leq \frac{\delta}{2}.
\end{equation}
Then, by \eqref{eq:initial-data-convergent}, we choose a $\tau_{\scriptscriptstyle\Lambda}\in(0,\tau_{\scriptscriptstyle\mathscr{E}})$ small enough such that for any $x_2-x_1\in(0,4\sigma]$,
\begin{equation}\label{eq:approximate-initial}
	c_{2}(T)\Big(\|u^{(\tau)}(x,0)-u_{0}(x)\|_{\mathrm{L}^{2}([x_1,x_2])}^{2}
	+\|v^{(\tau)}(x,0)-v_{0}(x)\|_{\mathrm{L}^{2}([x_1,x_2])}^{2}\Big)\leq \frac{\delta}{8}.
\end{equation}
By the first estimate in \eqref{4.2} and by \eqref{eq:approximate-initial}, we deduce that if $\tau\in(0,\tau_{\scriptscriptstyle\Lambda})$, then for any interval $[x_1,x_2]\subset(-2\Theta-4T,2\Theta+4T)$ with $x_2-x_1\in(0,4\sigma]$,
   \begin{equation}\label{5.2.88888}
   	c_{2}(T)\Big(\|u^{(\tau)}(x,0)\|_{\mathrm{L}^{2}([x_1,x_2])}^{2}
   	+\|v^{(\tau)}(x,0)\|_{\mathrm{L}^{2}([x_1,x_2])}^{2}\Big)\leq \frac{\delta}{2}.
   \end{equation}

Consequently, for the latter case, substituting the second estimate \eqref{4.2} and \eqref{5.2.88888} into
\eqref{eq:interval-any} and using \eqref{eq:former-case-2}, we conclude that there exists a constant $\sigma>0$ such that, if $\tau\in(0,\tau_{\scriptscriptstyle\Lambda})$ and $n\tau\in[0,T]$, then for any interval $[x_1,x_2]\subset(-\Theta-4T+n\tau, \Theta+4T-n\tau)$ with $x_2-x_1\in(0,4\sigma]$ and for $|\mu|\leq \Theta/4$,
\begin{equation}\label{eq:interval-initial-2}
\int_{x_1}^{x_2}
\big(|u^{(\tau)}|^{2}+|v^{(\tau)}|^{2}\big)(x,n\tau)\mathrm{d}x
\leq\delta
\:\:\text{ and } \:\:
\int_{x_1}^{x_2}
\big(|u^{(\tau)}|^{2}+|v^{(\tau)}|^{2}\big)(x+\mu,n\tau)\mathrm{d}x
\leq\delta.
\end{equation}

Next, without loss of generality, we assume that
$\Theta+4T=N(4\sigma)$ with $N$ being a positive integer. Choose the grid points $P_{j,n}=(j\sigma,n\sigma)$ with $j=0,\pm1,\ldots,\pm4N$ and $n=0,1,\ldots,[\frac{T}{\sigma}]$,
in the domain $\Lambda(0,\Theta+4T;0)$, and we then consider the $L^2$ stability of time-splitting solutions on a characteristic triangle $\Lambda((j+2)\sigma,(n+2)\sigma;n\sigma)$, see Fig.~\ref{fig-4sigma}.
\begin{figure}[h]
	\begin{tikzpicture}
	\draw[densely dashed] (-6,0)--(6,0) node [right] {\small $t=n\sigma$};
	\draw[densely dashed] (-6,0.8)--(6,0.8)node [right] {\small $t=(n+1)\sigma$};;
	\draw[densely dashed] (-6,1.6)--(6,1.6)node [right] {\small $t=(n+2)\sigma$};;
	\draw[densely dashed] (0,0)--(0,1.6);
	\draw[densely dashed] (0.8,0)--(0.8,1.6);
	\draw[densely dashed] (-0.8,0)--(-0.8,1.6);
	\draw[densely dashed] (1.6,0)--(1.6,1.6);
	\draw[densely dashed] (-1.6,0)--(-1.6,1.6);
	\draw[densely dashed] (2.4,0)--(2.4,1.6);
	\draw[densely dashed] (-2.4,0)--(-2.4,1.6);
	\draw[densely dashed] (3.2,0)--(3.2,1.6);
	\draw[densely dashed] (-3.2,0)--(-3.2,1.6);
	\draw[densely dashed] (-4,0)--(-4,1.6);
	\draw[densely dashed] (4,0)--(4,1.6);
	\draw[thick] (-3.2,0)--(-1.6,1.6)--(0,0)--(-3.2,0);
	\draw[thin] (-3.2+0.8,0)--(-1.6+0.8,1.6)--(0+0.8,0)--(-3.2+0.8,0);
	\draw[thin] (-3.2+0.8*2,0)--(-1.6+0.8*2,1.6)--(0+0.8*2,0)--(-3.2+0.8*2,0);
	\draw[thin] (-3.2+0.8*3,0)--(-1.6+0.8*3,1.6)--(0+0.8*3,0)--(-3.2+0.8*3,0);
	\draw[thin] (-3.2+0.8*4,0)--(-1.6+0.8*4,1.6)--(0+0.8*4,0)--(-3.2+0.8*4,0);
	\fill (-3.2,0) circle (1pt);
	\node at (-3.2,-0.3) {\small $(j\sigma,n\sigma)$};
	\fill (-3.2+0.8*2,1.6) circle (1pt);
	\node at (-3.2+0.8*2,1.8) {\small $((j+2)\sigma,(n+2)\sigma)$};
	\fill (-3.2+0.8*4,0) circle (1pt);
	\node at (-3.2+0.8*4,-0.3) {\small $((j+4)\sigma,n\sigma)$};
	\end{tikzpicture}
	\caption{\label{fig-4sigma} The characteristic triangle $\Lambda((j+2)\sigma,(n+2)\sigma;n\sigma)$}
\end{figure}

Recall that $\left[\frac{t}{\sigma}\right]$ denotes the greatest integer less than or equal to $\frac{t}{\sigma}$. For any $t>0$, we let
\begin{equation*}\tilde{C}_2(t)
=(2c_0\kappa+1)4^{\left[\frac{t}{\sigma}\right]}
\mathrm{e}^{(C_{3}+\kappa C_{4})\left[\frac{t}{\sigma}\right]\sigma
+\kappa C_{4}\sum_{p=0}^{\left[\frac{t}{\sigma}\right]-1}
(\mathfrak{q}(p;\Delta)\sigma^2+\tilde{\mathfrak{q}}(p;\Delta)\sigma^2)}.\end{equation*}
Then using \eqref{eq:F-L}, \eqref{eq:former-case-2}, \eqref{eq:interval-initial-2},  and Proposition~\ref{pro:stability-bounds}, we find that for any $-4N\leq j\leq 4N-4$,
   \begin{align*}		
   	&\sup_{t\in[0,\sigma]}
   	\int_{j\sigma+t}^{(j+4)\sigma-t}
   	\big(|u^{(\tau)}(x+\mu,t)-u^{(\tau)}(x,t)|^{2}
   	+
   	|v^{(\tau)}(x+\mu,t)-v^{(\tau)}(x,t)|^{2}\big)\mathrm{d}x\\
   	\leq &
\frac{\tilde{C}_2(\sigma)}{4}\int_{j\sigma}^{(j+4)\sigma}\big(|u^{(\tau)}(x+\mu,0)-u^{(\tau)}(x,0)|^{2}
   	+|v^{(\tau)}(x+\mu,0)-v^{(\tau)}(x,0)|^{2}\big)\mathrm{d}x.
   \end{align*}
Hence, by summing the above inequalities over all $j\in[-4N,4N-4]$, we have
	\begin{align}\label{eq:induction-assumption}
	&\sup_{0\leq t\leq \sigma}\int_{-\Theta-4T+t}^{\Theta+4T-t}
	\big(|u^{(\tau)}(x+\mu,t)-u^{(\tau)}(x,t)|^{2}
     +
     |v^{(\tau)}(x+\mu,t)-v^{(\tau)}(x,t)|^{2}\big)\mathrm{d}x\nonumber\\
	\leq & \tilde{C}_2(\sigma)
	\int_{-\Theta-4T}^{\Theta+4T}\big(|u^{(\tau)}(x+\mu,0)-u^{(\tau)}(x,0)|^{2}
+|v^{(\tau)}(x+\mu,0)-v^{(\tau)}(x,0)|^{2}\big)\mathrm{d}x,
\end{align}
		
Arguing by induction on $n$, we assume that for $t\in[0,n\sigma]\cap[0,T]$,
	\begin{align*}
	&\sup_{t\in[0,n\sigma]\cap[0,T]}\int_{-\Theta-4T+t}^{\Theta+4T-t}
	\big(|u^{(\tau)}(x+\mu,t)-u^{(\tau)}(x,t)|^{2}
     +
     |v^{(\tau)}(x+\mu,t)-v^{(\tau)}(x,t)|^{2}\big)\mathrm{d}x\\
	\leq &\tilde{C}_2(n\sigma)
	\int_{-\Theta-4T}^{\Theta+4T}\big(|u^{(\tau)}(x+\mu,0)-u^{(\tau)}(x,0)|^{2}
+|v^{(\tau)}(x+\mu,0)-v^{(\tau)}(x,0)|^{2}\big)\mathrm{d}x,
\end{align*}
Then, we proceed to prove \eqref{eq:stability-subdomain2} for $t\in[0,(n+1)\sigma]\cap[0,T]$. Note that $\Lambda(0,\Theta + 4T;0)\cap(\mathbb{R}\times[0,(n+1)\sigma])$ is a subset of
\begin{equation*}\left( \Lambda(0,\Theta + 4T;0)\cap(\mathbb{R}\times[0,n\sigma])\right) \cup
   \left( \cup_{-4N+n\leq j\leq 4N-n-4} \Lambda((j+2)\sigma,(n+2)\sigma;n\sigma)\right).\end{equation*}
For any characteristic triangle
\begin{equation*}
	\Delta((j+2)\sigma,(n+2)\sigma;n\sigma)
	\end{equation*}
with $-4N+n\leq j\leq4N-n-4$, it follows from \eqref{eq:F-L},
\eqref{eq:former-case-2}, \eqref{eq:interval-initial-2}--\eqref{eq:induction-assumption} and Proposition~\ref{pro:stability-bounds}
that
	\begin{align*}
		&\sup_{t\in[0,(n+1)\sigma]\cap[0,T]}
        \int_{j\sigma+t-n\sigma}^{(j+4)\sigma-t+n\sigma}
		\Big(|u^{(\tau)}(x+\mu,t)-u^{(\tau)}(x,t)|^{2}
		+|v^{(\tau)}(x+\mu,t)-v^{(\tau)}(x,t)|^{2}\Big)\mathrm{d}x\\
\leq &
\sup_{t\in[0,(n+1)\sigma]\cap[0,T]}
\Big(\mathrm{e}^{(C_{3}+\kappa C_{4})\left(\left[\frac{t}{\sigma}\right]-n\right)\sigma
+\kappa C_{4}\sum_{p=n}^{\left[\frac{t}{\sigma}\right]-1}\big(\mathfrak{q}(p;\Delta)\sigma^2+\tilde{\mathfrak{q}}(p;\Delta)\sigma^2\big)}\Big)\\
&\cdot \int_{j\sigma}^{(j+4)\sigma}
	\big(|u^{(\tau)}(x+\mu,n\sigma)-u^{(\tau)}(x,n\sigma)|^{2}
     +
     |v^{(\tau)}(x+\mu,n\sigma)-v^{(\tau)}(x,n\sigma)|^{2}\big)\mathrm{d}x.
	\end{align*}

Finally, summing the above inequalities over $j=-4N+n,\ldots,4N-n-4$, we obtain
	\begin{align*}
	&\sup_{t\in[0,(n+1)\sigma]\cap[0,T]}\int_{-\Theta-4T+t}^{\Theta+4T-t}
	\big(|u^{(\tau)}(x+\mu,t)-u^{(\tau)}(x,t)|^{2}
     +
     |v^{(\tau)}(x+\mu,t)-v^{(\tau)}(x,t)|^{2}\big)\mathrm{d}x\\
     \leq&
     4\sup_{t\in[0,(n+1)\sigma]\cap[0,T]}
     \Big(\mathrm{e}^{(C_{3}+\kappa C_{4})\left(\left[\frac{t}{\sigma}\right]-n\right)\sigma
+\kappa C_{4}\sum_{p=n}^{\left[\frac{t}{\sigma}\right]-1}\big(\mathfrak{q}(p;\Delta)\sigma^2+\tilde{\mathfrak{q}}(p;\Delta)\sigma^2\big)}\Big)\\
&\cdot\int_{-\Theta-4T+n\sigma}^{\Theta+4T-n\sigma}
	\big(|u^{(\tau)}(x+\mu,n\sigma)-u^{(\tau)}(x,n\sigma)|^{2}
     +
     |v^{(\tau)}(x+\mu,n\sigma)-v^{(\tau)}(x,n\sigma)|^{2}\big)\mathrm{d}x\\
	\leq &\tilde{C}_2((n+1)\sigma)
	\int_{-\Theta-4T}^{\Theta+4T}\big(|u^{(\tau)}(x+\mu,0)-u^{(\tau)}(x,0)|^{2}
+|v^{(\tau)}(x+\mu,0)-v^{(\tau)}(x,0)|^{2}\big)\mathrm{d}x.
\end{align*}

Since $\sigma>0$ is fixed in the sequel, there exists a constant $\tilde{C}_{0}(T)$, depending on $T$, $c_0$ and the system \eqref{eq:NLDE}, such that, for $t\in[0,T]$,
 \begin{align*}
  \tilde{C}_2(t)
=&(2c_0\kappa+1)4^{\left[\frac{t}{\sigma}\right]}
\mathrm{e}^{(C_{3}+\kappa C_{4})\left[\frac{t}{\sigma}\right]\sigma
+\kappa C_{4}\sum_{p=0}^{\left[\frac{t}{\sigma}\right]-1}
(\mathfrak{q}(p;\Delta)\sigma^2
+\tilde{\mathfrak{q}}(p;\Delta)\sigma^2)}\\
&\leq
(2c_0\kappa+1)4^{\left[\frac{T}{\sigma}\right]}
\mathrm{e}^{C_5(T)}
\leq\tilde{C}_0(T),
\end{align*}
where $\displaystyle C_5(T)=(C_{3}+\kappa C_{4})(T+1)
+2\kappa C_{4}c(T)$ given in the proof of Proposition~\ref{pro:functional-decreasing}. This completes the proof.
\end{proof}

By Lemmas \ref{lem:subcase1}--\ref{lem:subcase2}, we can extend the case $\mu=j_0\tau$ for some integer $j_0\neq0$ to the general case $\mu\in\mathbb{R}$. This result will be used to establish the compactness of the time-splitting solutions.
\begin{corollary}\label{eq:coro-strip}
Let $\Theta>0$ be given by \eqref{eq:initial-edge}.
For any $T>0$ and any $\varepsilon>0$, there exists a constant $\tau_{\flat}>0$, such that, for all  $\tau\in(0,\tau_{\flat})$ and $\mu\in(0,\tau_{\flat})$, it holds that
		\begin{align*}
			\sup_{0\leq t\leq T}\int_{\mathbb{R}}
			\big(|u^{(\tau)}(x+\mu,t)-u^{(\tau)}(x,t)|^{2}+|v^{(\tau)}(x+\mu,t)-v^{(\tau)}(x,t)|^{2}\big)\mathrm{d}x
			\leq
			\varepsilon.
			\end{align*}		
\end{corollary}
\begin{proof}
	By Lemma \ref{lem:subcase1} and Lemma \ref{lem:subcase2}, we deduce that for all $\tau \in (0,\min\{\tau_{\scriptscriptstyle\mathscr{E}},\tau_{\scriptscriptstyle\Lambda}\})$ and $|\mu|=|j_0\tau|\leq \frac{\Theta}{4}$ with some integer $j_0\neq0$,
		\begin{align}\label{eq:strip}
			&\sup_{0\leq t\leq T}\int_{\mathbb{R}}
			\big(|u^{(\tau)}(x+\mu,t)-u^{(\tau)}(x,t)|^{2}
			+
			|v^{(\tau)}(x+\mu,t)-v^{(\tau)}(x,t)|^{2}\big)\mathrm{d}x\nonumber\\
			\leq &
		\big( C_{0}(T) + \tilde{C}_{0}(T)\big)
			\int_{\mathbb{R}}\big(|u^{(\tau)}(x+\mu,0)-u^{(\tau)}(x,0)|^{2}
			+|v^{(\tau)}(x+\mu,0)-v^{(\tau)}(x,0)|^{2}\big)\mathrm{d}x,
		\end{align}
where the positive constants $C_0(T)$ and $\tilde{C}_0(T)$ are given by Lemmas \ref{lem:subcase1} and \ref{lem:subcase2}, respectively.
	
Now, we consider the case in which $\mu\in(0,\frac{\Theta}{4})$.
Since $x\in\mathbb{R}$ and $\tau \in (0,\min\{\tau_{\scriptscriptstyle\mathscr{E}},\tau_{\scriptscriptstyle\Lambda}\})$, there exists a unique $j\in\{0,\pm1,\pm2,\ldots\}$ such that $x\in[j\tau,(j+1)\tau)$.
Moreover, for $\mu\in(0,\frac{\Theta}{4})$, we can rewrite $x+\mu$ as follows:
\begin{equation*}
	x+\mu=x+\left[\frac{\mu}{\tau}\right]\tau+\xi, \quad 0\leq \xi< \tau,
\end{equation*}
where $\left[\frac{\mu}{\tau}\right]$ denotes the greatest integer less than or equal to $\frac{\mu}{\tau}$.

If $x\in [j\tau,(j+1)\tau-\xi]$, then $x+\mu\in [(j+\left[\frac{\mu}{\tau}\right])\tau,(j+\left[\frac{\mu}{\tau}\right]+1)\tau]$. Therefore, from our construction in \eqref{eq:construction}, we have
\begin{equation*}
	u^{(\tau)}(x+\mu,t)=u^{(\tau)}((j+\left[\frac{\mu}{\tau}\right])\tau,t),\quad
	v^{(\tau)}(x+\mu,t)=v^{(\tau)}((j+\left[\frac{\mu}{\tau}\right])\tau,t).
\end{equation*}

If $x\in [(j+1)\tau-\xi,(j+1)\tau]$, then $x+\mu\in[(j+\left[\frac{\mu}{\tau}\right]+1)\tau,(j+\left[\frac{\mu}{\tau}\right]+2)\tau]$. Thus, it follows from \eqref{eq:construction} that
\begin{equation*}
	u^{(\tau)}(x+\mu,t)=u^{(\tau)}((j+\left[\frac{\mu}{\tau}\right]+1)\tau,t),\quad
	v^{(\tau)}(x+\mu,t)=v^{(\tau)}((j+\left[\frac{\mu}{\tau}\right]+1)\tau,t).
\end{equation*}

Consequently, for all $\tau \in (0,\min\{\tau_{\scriptscriptstyle\mathscr{E}},\tau_{\scriptscriptstyle\Lambda}\})$, $\mu\in(0,\frac{\Theta}{4})$ and $t\in[0,T]$, using \eqref{eq:strip}, we obtain
	\begin{align*}
		&\int_{\mathbb{R}}
		\big(|u^{(\tau)}(x+\mu,t)-u^{(\tau)}(x,t)|^{2}
		+
		|v^{(\tau)}(x+\mu,t)-v^{(\tau)}(x,t)|^{2}\big)\mathrm{d}x\\
		= &
		\sum_{j\in\mathbb{Z}}\int_{j\tau}^{(j+1)\tau-\xi}
		\Big(|u^{(\tau)}((j+\left[\frac{\mu}{\tau}\right])\tau,t)-u^{(\tau)}(j\tau,t)|^{2}
		+
		|v^{(\tau)}((j+\left[\frac{\mu}{\tau}\right])\tau,t)-v^{(\tau)}(j\tau,t)|^{2}\Big)\mathrm{d}x\\
		+&
		\sum_{j\in\mathbb{Z}}\int_{(j+1)\tau-\xi}^{(j+1)\tau}
		\Big(|u^{(\tau)}((j+\left[\frac{\mu}{\tau}\right]+1)\tau,t)-u^{(\tau)}(j\tau,t)|^{2}
		+
		|v^{(\tau)}((j+\left[\frac{\mu}{\tau}\right]+1)\tau,t)-v^{(\tau)}(j\tau,t)|^{2}\Big)\mathrm{d}x\\
		\leq&
		\big( C_{0}(T) + \tilde{C}_{0}(T)\big)\mathrm{I}_1+\big( C_{0}(T) + \tilde{C}_{0}(T)\big)\mathrm{I}_2,
	\end{align*}
where
\begin{small}
\begin{align*}
&\mathrm{I}_1=	\sum_{j\in\mathbb{Z}}\int_{j\tau}^{(j+1)\tau}\Big(|u^{(\tau)}((j+\left[\frac{\mu}{\tau}\right])\tau,0)-u^{(\tau)}(j\tau,0)|^{2}
	+|v^{(\tau)}((j+\left[\frac{\mu}{\tau}\right])\tau,0)-v^{(\tau)}(j\tau,0)|^{2}\Big)\mathrm{d}x,\\
& \mathrm{I}_2=
\sum_{j\in\mathbb{Z}}\int_{j\tau}^{(j+1)\tau}	\Big(|u^{(\tau)}((j+\left[\frac{\mu}{\tau}\right]+1)\tau,0)-u^{(\tau)}(j\tau,0)|^{2}
+|v^{(\tau)}((j+\left[\frac{\mu}{\tau}\right]+1)\tau,0)-v^{(\tau)}(j\tau,0)|^{2}\Big)\mathrm{d}x.
\end{align*}
\end{small}
Then, by \eqref{eq:initial-data-convergent}, we know that $(u^{\tau}(x,0),v^{\tau}(x,0))\in \mathrm{L}^2(\mathbb{R})$.
Thus, given $\varepsilon>0$, it follows from \cite[Lemma 4.3]{BrezisBook} that there exists a positive constant $\tau_{\flat}\leq \min\{\tau_{\scriptscriptstyle\mathscr{E}},\tau_{\scriptscriptstyle\Lambda}\}$, such that for all $\tau \in (0,\tau_{\flat})$ and $\mu\in(0,\tau_{\flat})$, there holds that
\begin{align*}
	\mathrm{I}_1
	\leq \frac{\varepsilon}{2\big( C_{0}(T) + \tilde{C}_{0}(T)\big)},\quad
	\mathrm{I}_2
	\leq \frac{\varepsilon}{2\big( C_{0}(T) + \tilde{C}_{0}(T)\big)}.
\end{align*}
Therefore, the proof is complete.
\end{proof}

\subsection{Compactness of the family of time-splitting solutions}
In this subsection, we prove the compactness of the family of time-splitting solutions. For any $T>0$, we define
\begin{equation}\label{eq:ST}
S_T := \Big\{\big(u^{(\tau)},v^{(\tau)}\big)\big|_{\mathbb{R}\times[0,T]}\:: \:\tau\in(0,\tau_{*})\Big\},
\end{equation}
where $\tau_{*}=\min\{\bar{\tau},\tau_{\flat}\}$, with $\bar{\tau}$ and
$\tau_{\flat}$ defined in \eqref{eq:c_0} and Corollary~\ref{eq:coro-strip}.

It follows from Corollary~\ref{eq:coro-strip} that, to prove the compactness of $S_T$, it suffices to establish a uniform estimate on $(u^{(\tau)}, v^{(\tau)})$ along the characteristics, uniformly in both $x$ and $t$. This is provided by the following lemma.
\begin{lemma}\label{5.6.6}
For $T>0$ and any $\varepsilon>0$, there exist constants $\delta_1>0$ and $C(T)>0$, such that for any $(u^{(\tau)},v^{(\tau)})\in S_T$, if $\tau\in(0,\delta_1)$ and $h\in(0,\delta_1)$, then for any $t_0\in[0,T]$, it holds that
\begin{equation*}
\int_{-\infty}^{+\infty}
\left|u^{(\tau)}(x+h,t_{0}+h)-u^{(\tau)}(x,t_{0})\right|^{2}\mathrm{d}x<C(T)\varepsilon
\end{equation*}
and
\begin{equation*}
\int_{-\infty}^{+\infty}
\left|v^{(\tau)}(x-h,t_{0}+h)-v^{(\tau)}(x,t_{0})\right|^{2}\mathrm{d}x<C(T)\varepsilon.
\end{equation*}
\end{lemma}
\begin{proof}
We divide the proof into two steps.

\textit{Step 1. Estimates on $(u^{(\tau)},v^{(\tau)})$ at the grid points $(x,t_0)=(j\tau,n_0\tau)$ for $h=(n_2-n_0)\tau<1$ with $0\leq n_0\leq n_2$ and $n_2\leq (T+1)/\tau$.}
The first equation in \eqref{eq:subproblem2}, evaluated at $x=j\tau$, yields
	\begin{equation*}
	\Big|\int_{0}^{\tau}\frac{\mathrm{d}u_{j}^{n,2}(s)}{\mathrm{d}s}\mathrm{d}s\Big|
	\leq m\int_{0}^{\tau}|v_{j}^{n,2}(s)|\mathrm{d}s+(|\alpha|+4|\beta|)\int_{0}^{\tau}|u_{j}^{n,2}(s)||v_{j}^{n,2}(s)|^{2}\mathrm{d}s,
	\end{equation*}
which, together with \eqref{eq:characteristic-estimate-uv}, \eqref{2.11} and \eqref{2.14}, implies
\begin{align}\label{eq:sum}
|u_{j}^{n+1}-u_{j-1}^{n}|
\leq&
\left[ m+(|\alpha|+4|\beta|)\mathrm{e}^{2|\beta|c_0}\Big(|u_{j-1}^{n}|^{2}|v_{j+1}^{n}|^{2}
+m_{1}(|u_{j-1}^{n}|^{2}+|v_{j+1}^{n}|^{2})^{2}\tau\Big)^\frac{1}{2}\right]\nonumber \\
& \cdot\Big(|v_{j+1}^{n}|^{2}+m_{1}(|u_{j-1}^{n}|^{2}+|v_{j+1}^{n}|^{2})\tau+c_{1}|u_{j-1}^{n}|^{2}|v_{j+1}^{n}|^{2}\tau\Big)^{\frac{1}{2}}\tau.
\end{align}
Let $j_{h}$ denote $j+n_2-n_0$. Then using \eqref{eq:sum}  and setting $\bar{C}=(|\alpha|+4|\beta|)\mathrm{e}^{2|\beta|c_0}$, we deduce that
\begin{align*}
&|u_{j+n_2-n_0}^{n_2}-u_{j}^{n_0}|
=|u_{j_{h}}^{n_2}-u_{j}^{n_0}|\\
=&|u_{j_{h}}^{n_2}
-u_{j_{h}-1}^{n_2-1}
+u_{j_{h}-1}^{n_2-1}
-u_{j_{h}-2}^{n_2-2}
+\ldots
+u_{j_{h}-p}^{n_2-p}
-u_{j_{h}-p-1}^{n_2-p-1}
+\ldots
+u_{j+1}^{n_0+1}
-u_{j}^{n_0}|\\
\leq &
\sum_{p=0}^{n_2-n_0-1}
|u_{j_{h}-p}^{n_2-p}
-u_{j_{h}-p-1}^{n_2-p-1}|\\
\leq&
\sum_{p=0}^{n_2-n_0-1}
\left[ m+\bar{C}
\Big(|u_{j_{h}-p-1}^{n_2-p-1}|^{2}|v_{j_{h}-p+1}^{n_2-p-1}|^{2}
+m_{1}(|u_{j_{h}-p-1}^{n_2-p-1}|^{2}+|v_{j_{h}-p+1}^{n_2-p-1}|^{2})^{2}\tau\Big)^\frac{1}{2}\right] \\
& \cdot\Big(|v_{j_{h}-p+1}^{n_2-p-1}|^{2}+m_{1}(|u_{j_{h}-p-1}^{n_2-p-1}|^{2}+|v_{j_{h}-p+1}^{n_2-p-1}|^{2})\tau+c_{1}|u_{j_{h}-p-1}^{n_2-p-1}|^{2}|v_{j_{h}-p+1}^{n_2-p-1}|^{2}\tau\Big)^{\frac{1}{2}}\tau\\
\leq &
m\sum_{p=0}^{n_2-n_0-1} \Big(|v_{j_{h}-p+1}^{n_2-p-1}|^{2}+m_{1}(|u_{j_{h}-p-1}^{n_2-p-1}|^{2}+|v_{j_{h}-p+1}^{n_2-p-1}|^{2})\tau+c_{1}|u_{j_{h}-p-1}^{n_2-p-1}|^{2}|v_{j_{h}-p+1}^{n_2-p-1}|^{2}\tau\Big)^{\frac{1}{2}}\tau\\
&+
\bar{C}\sum_{p=0}^{n_2-n_0-1}
\Big(|u_{j_{h}-p-1}^{n_2-p-1}|^{2}|v_{j_{h}-p+1}^{n_2-p-1}|^{2}
+m_{1}(|u_{j_{h}-p-1}^{n_2-p-1}|^{2}+|v_{j_{h}-p+1}^{n_2-p-1}|^{2})^{2}\tau\Big)^\frac{1}{2}\\
& \cdot\Big(|v_{j_{h}-p+1}^{n_2-p-1}|^{2}+m_{1}(|u_{j_{h}-p-1}^{n_2-p-1}|^{2}+|v_{j_{h}-p+1}^{n_2-p-1}|^{2})\tau+c_{1}|u_{j_{h}-p-1}^{n_2-p-1}|^{2}|v_{j_{h}-p+1}^{n_2-p-1}|^{2}\tau\Big)^{\frac{1}{2}}\tau.
\end{align*}
By Young's inequality, we have

\begin{align*}
&|u_{j_{h}}^{n_2}-u_{j}^{n_0}|^2\\
\leq &
2m^2\tau^2
\Bigg(\sum_{p=0}^{n_2-n_0-1} \Big(|v_{j_{h}-p+1}^{n_2-p-1}|^{2}+m_{1}(|u_{j_{h}-p-1}^{n_2-p-1}|^{2}+|v_{j_{h}-p+1}^{n_2-p-1}|^{2})\tau+c_{1}|u_{j_{h}-p-1}^{n_2-p-1}|^{2}|v_{j_{h}-p+1}^{n_2-p-1}|^{2}\tau\Big)^{\frac{1}{2}}\Bigg)^2\\
&+2\bar{C}^2\tau^2
\Bigg(  \sum_{p=0}^{n_2-n_0-1}
\Big(|u_{j_{h}-p-1}^{n_2-p-1}|^{2}|v_{j_{h}-p+1}^{n_2-p-1}|^{2}
+m_{1}(|u_{j_{h}-p-1}^{n_2-p-1}|^{2}+|v_{j_{h}-p+1}^{n_2-p-1}|^{2})^{2}\tau\Big)^\frac{1}{2} \nonumber\\
& \qquad\ \  \cdot\Big(|v_{j_{h}-p+1}^{n_2-p-1}|^{2}+m_{1}(|u_{j_{h}-p-1}^{n_2-p-1}|^{2}+|v_{j_{h}-p+1}^{n_2-p-1}|^{2})\tau+c_{1}|u_{j_{h}-p-1}^{n_2-p-1}|^{2}|v_{j_{h}-p+1}^{n_2-p-1}|^{2}\tau\Big)^{\frac{1}{2}}\Bigg) ^2.
\end{align*}
Then
	\begin{align}\label{eq:sum-discrete-continuity}
		&|u_{j_{h}}^{n_2}-u_{j}^{n_0}|^2\nonumber\\
		\leq &
		2m^2\tau^2(n_2-n_0)
		\hspace*{-0.15cm}\sum_{p=0}^{n_2-n_0-1}\hspace*{-0.15cm}\Big(|v_{j_{h}-p+1}^{n_2-p-1}|^{2}+m_{1}(|u_{j_{h}-p-1}^{n_2-p-1}|^{2}+|v_{j_{h}-p+1}^{n_2-p-1}|^{2})\tau+c_{1}|u_{j_{h}-p-1}^{n_2-p-1}|^{2}|v_{j_{h}-p+1}^{n_2-p-1}|^{2}\tau\Big)\nonumber\\
		&+2\bar{C}^2\tau^2
		\sum_{p=0}^{n_2-n_0-1}
		\Big(|u_{j_{h}-p-1}^{n_2-p-1}|^{2}|v_{j_{h}-p+1}^{n_2-p-1}|^{2}
		+m_{1}(|u_{j_{h}-p-1}^{n_2-p-1}|^{2}+|v_{j_{h}-p+1}^{n_2-p-1}|^{2})^{2}\tau\Big)\nonumber\\
		& \cdot\sum_{p=0}^{n_2-n_0-1}\Big(|v_{j_{h}-p+1}^{n_2-p-1}|^{2}+m_{1}(|u_{j_{h}-p-1}^{n_2-p-1}|^{2}+|v_{j_{h}-p+1}^{n_2-p-1}|^{2})\tau
		+c_{1}|u_{j_{h}-p-1}^{n_2-p-1}|^{2}|v_{j_{h}-p+1}^{n_2-p-1}|^{2}\tau\Big).
\end{align}

To estimate $\sum_{j=-\infty}^{+\infty}|u_{j_{h}}^{n_2}-u_{j}^{n_{0}}|^{2}\tau$ with $j_{h}=j+n_2-n_0$,
we next estimate the right-hand side of
\eqref{eq:sum-discrete-continuity} one by one.

First, by \eqref{eq:l2-conserve-i}, it follows that
\begin{align*}
		&\sum_{j=-\infty}^{+\infty}\sum_{p=0}^{n_2-n_{0}-1}
		\left( |u_{j_{h}-p-1}^{n_2-p-1}|^{2}
		+|v_{j_{h}-p+1}^{n_2-p-1}|^{2}\right)\tau^{2}
		=\sum_{p=0}^{n_2-n_{0}-1}\sum_{j=-\infty}^{+\infty}
		\left(|u_{j}^{n_2-p-1}|^{2}+|v_{j}^{n_2-p-1}|^{2}\right)\tau^{2}\\
		\leq&
		\sum_{p=0}^{n_2-n_{0}-1}
		\Big(\sum_{j=-\infty}^{+\infty}
		(|u_{j}^{0}|^{2}+|v_{j}^{0}|^{2})\tau\Big)\tau
		\leq 2c_0 (n_2-n_{0})\tau.
\end{align*}
Since $0\leq n_0\leq n_2$, $n_2\tau\leq T+1$, Lemma \ref{lem:nonlinear-estimates} gives
\begin{equation*}
  \sum_{j=-\infty}^{+\infty}\sum_{p=0}^{n_{2}-n_{0}-1}
   |u_{j_{h}-p-1}^{n_2-p-1}|^{2}|v_{j_{h}-p+1}^{n_2-p-1}|^{2}\tau^2
   =\sum_{p=0}^{n_2-n_{0}-1}\sum_{j=-\infty}^{+\infty}
   |u_{j}^{n_2-p-1}|^{2}|v_{j}^{n_2-p-1}|^{2}\tau^2
    \leq c(T).
\end{equation*}

Moreover, using \eqref{2.5} and \eqref{2.8}--\eqref{2.9} yields
\begin{align*}
&\sum_{p=0}^{n_2-n_{0}-1}|v_{j_{h}-p+1}^{n_2-p-1}|^{2}\tau
=
\sum_{k=1}^{n_2-n_{0}}
|v_{j_{h}+2-k}^{n_{2}-k}|^{2}\tau
\leq
\sum_{l=j_{h}+2-n_2+n_{0}}^{j_{h}+2+n_2-n_{0}}
\big(|u_{l}^{n_{0}}|^{2}+|v_{l}^{n_{0}}|^{2}\big)\tau\\
&\leq
\sum_{l=j_{h}+2-n_2+n_{0}}^{j_{h}+2+n_2-n_{0}}
c_{2}(T)\Big(\big(|u_{l-n_{0}}^{0}|^{2}+|v_{l+n_{0}}^{0}|^{2}\big)+2m_{1}c_0\Big)\tau\\
&\leq
c_{2}(T)\int_{(j_{h}+2-n_2+n_{0})\tau}^{(j_{h}+2+n_2-n_{0})\tau}
\big(|u^{(\tau)}(x-n_0\tau,0)|^2+|v^{(\tau)}(x+n_0\tau,0)|^2\big)\mathrm{d}x+4m_{1}c_0c_{2}(T)(n_2-n_{0})\tau.
\end{align*}

For the remaining terms in \eqref{eq:sum-discrete-continuity}, by \eqref{eq:l2-conserve-i} and \eqref{eq:c_0}, we deduce that for $p=0,1,\ldots,n_2-n_{0}$,
\begin{equation*}
		\big(|u_{j_{h}-p-1}^{n_2-p-1}|^{2}
		+|v_{j_{h}-p+1}^{n_2-p-1}|^{2}\big)\tau
		\leq
		\sum_{j=-\infty}^{\infty}
		\big(|u_{j_{h}-p-1}^{n_2-p-1}|^{2}
		+|v_{j_{h}-p+1}^{n_2-p-1}|^{2}\big)\tau
		\leq c_0.
\end{equation*}
Then, a direct calculation yields
\begin{align*}
\sum_{p=0}^{n_2-n_0-1}
\big(|u_{j_{h}-p-1}^{n_2-p-1}|^{2}
+|v_{j_{h}-p+1}^{n_2-p-1}|^{2}\big)^2\tau^3
\leq
c_0\sum_{p=0}^{n_2-n_{0}-1}
\big(|u_{j_{h}-p-1}^{n_2-p-1}|^{2}
+|v_{j_{h}-p+1}^{n_2-p-1}|^{2}\big)\tau^2
\end{align*}
and

\begin{equation*}
   \sum_{p=0}^{n_2-n_0-1}
   |u_{j_{h}-p-1}^{n_2-p-1}|^{2}|v_{j_{h}-p+1}^{n_2-p-1}|^{2}\tau^2
\leq c_0\sum_{p=0}^{n_2-n_0-1}|v_{j_{h}-p+1}^{n_2-p-1}|^{2}\tau.
\end{equation*}

Since $(u_{0},v_{0})\in \mathrm{L}^{2}(\mathbb{R})$, for any $\varepsilon\in(0,1)$, it follows from \eqref{eq:initial-data-convergent} that there is $0<\delta_1<\varepsilon/2$, such that, for all $h=(n_2-n_0)\tau<2\delta_1$,
\begin{equation*}
\int_{(j_{h}+2-n_2+n_{0})\tau}^{(j_{h}+2+n_2-n_{0})\tau}
	\big(|u^{(\tau)}(x-n_0\tau,0)|^2+|v^{(\tau)}(x+n_0\tau,0)|^2\big)\mathrm{d}x\leq \varepsilon,
\end{equation*}
and
\begin{equation*}
(n_2-n_0)\tau <\varepsilon.
\end{equation*}

Finally, by using the above estimates and \eqref{eq:sum-discrete-continuity}, we have
\begin{align}\label{lemma4.5-1}
&\sum_{j=-\infty}^{\infty}|u_{j+n_2-n_0}^{n_2}-u_{j}^{n_{0}}|^{2}\tau\nonumber\\
\leq &
2m^2\varepsilon\big(2c_0\varepsilon+2m_1\varepsilon\tau+c_1c(T)\tau\big)
+2\bar{C}^2\big(c(T)+m_12c_0\varepsilon c_0\big)
\big(c_2(T)\varepsilon+4m_1c_0c_2(T)\varepsilon+m_12c_0\varepsilon
	\big)
\nonumber\\
&+2\bar{C}^2\big(c(T)+m_12c_0\varepsilon c_0\big)
c_14m_{1}c_0^2c_{2}(T)\varepsilon
\leq
C(T)\varepsilon,
\end{align}
where
\begin{equation*} C(T)=2m^2\big(2c_0+2m_1+c_1c(T)\big)+2\bar{C}^2\big(c(T)+2m_2c_0^2\big)
\big(
c_2(T)(1+c_1c_0+2m_1c_0(3+2c_0c_1))
\big).\end{equation*}

Similarly, using the second equation in \eqref{eq:subproblem2}, we derive that
\begin{equation}\label{lemma4.5-2}
\sum_{j=-\infty}^{+\infty}|v_{j-n_2+n_{0}}^{n_2}-v_{j}^{n_{0}}|^{2}\tau
\leq C(T)\varepsilon.
\end{equation}

\textit{Step 2. Estimates on $(u^{(\tau)},v^{(\tau)})$ for $(x,t_0)\in\mathbb{R}\times[0,T]$ and sufficiently small $h>0$.}
In this case, if $\tau\in(0,\delta_1)$ and $h\in(0,\delta_1)$, then it follows from \eqref{eq:characteristic-estimate-uv} that
\begin{align*}
&\int_{-\infty}^{\infty}
\big|u^{(\tau)}(x+h,t_{0}+h)-u^{(\tau)}(x,t_{0})\big|^{2}\mathrm{d}x\\
=&
\int_{-\infty}^{\infty}
\Big|
u^{(\tau)}\Big( x+h-
\big(t_0+h-\Big[\frac{t_0+h}{\tau}\Big]\tau \big),
\Big[\frac{t_0+h}{\tau}\Big]\tau\Big)
-
u^{(\tau)}\Big( x-
\big( t_0-\Big[\frac{t_0}{\tau}\Big]\tau\big) ,
\Big[\frac{t_0}{\tau}\Big]\tau\Big)
\Big|^{2}
\mathrm{d}x\\
\leq&
\sum_{j=-\infty}^{\infty}
\int_{j\tau}^{(j+1)\tau}\Big|u^{(\tau)}\Big(x,\Big[\frac{t_0+h}{\tau}\Big]\tau \Big)
-u^{(\tau)}\Big(x-\Big[\frac{t_0+h}{\tau}\Big]\tau+\Big[\frac{t_0}{\tau}\Big]\tau  ,\Big[\frac{t_0}{\tau}\Big]\tau \Big)
\Big|^{2}\mathrm{d}x\\
\leq&
\sum_{j=-\infty}^{\infty}
\big| u_{j}^{\big[\frac{t_0+h}{\tau}\big]}-u_{j-\big[\frac{t_0+h}{\tau}\big]+\big[\frac{t_0}{\tau}\big]}^{\big[\frac{t_0}{\tau}\big]}\big|^2\tau\leq C(T)\varepsilon,
\end{align*}
where the last inequality is a consequence
of \eqref{lemma4.5-1}.

The remaining estimate in Lemma~\ref{5.6.6} for $v^{(\tau)}$ is obtained in the same way as that for $u^{(\tau)}$, by using \eqref{lemma4.5-2} in place of \eqref{lemma4.5-1}. This completes the proof.
\end{proof}

In view of Corollary \ref{eq:coro-strip} and Lemma \ref{5.6.6}, the following lemma shows the compactness of the set of time-splitting solutions
\begin{equation*}
S_T= \Big\{\big(u^{(\tau)},v^{(\tau)}\big)\big|_{\mathbb{R}\times[0,T]}\:: \:\tau\in(0,\tau_{*})\Big\}
\end{equation*}
defined as in \eqref{eq:ST}.

\begin{lemma}\label{lem:proposition4.1}
Assume that the initial data $(u_0,v_0)\in\mathrm{L}^2(\mathbb{R})$.
Let $\{\tau_l\}_{l=1}^{+\infty}$ be any sequence of mesh sizes such that $\tau_l \to 0$ as $l \to +\infty$.
Then, for any $T>0$, there exists a constant $0<\tau_{\sharp}<1$ such that if $\{\tau_l\}_{l=1}^{+\infty}\subset(0,\tau_{\sharp})$, then the set of time-splitting solutions
$\{(u^{(\tau_l)},v^{(\tau_l)})\}_{l=1}^{+\infty}
$
is relatively compact in $\mathrm{C}([0,T];\mathrm{L}^{2}(\mathbb{R}))$.
\end{lemma}
\begin{proof}
Let $\varepsilon>0$ be given. By \eqref{eq:former-case-2} and \eqref{eq:interval-initial-2}, there exist constants $C_*(T)>0$ and $R>0$, such that,
for all $\tau\in(0,\tau_*)$, it holds that
\begin{equation*}
	\sup_{0\leq t \leq T}
	\left(
	\|u^{(\tau)}(\cdot,t)\|_{L^{2}(\mathbb{R})}^{2}
	+
	\|v^{(\tau)}(\cdot,t)\|_{L^{2}(\mathbb{R})}^{2}
	\right)
	\le C_*(T),
\end{equation*}
and
\begin{equation*}
	\sup_{0\leq t \leq T}
	\int_{|x|>R}
	\left(
	|u^{(\tau)}(x,t)|^2
	+
	|v^{(\tau)}(x,t)|^2
	\right)\mathrm{d}x
	< \varepsilon.
\end{equation*}
Moreover, by Corollary \ref{eq:coro-strip}, we know that for all $\tau\in(0,\tau_{*})$ and $\mu\in(0,\tau_{*})$,
\begin{equation}\label{eq:space}
\sup_{0\leq t\leq T}\int_{\mathbb{R}}
	\Big(
	|u^{(\tau)}(x+\mu,t)-u^{(\tau)}(x,t)|^{2}
	+
	|v^{(\tau)}(x+\mu,t)-v^{(\tau)}(x,t)|^{2}
	\Big)\,\mathrm{d}x
	\leq \varepsilon.
\end{equation}
Hence, by the Kolmogorov--Riesz--Fr\'echet Theorem; see \cite[Theorem 4.26 and  Corollary 4.27]{BrezisBook}, $S_T(t)=\big\{\big(u^{(\tau)}(\cdot,t),v^{(\tau)}(\cdot,t)\big): \big(u^{(\tau)},v^{(\tau)}\big)\in S_T\big\}$ is relatively compact in $L^{2}(\mathbb{R})$.

It remains to prove that $S_T$ is uniformly equicontinuous. To this end, we estimate
\begin{equation*}
	\|\big(u^{(\tau)}(\cdot,t_2),v^{(\tau)}(\cdot,t_2)\big)-\big(u^{(\tau)}(\cdot,t_1),v^{(\tau)}(\cdot,t_1)\big)\|_{L^{2}(\mathbb{R})}
\end{equation*}
for all $0<t_1\le t_2<T$  with $t_2-t_1= h>0$.

In view of Lemma~\ref{5.6.6}, there exist constants $\delta_1>0$ and $C(T)>0$ such that, for any $(u^{(\tau)},v^{(\tau)})\in S_T$, if $\tau\in(0,\delta_1)$, $h\in(0,\delta_1)$ and $t\in[0,T]$, then
		\begin{equation*}
		\int_{\mathbb R}
		\left( |u^{(\tau)}(x+h,t+h)-u^{(\tau)}(x,t)|^2
		+|v^{(\tau)}(x-h,t+h)-v^{(\tau)}(x,t)|^2\right)
		\mathrm{d}x
		\le C(T)\varepsilon.
	    \end{equation*}
Using the above estimates and taking $\mu=\pm h$ in \eqref{eq:space}, we deduce that there exists a constant $\tau_\sharp\in(0,\min\{\tau_{*},\delta_1\})$ such that, for all $\tau\in(0,\tau_\sharp)$ and $h=t_2-t_1\in(0,\tau_\sharp)$,
\begin{align*}
	&\int_{\mathbb R}|u^{(\tau)}(x,t_2)-u^{(\tau)}(x,t_1)|^2\mathrm{d}x
	+\int_{\mathbb R}|v^{(\tau)}(x,t_2)-v^{(\tau)}(x,t_1)|^2\mathrm{d}x\\
	\leq
	&2\int_{\mathbb R}|u^{(\tau)}(x,t_1+h)-u^{(\tau)}(x+h,t_1+h)|^2\mathrm{d}x
	+
	2\int_{\mathbb R}|u^{(\tau)}(x+h,t_1+h)-u^{(\tau)}(x,t_1)|^2\mathrm{d}x\\
	&+
	2\int_{\mathbb R}|v^{(\tau)}(x,t_1+h)-v^{(\tau)}(x-h,t_1+h)|^2\mathrm{d}x
	+
	2\int_{\mathbb R}|v^{(\tau)}(x-h,t_1+h)-v^{(\tau)}(x,t_1)|^2\mathrm{d}x\\
	\leq& 4\varepsilon+4C(T)\varepsilon.
\end{align*}

Therefore, we show that the set of time-splitting solutions $\{(u^{(\tau_l)},v^{(\tau_l)})\}_{l=1}^{+\infty}$ with mesh sizes $\{\tau_l\}_{l=1}^{+\infty}\subset(0,\tau_{\sharp})$ satisfies the conditions of the Lemma \ref{lem:A2} and hence is relatively compact in $\mathrm{C}([0,T];\mathrm{L}^{2}(\mathbb{R}))$.
\end{proof}
\section{Uniqueness of the limit of time-splitting solutions}
\label{sec:strong}
In this section, we prove the main theorem of the paper. For any initial data $(u_{0},v_{0}) \in \mathrm{L}^{2}(\mathbb{R})$, the NLDE Cauchy problem \eqref{eq:NLDE}--\eqref{eq:NLDE-initail-data} admits a global strong solution, see Lemma \ref{lem:A3}.
By Lemma~\ref{lem:proposition4.1}, we can extract a convergent subsequence from any sequence of time-splitting solutions.  We then show that, as the mesh size tends to zero, any convergent subsequence converges strongly in
$\mathrm{L}^2(\mathbb{R}\times[0,T])$ to the strong solution of
\eqref{eq:NLDE}--\eqref{eq:NLDE-initail-data}, and the limit is unique.

To begin with, let $(\hat{u},\hat{v})\in\mathrm{C}([0,\infty);\mathrm{L}^{2}(\mathbb{R}))$ be the strong solution of the NLDE Cauchy problem \eqref{eq:NLDE}--\eqref{eq:NLDE-initail-data} with initial data $(u_{0},v_{0}) \in \mathrm{L}^{2}(\mathbb{R})$, in the sense of Definition \ref{definition1.1}. More precisely, we can choose a sequence of smooth functions
\begin{equation*}
	(\hat{u}_{\mathrm{k}},\hat{v}_{\mathrm{k}})(x,0)\in \mathrm{C}_{\mathrm{c}}^\infty(\mathbb{R}),\quad \mathrm{k}=1,2,\ldots
	\end{equation*}
such that
\begin{equation}\label{5.1}
\lim_{\mathrm{k}\to \infty}\Big(
\|\hat{u}_{\mathrm{k}}(x,0)-u_{0}(x)\|_{\mathrm{L}^{2}(\mathbb{R})} +
\|\hat{v}_{\mathrm{k}}(x,0)-v_{0}(x)\|_{\mathrm{L}^{2}(\mathbb{R})}\Big)=0.
\end{equation}
Then for any $T>0$, \cite[Lemma 3.1]{Y.ZhangandQ.Zhao1} ensures the existence of a sequence of smooth solutions
\begin{equation*}
(\hat{u}_{\mathrm{k}},\hat{v}_{\mathrm{k}})\in \mathrm{C}_c^\infty(\mathbb{R}\times [0,T])
\end{equation*}
to \eqref{eq:NLDE} with initial data $(\hat{u}_{\mathrm{k}}(x,0),\hat{v}_{\mathrm{k}}(x,0))$, such that
\begin{equation}\label{eq:uniqueness-s}
\lim_{\mathrm{k}\to \infty}\Big(
\|\hat{u}_{\mathrm{k}}-\hat{u}\|_{\mathrm{L}^2(\mathbb{R}\times[0,T])}
+
\|\hat{v}_{\mathrm{k}}-\hat{v}\|_{\mathrm{L}^2(\mathbb{R}\times[0,T])}\Big)=0.
\end{equation}

Next, let $\{\tau_l\}_{l=1}^{+\infty} \subset (0,\tau_{\sharp}) \subset (0,1)$ be a sequence of mesh sizes such that $\tau_l \to 0$ as $l \to +\infty$. Given initial data $(u_{0},v_{0}) \in \mathrm{L}^{2}(\mathbb{R})$,  we obtain a sequence of time-splitting solutions $\{(u^{(\tau_l)},v^{(\tau_l)})\}_{l=1}^{+\infty}$ corresponding to the sequence of mesh sizes  $\{\tau_l\}_{l=1}^{+\infty}$.
Thus, by Lemma \ref{lem:proposition4.1}, there is a convergent subsequence of any sequence of time-splitting solutions $\{(u^{(\tau_l)},v^{(\tau_l)})\}_{l=1}^{+\infty}$, still denote by $\{(u^{(\tau_l)},v^{(\tau_l)})\}_{l=1}^{+\infty}$, such that
\begin{equation}\label{eq:limit-solution}
	\lim\limits_{l\to +\infty}
	\left(
	\|u^{(\tau_l)}-u_\flat\|_{\mathrm{C}([0,T];\mathrm{L}^2(\mathbb{R}))}
	+
	\|v^{(\tau_l)}-v_\flat\|_{\mathrm{C}([0,T];\mathrm{L}^2(\mathbb{R}))}
	\right)
	= 0.
\end{equation}
We remark that \eqref{eq:limit-solution} implies that
\begin{equation*}
	\lim\limits_{l\to +\infty}
	\left(
	\|u^{(\tau_l)}-u_\flat\|_{\mathrm{L}^2(\mathbb{R}\times[0,T])}
	+
	\|v^{(\tau_l)}-v_\flat\|_{\mathrm{L}^2(\mathbb{R}\times[0,T])}
	\right)
	= 0.
\end{equation*}

Hence, to prove Theorem \ref{mainresult}, it suffices to show that, for any $T>0$,
\begin{equation}\label{eq:unique-goal}
	(u_\flat,v_\flat)(x,t)=(\hat{u},\hat{v})(x,t),\qquad \text{ in }\:  \mathrm{L}^2(\mathbb{R}\times[0,T]).
\end{equation}

For simplicity, we drop the subscripts $k$ and $l$ whenever no confusion arises from the context, namely, we first consider the case where $(\hat{u},\hat{v})$ is the smooth solution with initial data $(\hat{u}(x,0),\hat{v}(x,0))\in \mathrm{C}_{\mathrm{c}}^\infty(\mathbb{R})$ and $(u^{(\tau)},v^{(\tau)})$ is the time-splitting solution. Moreover, we set
\begin{equation}\label{eq:mathcal-N}
	\mathcal{N}_{1}(u,v)
=\alpha u|v|^{2}+2\beta(\overline{u}v+u\overline{v})v,\quad  \mathcal{N}_{2}(u,v)
=\alpha v|u|^{2}+2\beta(\overline{u}v+u\overline{v})u.
\end{equation}

To prove \eqref{eq:unique-goal}, we now consider the difference between the smooth solution $(\hat{u},\hat{v})$ and the time-splitting solution $(u^{(\tau)},v^{(\tau)})$ in a characteristic triangle $\Lambda((j+3/2)\tau,(n+3/2)\tau;n\tau)$ for $j=0,\pm1,\pm2,\ldots,n=0,1,2,\ldots$ with $n\leq (T+1)/\tau$, as shown in
Fig.~\ref{fig1-4}.
\begin{figure}[h]
	\centering
	\begin{tikzpicture}[scale=0.75]
		\draw[densely dashed] (-2,0) -- (12,0) node[right] {\small$t=n \tau$};
		\draw[densely dashed] (-2,2) -- (12,2) node[right] {\small$t=(n+1)\tau$};
			\draw[densely dashed] (-2,3) -- (12,3) node[right] {\small$t=(n+3/2)\tau$};
		\draw[densely dashed] (10,0)--(10,2);
		\draw[densely dashed] (8,0)--(8,2);
		\draw[densely dashed] (6,0)--(6,2);
		\draw[densely dashed] (4,0)--(4,2);
		\draw[densely dashed] (2,0)--(2,2);
		\draw[densely dashed] (0,0)--(0,2);
		\draw[thin](4,0)--(6,2)--(8,0)--(4,0);
		\draw[thin] (2,0) -- (4,2) -- (6,0) -- cycle;
		\draw[thick] (4,2)--(5,3)--(6,2);
		\node at (5,3.3) {$((j+3/2)\tau,(n+3/2)\tau)$};
		\draw[thick] (2,0) -- (4,2) -- (6,2) --(8,0)-- cycle;;
		\fill (2,0) circle (1pt);
		\fill (4,2) circle (1pt);
		\fill (6,0) circle (1pt);
		\fill (8,0) circle (1pt);
		\node[below] at (2,0) {\small$j\tau$};
		\node[below] at (6,0) {\small$(j+2)\tau$};
		\node[below] at (4,0) {\small$(j+1)\tau$};
		\node[below] at (8,0) {\small$(j+3)\tau$};
	\end{tikzpicture}
	\caption{\label{fig1-4}  The characteristic triangle $\Lambda((j+3/2)\tau,(n+3/2)\tau;n\tau)$}
\end{figure}

Then, for the smooth solution $(\hat{u},\hat{v})$, by the method of characteristics, we deduce from \eqref{eq:NLDE} that
\begin{equation}\label{5.2}
	\left\{
	\begin{aligned}
		&\frac{\mathrm{d}}{\mathrm{d}s}\hat{u}(j\tau+s,n\tau+s)
		=\big(im\hat{v}+i\mathcal{N}_{1}(\hat{u},\hat{v})\big)(j\tau+s,n\tau+s),\\
		&\frac{\mathrm{d}}{\mathrm{d}s}\hat{v}(j\tau+2\tau-s,n\tau+s)
		=\big(im\hat{u}+i\mathcal{N}_{2}(\hat{u},\hat{v})\big)(j\tau+2\tau-s,n\tau+s).
	\end{aligned}\right.
\end{equation}

For the time-splitting solution $(u^{(\tau)},v^{(\tau)})$, the linear subproblem \eqref{eq:subproblem1}--\eqref{eq:subproblem1-initial-data} on each time interval $t\in[n\tau,(n+1)\tau)$ takes the following equivalent form:
\begin{equation}\label{5.5}
	\left\{
	\begin{aligned}
		&\frac{\mathrm{d}}{\mathrm{d}s}{u}^{(\tau)}(j\tau+s,n\tau+s)=0,\\
		&\frac{\mathrm{d}}{\mathrm{d}s}{v}^{(\tau)}(j\tau+2\tau-s,n\tau+s)=0,
	\end{aligned}\right.
\end{equation}
with $s \in [0,\tau)$ and
\begin{equation*}
(u^{(\tau)},v^{(\tau)})(x,n\tau)=(u^{(\tau)},v^{(\tau)})(j\tau,n\tau)=(u_{j}^n,v_{j}^n),\quad x\in[j\tau,(j+1)\tau).
\end{equation*}
Moreover, the nonlinear equation \eqref{eq:subproblem2}--\eqref{eq:subproblem2-initial-data} on $x\in[(j+1)\tau,(j+2)\tau)$ and $s\in [0,\tau]$ can be equivalently written as:
\begin{equation}\label{5.3}
	\left\{
	\begin{aligned}
		&\frac{\mathrm{d}}{\mathrm{d}s}
		u_{j+1}^{n,2}(s)
		=imv_{j+1}^{n,2}(s)+i\mathcal{N}_{1}(u_{j+1}^{n,2}(s),v_{j+1}^{n,2}(s)),\\
		&\frac{\mathrm{d}}{\mathrm{d}s}
		v_{j+1}^{n,2}(s)
		=imu_{j+1}^{n,2}(s)+i\mathcal{N}_{2}(u_{j+1}^{n,2}(s),v_{j+1}^{n,2}(s)),\\
		&(u_{j+1}^{n,2}(s),v_{j+1}^{n,2}(s))\big|_{s=0}
		=(u_{j+1}^{n+1-},v_{j+1}^{n+1-})=(u_j^n,v_{j+2}^n),
	\end{aligned}\right.
\end{equation}
with
\begin{equation*}
(u_{j+1}^{n,2}(s),v_{j+1}^{n,2}(s))\big|_{s=\tau}
=(u_{j+1}^{n+1},v_{j+1}^{n+1}).
\end{equation*}
\subsection{Pointwise estimates on difference}
For simplicity, we denote the difference between the smooth solution $(\hat{u},\hat{v})$ and the time-splitting solution $(u^{(\tau)},v^{(\tau)})$ by
\begin{equation*}
(\mathcal{U},\mathcal{V})
:=(\hat{u}-u^{(\tau)},\hat{v}-v^{(\tau)}).
\end{equation*}
In this subsection, we derive pointwise estimates for $(\mathcal{U},\mathcal{V})$ at $t=(n+1)\tau$ on an arbitrary interval $x\in[(j+1)\tau,(j+2)\tau)$ in the characteristic triangle $\Lambda((j+3/2)\tau,(n+3/2)\tau;n\tau)$, see Fig.~\ref{fig1-4}.
Observe that if $t=n\tau$, then the time-splitting solution is constant on $[j\tau,(j+1)\tau)$. Accordingly, we decompose
	\begin{equation*}
		(\hat{u},\hat{v})(x,n\tau)
		=(\hat{u},\hat{v})(x,n\tau)
		-(\hat{u},\hat{v})(j\tau,n\tau)
		+(\hat{u},\hat{v})(j\tau,n\tau),
		\quad x\in[j\tau,(j+1)\tau).
\end{equation*}

Since $(\hat{u},\hat{v})\in \mathrm{C}_{\mathrm{c}}^\infty(\mathbb{R}\times [0,T])$, there exist constants $\mathcal{M}_0(T)>0$ and $\mathcal{M}(T)>0$, such that
\begin{align*}
	&\mathcal{M}_0=\mathcal{M}_0(T)=
	\max_{\mathbb{R}\times[0,T]}(|\hat{u}|+|\hat{v}|+1),\\
&\mathcal{M}=\mathcal{M}(T)=\max_{\mathbb{R}\times[0,T]}(|\partial_t\hat{u}|+|\partial_x\hat{u}|
	+|\partial_t\hat{v}|+|\partial_x\hat{v}|+1).
	\end{align*}
Therefore, by the mean value theorem, for every $x\in[j\tau,(j+1)\tau)$,
	\begin{equation*}
		|(\hat{u},\hat{v})(x,n\tau)
		-(\hat{u},\hat{v})(j\tau,n\tau)|\leq |x-j\tau|\max_{\mathbb{R}\times[0,T]}(|\partial_x\hat{u}|+|\partial_x\hat{v}|)	\leq \mathcal{M}\tau.
	\end{equation*}
Let $(\hat{u},\hat{v})(j\tau,n\tau)=(\hat{u}_{j}^n,\hat{v}_{j}^n)$ Let $(\hat{u},\hat{v})(j\tau,n\tau)
=(\hat{u}_{j}^{n},\hat{v}_{j}^{n})$
for $j\in\mathbb Z$ and $n=0,1,2,\ldots$. Then
	\begin{align}
		&|\mathcal{U}(x,n\tau)|
		=|\hat{u}(x,n\tau)-u^{(\tau)}(x,n\tau)|
		\leq \mathcal{M}\tau+|\hat{u}_{j}^{n}-u_{j}^{n}|,\label{eq:error1}\\
		&|\mathcal{V}(x,n\tau)|
		=|\hat{v}(x,n\tau)-v^{(\tau)}(x,n\tau)|
		\leq
		\mathcal{M}\tau+|\hat{v}_{j}^{n}-v_{j}^{n}|.\label{eq:error2}
	\end{align}
	
To obtain the pointwise estimates for $(\mathcal{U},\mathcal{V})$ at $t=(n+1)\tau$ on the interval $x\in[(j+1)\tau,(j+2)\tau)$, applying \eqref{eq:error1} and \eqref{eq:error2} with $(j,n)$ replaced by $(j+1,n+1)$, we obtain
	\begin{equation*}
		|\mathcal{U}(x,(n+1)\tau)|
		\leq \mathcal{M}\tau
		+|\hat{u}_{j+1}^{n+1}-u_{j+1}^{n+1}|
	\end{equation*}
	and
	\begin{equation*}
		|\mathcal{V}(x,(n+1)\tau)|
		\leq \mathcal{M}\tau+|\hat{v}_{j+1}^{n+1}-v_{j+1}^{n+1}|.
	\end{equation*}

Next, we consider the difference
	\begin{equation*}
		(\hat{u}_{j+1}^{n+1}-u_{j+1}^{n+1},\hat{v}_{j+1}^{n+1}-v_{j+1}^{n+1})=((\hat{u}-u^{(\tau)})((j+1)\tau,(n+1)\tau),(\hat{v}-v^{(\tau)})((j+1)\tau,(n+1)\tau)).
	\end{equation*}
Integrating \eqref{5.3} over $[0,\tau]$ with respect to $s$, we have
\begin{equation}
	\left\{
	\begin{aligned}
		&u_{j+1}^{n+1}-u_j^n
		=\int_{0}^{\tau}\Big(imv_{j+1}^{n,2}(s)+i\mathcal{N}_{1}(u_{j+1}^{n,2}(s),v_{j+1}^{n,2}(s))\Big)\mathrm{d}s,\\
		&v_{j+1}^{n+1}-v_{j+2}^n
		=\int_{0}^{\tau}\Big(imu_{j+1}^{n,2}(s)+i\mathcal{N}_{2}(u_{j+1}^{n,2}(s),v_{j+1}^{n,2}(s))\Big)\mathrm{d}s,
	\end{aligned}\right.
\end{equation}
Moreover, integrating \eqref{5.2} over $[0,\tau]$ with respect to $s$, we have
\begin{equation}\label{eq:int-smooth}
	\left\{
	\begin{aligned}
		&\hat{u}_{j+1}^{n+1}
		-\hat{u}_{j}^{n}
		=\int_{0}^{\tau}\big(im\hat{v}+i\mathcal{N}_{1}(\hat{u},\hat{v})\big)(j\tau+s,n\tau+s)\mathrm{d}s,\\
		&\hat{v}_{j+1}^{n+1}
		-\hat{v}_{j+2}^{n}
		=\int_{0}^{\tau}\big(im\hat{u}+i\mathcal{N}_{2}(\hat{u},\hat{v})\big)(j\tau+2\tau-s,n\tau+s)\mathrm{d}s,.
	\end{aligned}\right.
\end{equation}

To estimate the difference $(\hat{u}_{j+1}^{n+1}-u_{j+1}^{n+1},\hat{v}_{j+1}^{n+1}-v_{j+1}^{n+1})$, we introduce an auxiliary pair of functions
\begin{equation*}
	(\hat{u}_{n,2}((j+1)\tau,s),\hat{v}_{n,2}((j+1)\tau,s)),\quad s\in[0,\tau],
\end{equation*}
which is defined by the following system
\begin{equation}\label{eq:4.17}
	\left\{
	\begin{aligned}
		&\frac{\mathrm{d}}{\mathrm{d}s}
		\hat{u}_{n,2}((j+1)\tau,s)
		=\big(im\hat{v}_{n,2}+i\mathcal{N}_{1}(\hat{u}_{n,2}(s),\hat{v}_{n,2})\big)((j+1)\tau,s)+\mathfrak{g}_{j}^{n}(s),\\
		&\frac{\mathrm{d}}{\mathrm{d}s}
		\hat{v}_{n,2}((j+1)\tau,s)
		=\big(im\hat{u}_{n,2}+i\mathcal{N}_{2}(\hat{u}_{n,2},\hat{v}_{n,2})\big)((j+1)\tau,s)+\mathfrak{f}_{j}^{n}(s),\\
		&\big(\hat{u}_{n,2}((j+1)\tau,s),\hat{v}_{n,2}((j+1)\tau,s)\big)\big|_{s=0}
		=\big(\hat{u}(j\tau,n\tau),\hat{v}(j\tau+2\tau,n\tau)\big),
	\end{aligned}\right.
\end{equation}
where the error terms $\mathfrak{g}_{j}^{n}(s)$ and $\mathfrak{f}_{j}^{n}(s)$ are defined by
\begin{equation*}
	\mathfrak{g}_{j}^{n}(s)
	=	\big(im\hat{v}+i\mathcal{N}_{1}(\hat{u},\hat{v})\big)(j\tau+s,n\tau+s)
	-\Big(im\hat{v}_{n,2}+i\mathcal{N}_{1}(\hat{u}_{n,2},\hat{v}_{n,2})\Big)((j+1)\tau,s),
\end{equation*}
and
\begin{equation*}
	\mathfrak{f}_{j}^{n}(s)
	=\big(im\hat{u}+i\mathcal{N}_{2}(\hat{u},\hat{v})\big)(j\tau+2\tau-s,n\tau+s)
	-\big(im\hat{u}_{n,2}+i\mathcal{N}_{2}(\hat{u}_{n,2},\hat{v}_{n,2})\big)((j+1)\tau,s).
\end{equation*}
Integrating \eqref{eq:4.17} over $[0,\tau]$ with respect to $s$ yields
\begin{equation*}
	\left\{
	\begin{aligned}
		&\hat{u}_{n,2}((j+1)\tau,\tau)
		-\hat{u}_{j}^{n}
		=
		\int_{0}^{\tau}\Big(im\hat{v}_{n,2}+i\mathcal{N}_{1}(\hat{u}_{n,2},\hat{v}_{n,2})\Big)((j+1)\tau,s)\mathrm{d}s
		+\int_{0}^{\tau}\mathfrak{g}_{j}^{n}(s)\mathrm{d}s,\\
		&\hat{v}_{n,2}((j+1)\tau,\tau)
		-\hat{v}_{j+2}^{n}
		=
		\int_{0}^{\tau}\big(im\hat{u}_{n,2}+i\mathcal{N}_{2}(\hat{u}_{n,2},\hat{v}_{n,2})\big)((j+1)\tau,s)\mathrm{d}s
		+\int_{0}^{\tau}\mathfrak{f}_{j}^{n}(s)\mathrm{d}s.
	\end{aligned}\right.
\end{equation*}

Comparing the above equations with \eqref{eq:int-smooth} and using the definitions of
$\mathfrak{g}_{j}^{n}$ and $\mathfrak{f}_{j}^{n}$, we conclude that
\begin{equation*}
	(\hat{u}_{n,2}((j+1)\tau,s),\hat{v}_{n,2}((j+1)\tau,s))\big|_{s=\tau}
	=(\hat{u}_{j+1}^{n+1},\hat{v}_{j+1}^{n+1}).
\end{equation*}

For simplicity of notation, let
\begin{equation*}
	\hat{u}_{j+1}^{n,2}(s):=\hat{u}_{n,2}((j+1)\tau,s),\qquad \hat{v}_{j+1}^{n,2}(s):=\hat{v}_{n,2}((j+1)\tau,s).
\end{equation*}
Therefore, \eqref{eq:4.17} can be equivalently written as
\begin{equation}\label{eq:4.13-1}
	\left\{
	\begin{aligned}
		&
		\frac{\mathrm{d}\hat{u}_{j+1}^{n,2}(s)}{\mathrm{d}s}
		=
		im\hat{v}_{j+1}^{n,2}(s)+i\mathcal{N}_{1}(\hat{u}_{j+1}^{n,2}(s),\hat{v}_{j+1}^{n,2}(s))
		+\mathfrak{g}_{j}^{n}(s),\\
		&
		\frac{\mathrm{d}\hat{v}_{j+1}^{n,2}(s)}{\mathrm{d}s}
		=
		im\hat{u}_{j+1}^{n,2}(s)+i\mathcal{N}_{2}(\hat{u}_{j+1}^{n,2}(s),\hat{v}_{j+1}^{n,2}(s))
		+\mathfrak{f}_{j}^{n}(s),\\
		&\big(\hat{u}_{j+1}^{n,2}(s),\hat{v}_{j+1}^{n,2}(s)\big)\big|_{s=0}
		=\big(\hat{u}_j^n,\hat{v}_{j+2}^n\big).
	\end{aligned}\right.
\end{equation}

Then, the following lemma establishes the uniform bounds for $\mathfrak{g}_{j}^{n}$ and
$\mathfrak{f}_{j}^{n}$.
\begin{lemma}\label{lem:error}
Let $T>0$. For any $\tau\in(0,1)$ and any $s\in[0,\tau]$, it holds that
	\begin{align}
		&\max_{n\geq 0,-\infty<j<
		\infty}|\mathfrak{g}_{j}^{n}(s)|\leq
		\hat{C}(m,|\alpha|,|\beta|)
		 (\mathcal{M}_0(T))^6\mathcal{M}(T)\tau,\label{eq:lem-1}\\
		& \max_{n\geq 0,-\infty<j<
			\infty}|\mathfrak{f}_{j}^{n}(s)|
			\leq \hat{C}(m,|\alpha|,|\beta|) (\mathcal{M}_0(T))^6\mathcal{M}(T)\tau,\label{eq:lem-2}
	\end{align}
	where
	$\hat{C}(m,|\alpha|,|\beta|)=2(m+|\alpha|+4|\beta|+1)$.
	Moreover, for any given constant $\delta>0$, there exists a constant $\hat{\tau}>0$, depending on $T$ $\mathcal{M}_0$, $\mathcal{M}$ and the initial data $(\hat{u}(x,0),\hat{u}(x,0))\in \mathrm{C}_{\mathrm{c}}^\infty(\mathbb{R})$, such that if $s\in[0,\tau]$ and $0\leq n\leq (T+1)/\tau$, then
	\begin{equation}\label{eq:sum-g}
		\sup_{\tau\in(0,\hat{\tau})}
		\sum_{j=\infty}^{+\infty}
		\big(|\mathfrak{g}_{j}^{n}(s)|^2+|\mathfrak{f}_{j}^{n}(s)|^2\big)\leq
		\frac{\delta}{4},\quad
		\sup_{\tau\in(0,\hat{\tau})}
		\sum_{j=\infty}^{+\infty}
		\sum_{p=0}^{n}
		\big(|\mathfrak{g}_{j}^{p}(s)|^2+|\mathfrak{f}_{j}^{p}(s)|^2\big)<+\infty.
	\end{equation}
\end{lemma}
\begin{proof}
Let
	\begin{align*}
		&\hat{u}_{j}^{n}(s):=\hat{u}(j\tau+s,n\tau+s),\qquad \hat{v}_{j+2}^{n}(s):=\hat{v}(j\tau+2\tau-s,n\tau+s),\\
		&\hat{v}_{j,+}^{n}(s):=\hat{v}(j\tau+s,n\tau+s),\quad	
		\hat{u}_{j+2,-}^{n}(s):=\hat{u}((j+2)\tau-s,n\tau+s).
	\end{align*}
Then, from \eqref{eq:4.17} and the definition of $\mathfrak{g}_{j}^{n}$, we have
\begin{equation}\label{eq:4.17-1}
	\left\{
	\begin{aligned}
		&\frac{\mathrm{d}}{\mathrm{d}s}
		\hat{u}_{j+1}^{n,2}(s)
		=im\hat{v}_{j,+}^{n}(s)+i\mathcal{N}_{1}(\hat{u}_{j}^{n}(s),\hat{v}_{j,+}^{n}(s)),\\
		&\frac{\mathrm{d}}{\mathrm{d}s}
		\hat{v}_{j+1}^{n,2}(s)
		=im\hat{u}_{j+2,-}^{n}(s)+i\mathcal{N}_{2}(\hat{u}_{j+2,-}^{n}(s),\hat{v}_{j+2}^{n}(s)),\\
		&\big(\hat{u}_{j+1}^{n,2}(s),\hat{v}_{j+1}^{n,2}(s)\big)\big|_{s=0}
		=\big(\hat{u}_j^n,\hat{v}_{j+2}^n\big).
	\end{aligned}\right.
\end{equation}
Integrating \eqref{eq:4.17-1} over $[0,s]$ yields
\begin{align*}
	&\hat{u}_{j+1}^{n,2}(s)-\hat{u}_j^n
	\leq \int_{0}^{s}\big(im\hat{v}_{j,+}^{n}(s)+i\mathcal{N}_{1}(\hat{u}_{j}^{n}(s),\hat{v}_{j,+}^{n}(s))\big)\mathrm{d}s,\\
	&\hat{v}_{j+1}^{n,2}(s)-\hat{v}_{j+2}^n
	\leq
	\int_{0}^{s}\big(im\hat{u}_{j+2,-}^{n}(s)+i\mathcal{N}_{2}(\hat{u}_{j+2,-}^{n}(s),\hat{v}_{j+2}^{n}(s))\big)\mathrm{d}s.
\end{align*}
Since $\mathcal{M}_0=\mathcal{M}_0(T)=
\max_{\mathbb{R}\times[0,T]}(|\hat{u}|+|\hat{v}|+1)$, $\tau\in[0,1]$
and $s\in[0,\tau]$, we deduce
\begin{equation}\label{eq:u-n2-0}
	|\hat{u}_{j+1}^{n,2}(s)|
	\leq (m+|\alpha|+4|\beta|+1)\mathcal{M}_0^3, \quad
	|\hat{v}_{j+1}^{n,2}(s)|\leq (m+|\alpha|+4|\beta|+1)\mathcal{M}_0^3.
\end{equation}
	
Next, subtracting \eqref{5.2} from \eqref{eq:4.17-1}, we get	
\begin{equation*}
	\frac{\mathrm{d}}{\mathrm{d}s}
	\big(\hat{u}_{j}^{n}(s)-\hat{u}_{j+1}^{n,2}(s)\big)=0,\quad
	\frac{\mathrm{d}}{\mathrm{d}s}
	\big(\hat{v}_{j+2}^{n}(s)-\hat{v}_{j+1}^{n,2}(s)\big)=0.
\end{equation*}
Since
\begin{equation*}
	\big(\hat{u}_{j}^{n}(s)-\hat{u}_{j+1}^{n,2}(s)\big)\big|_{s=0}=0,\quad
	\big(\hat{v}_{j+2}^{n}(s)-\hat{v}_{j+1}^{n,2}(s)\big)\big|_{s=0}=0,
\end{equation*}
we have
\begin{equation}\label{eq:u-n2}
	\hat{u}_{j}^{n}(s)-\hat{u}_{j+1}^{n,2}(s)=0,\quad
	\hat{v}_{j+2}^{n}(s)-\hat{v}_{j+1}^{n,2}(s)=0.
\end{equation}

Since $|2s-2\tau|\leq 2\tau$ for $s\in[0,\tau]$, the mean value theorem implies
\begin{align*}
	&
	|\hat{u}_{j+2,-}^{n}(s)-\hat{u}_{j}^n|=
	|\hat{u}((j+2)\tau-s,n\tau+s)-\hat{u}(j\tau+s,n\tau+s)|
	\leq
	2\tau\max_{\mathbb{R}\times[0,T]}|\partial_x\hat{u}|
	\leq 2\mathcal{M}\tau,\\
	&
	|\hat{v}_{j,+}^{n}(s)-\hat{v}_{j+2}^n|
	=|\hat{v}(j\tau+s,n\tau+s)-\hat{v}((j+2)\tau-s,n\tau+s)|
	\leq
	2\tau\max_{\mathbb{R}\times[0,T]}|\partial_x\hat{v}|
	\leq 2\mathcal{M}\tau.
\end{align*}
This together with \eqref{eq:u-n2-0} and \eqref{eq:u-n2} yields
\begin{align*}
	|\mathfrak{g}_{j}^{n}(s)|
	= &
	\Big|
	im(\hat{v}_{j,+}^n(s)-\hat{v}_{j+1}^{n,2}(s))
	+
	i\mathcal{N}_{1}(\hat{u}_{j}^n(s),\hat{v}_{j,+}^n)(s)-
	i\mathcal{N}_{1}(\hat{u}_{j+1}^{n,2}(s),\hat{v}_{j+1}^{n,2}(s))
	\Big|\\
	\leq&
	2(m+|\alpha|+4|\beta|+1)\mathcal{M}_0^6\mathcal{M}\tau,
\end{align*}
which gives \eqref{eq:lem-1} with	$\hat{C}(m,|\alpha|,|\beta|)=2(m+|\alpha|+4|\beta|+1)$.
Furthermore, \eqref{eq:lem-2} follows by a similar argument as in the proof of \eqref{eq:lem-1}.

Finally, since $(\hat u(x,0),\hat v(x,0))\in C_c^\infty(\mathbb R)$, there exists a constant $X>0$ such that the support of $(\hat u(x,0),\hat v(x,0))$ is contained in $[-X,X]$. Then, by \cite[Lemma 3.1]{Y.ZhangandQ.Zhao1}, the support of smooth solution $(\hat{u},\hat{v})$ is contained in $[-X-T,X+T]\times[0,T]$.
Therefore, there are at most $2\left[\frac{X+T}{\tau}\right]+2$ nonzero terms in the sum $\sum_{j=\infty}^{+\infty}
\big(|\mathfrak{g}_{j}^{n}(s)|^2+|\mathfrak{f}_{j}^{n}(s)|^2\big)$, where $\left[\frac{X+T}{\tau}\right]$ denote the greatest integer less than or equal to $\frac{X+T}{\tau}$.
Hence using \eqref{eq:lem-1} and \eqref{eq:lem-2}, there exists a constant $\hat{\tau}>0$ such that, for all $\tau\in(0,\hat{\tau})$,
\begin{align*}
	\sum_{j=\infty}^{+\infty}
	\big(|\mathfrak{g}_{j}^{n}|^2+|\mathfrak{f}_{j}^{n}|^2\big)
	 \leq
	2(2(X+T)+2)\hat{C}(m,|\alpha|,|\beta|)^2 (\mathcal{M}_0)^12\mathcal{M}^2\tau
	\leq\frac{\delta}{4},
\end{align*}
where $\hat{\tau}$ is chosen sufficiently small.
Furthermore, since $n\tau\leq T+1$, we obtain \eqref{eq:sum-g}. This completes the proof.
\end{proof}

Next, we want to estimate the difference between the smooth solution $(\hat{u},\hat{v})$ with initial data $(\hat{u}(x,0),\hat{v}(x,0))\in \mathrm{C}_c^\infty(\mathbb{R})$ and the time splitting solution $(u^{(\tau)},v^{(\tau)})$.
To this end, we define
\begin{equation*}
	\mathcal{U}_{j+1}^{n+1}:=\mathcal{U}((j+1)\tau,(n+1)\tau)=\hat{u}_{j+1}^{n+1}-u_{j+1}^{n+1}
\end{equation*}
and
\begin{equation*}
	\mathcal{U}_{j+1}^{n+1}:=\mathcal{U}((j+1)\tau,(n+1)\tau)=\hat{u}_{j+1}^{n+1}-u_{j+1}^{n+1}.
\end{equation*}
Then we estimate the difference $(\mathcal{U}_{j+1}^{n+1},\mathcal{V}_{j+1}^{n+1})$.
For simplicity, let
\begin{equation*}
	\mathcal{U}_{j+1}^{n,2}(s):=\hat{u}_{j+1}^{n,2}(s)-u_{j+1}^{n,2}(s), \qquad \mathcal{V}_{j+1}^{n,2}(s):=\hat{v}_{j+1}^{n,2}(s)-v_{j+1}^{n,2}(s).
	 \end{equation*}
By construction and by \eqref{5.5}--\eqref{5.3}, we observe that
\begin{equation}\label{eq:U-CAL}
	(\mathcal{U}_{j+1}^{n,2}(\tau),\mathcal{V}_{j+1}^{n,2}(\tau))
	=(\mathcal{U}_{j+1}^{n+1},\mathcal{V}_{j+1}^{n+1}),\qquad
		(\mathcal{U}_{j+1}^{n,2}(0),\mathcal{V}_{j+1}^{n,2}(0))=
	(\mathcal{U}_{j}^{n},\mathcal{V}_{j+2}^{n}).
\end{equation}

Thus, we now turn to the estimate of $(\mathcal{ U}_{j+1}^{n,2}(\tau),\mathcal {V}_{j+1}^{n,2}(\tau))$. To this end, we first derive the following differential inequalities, which differ from the estimates \eqref{3.3}--\eqref{3.5} in that they contain the additional source terms $|\mathfrak{g}_{j}^{n}|^2$ or $|\mathfrak{f}_{j}^{n}|^2$. Moreover, the differential inequality for $\tilde{\mathcal L}_{j+1}^{n,2}$ also involves $\tilde{\mathcal L}_{j+1}^{n,2}$ itself on the right-hand side.
	\begin{lemma}\label{lemma5.1}
		Let
	\begin{equation}\label{eq:def-Lij}
		\tilde{\mathcal{L}}_{j+1}^{n,2}(s):=|\mathcal{U}_{j+1}^{n,2}(s)|^{2}+|\mathcal{V}_{j+1}^{n,2}(s)|^{2}
	\end{equation}
	and
		\begin{equation}\label{eq:def-Dij}
		\tilde{\mathcal{D}}_{j+1}^{n,2}(s)
		:=
		|\mathcal{U}_{j+1}^{n,2}(s)|^{2}\big(|v_{j+1}^{n,2}(s)|^{2}
		+|\hat{v}_{j+1}^{n,2}(s)|^2\big)
		+|\mathcal{V}_{j+1}^{n,2}(s)|^{2}\big(|u_{j+1}^{n,2}(s)|^{2}
		+|\hat{u}_{j+1}^{n,2}(s)|^2\big).
		\end{equation}
	Let $\tau_{\sharp}\in(0,1)$ be the constant given by Lemma \ref{lem:proposition4.1}.
	Then for $j=0,\pm1,\ldots,n=0,1,\ldots$, $s\in[0,\tau]$ and $\tau\in(0,\tau_{\sharp})$, there exists a constant $\hat{C}_1>0$, depending only on the system \eqref{eq:NLDE}, such that
		\begin{align}
		\frac{\mathrm{d}|\mathcal{U}_{j+1}^{n,2}(s)|^{2}}{\mathrm{d}s}
		\leq&
		\hat{C}_1\big(\tilde{\mathcal{L}}_{j+1}^{n,2}(s)
		+\tilde{\mathcal{D}}_{j+1}^{n,2}(s)+|\mathfrak{g}_{j}^{n}(s)|^2\big),
		\label{eq:smooth-dU}\\
		\frac{\mathrm{d}|\mathcal{V}_{j+1}^{n,2}(s)|^{2}}{\mathrm{d}s}
		\leq &\hat{C}_1
		\big(\tilde{\mathcal{L}}_{j+1}^{n,2}(s)
		+\tilde{\mathcal{D}}_{j+1}^{n,2}(s)+|\mathfrak{f}_{j}^{n}(s)|^2\big),
		\label{eq:smooth-dV}
		\end{align}
		and
		\begin{equation}\label{eq:smooth-dUV}
			\frac{\mathrm{d}\tilde{\mathcal{L}}_{j+1}^{n,2}(s)}{\mathrm{d}s}
			\leq
			2\hat{C}_1
			\big(\tilde{\mathcal{L}}_{j+1}^{n,2}(s)+\tilde{\mathcal{D}}_{j+1}^{n,2}(s)+|\mathfrak{g}_{j}^{n}(s)|^2+|\mathfrak{f}_{j}^{n}(s)|^2\big).
		\end{equation}	
	Here the constant $\hat{C}_1$ is independent of $\mathcal{M}_0$, $\mathcal{M}$ and $\tau$.
	\end{lemma}
	\begin{proof}
	Subtracting \eqref{eq:4.13-1} from \eqref{5.3}, we get
	\begin{align}
			\frac{\mathrm{d}}{\mathrm{d}s}\mathcal{U}_{j+1}^{n,2}(s)
				=&im\mathcal{V}_{j+1}^{n,2}(s)
				+
				i\Big(
				\mathcal{N}_{1}(\hat{u}_{j+1}^{n,2}(s),\hat{v}_{j+1}^{n,2}(s)) -\mathcal{N}_{1}(u_{j+1}^{n,2}(s),v_{j+1}^{n,2}(s))
				\Big)
				+\mathfrak{g}_{j}^{n}(s)
				,\label{eq:smooth-minus-splitting}\\
				\frac{\mathrm{d}}{\mathrm{d}s}\mathcal{V}_{j+1}^{n,2}(s)
				=&im\mathcal{U}_{j+1}^{n,2}(s)
				+
				i\left(
				\mathcal{N}_{2}(\hat{u}_{j+1}^n(s),\hat{v}_{j+1}^{n,2}(s))
				-\mathcal{N}_{2}(u_{j+1}^{n,2}(s),v_{j+1}^{n,2}(s))\right)
				+\mathfrak{f}_{j}^{n}(s).
				\label{eq:smooth-minus-splitting-1}
			\end{align}
Then, multiplying the equation \eqref{eq:smooth-minus-splitting} by $\overline{\mathcal{U}_{j+1}^{n,2}}(s)$, we obtain
\begin{equation*}
	\begin{aligned}
		\frac{\mathrm{d}|\mathcal{U}_{j+1}^{n,2}(s)|^{2}}{\mathrm{d}s}
		=&2\mathrm{Re}\Big\{im\mathcal{V}_{j+1}^{n,2}(s) \overline{\mathcal{U}_{j+1}^{n,2}}(s)\Big\}
		+2\mathrm{Re}\Big\{\mathfrak{g}_{j}^{n}(s) \overline{\mathcal{U}_{j+1}^{n,2}}(s)\Big\}
		\\ &+2\mathrm{Re}\left\{
		i\big(
		\mathcal{N}_{1}(\hat{u}_{j+1}^{n,2}(s),\hat{v}_{j+1}^{n,2}(s)) -\mathcal{N}_{1}(u_{j+1}^{n,2}(s),v_{j+1}^{n,2}(s))
		\big)
		\overline{\mathcal{U}_{j+1}^{n,2}}(s)\right\},
	\end{aligned}	
\end{equation*}
where $\mathrm{Re}{z}$ is the real part of $z$.
Similarly, multiplying the equation \eqref{eq:smooth-minus-splitting-1} by $\overline{\mathcal{V}_{j+1}^{n,2}}(s)$ yields
\begin{equation*}
	\begin{aligned}
		\frac{\mathrm{d}|\mathcal{V}_{j+1}^{n,2}(s)|^{2}}{\mathrm{d}s}
		=&2\mathrm{Re}\left\lbrace im\mathcal{U}_{j+1}^{n,2}(s) \overline{\mathcal{V}_{j+1}^{n,2}}(s)\right\rbrace
		+
		2\mathrm{Re}\Big\{\mathfrak{f}_{j}^{n}(s) \overline{\mathcal{V}_{j+1}^{n,2}}(s)\Big\}
		\\ &+2\mathrm{Re}\left\lbrace
		i\big(
		\mathcal{N}_{2}(\hat{u}_{j+1}^{n,2}(s),\hat{v}_{j+1}^{n,2}(s)) -\mathcal{N}_{2}(u_{j+1}^{n,2}(s),v_{j+1}^{n,2}(s))
		\big)
		\overline{\mathcal{V}_{j+1}^{n,2}}(s)\right\rbrace.
	\end{aligned}	
\end{equation*}
By Young's inequality, we obtain
\begin{equation*}
	2\mathrm{Re}\Big\{\mathfrak{g}_{j}^{n}(s) \overline{\mathcal{U}_{j+1}^{n,2}}(s)\Big\}
	\leq
	|\mathcal{U}_{j+1}^{n,2}(s)|^{2}+|\mathfrak{g}_{j}^{n}(s)|^{2},\qquad
	2\mathrm{Re}\Big\{\mathfrak{f}_{j}^{n}(s) \overline{\mathcal{V}_{j+1}^{n,2}}(s)\Big\}
	\leq
	|\mathcal{V}_{j+1}^{n,2}(s)|^{2}+|\mathfrak{f}_{j}^{n}(s)|^{2}.
\end{equation*}	
Thus, recalling the notations $\mathcal{N}_{1}$ and $\mathcal{N}_{2}$ in \eqref{eq:mathcal-N} and carrying out similar arguments as in the proof of \eqref{3.3}--\eqref{3.5}, we obtain \eqref{eq:smooth-dU}-- \eqref{eq:smooth-dUV}
by taking
$\hat{C}_1=4(|\alpha|+4|\beta|)+m+1$.
The proof is complete.
\end{proof}

The following lemma first derives a
uniform estimate for the quantity $\tilde{\mathcal{L}}_{j+1}^{n,2}+\tilde{\mathcal{D}}_{j+1}^{n,2}$ on the right-hand side of \eqref{eq:smooth-dU}--\eqref{eq:smooth-dUV},and then establishes pointwise estimates for $(\mathcal{U}_{j+1}^{n,2}(\tau),\mathcal{V}_{j+1}^{n,2}(\tau))$.
Compared with the corresponding estimates \eqref{3.9}--\eqref{3.8} in Lemma \ref{lem:pointwise-estimates-UV}, these estimates involve the additional source terms $|\mathfrak{g}_{j}^{n}|^2+|\mathfrak{f}_{j}^{n}|^2$.
\begin{lemma}\label{lemma5.4}
	Let
	\begin{equation}\label{def:ae}
	\mathfrak{a}=\mathfrak{s}_{j+1}^{n,2}(0)+\hat{\mathfrak{s}}_{j+1}^{n,2}(0),\quad
	\mathfrak{e}(\tau)=\int_{0}^{\tau}(|\mathfrak{g}_{j}^{n}(s)|^2+|\mathfrak{f}_{j}^{n}(s)|^2)\mathrm{ds}.
	\end{equation}
	Let $\tau_{\sharp}\in(0,1)$ be the constant given by Lemma \ref{lem:proposition4.1}.
	Then for $j=0,\pm1,\ldots,n=0,1,\ldots$, $s\in[0,\tau]$ and $\tau\in(0,\tau_{\sharp})$, there exist constants $\hat{C}_{2}$, $\hat{C}_{3}$ and $\hat{C}_{4}>0$, depending only on $c_0$ and the system \eqref{eq:NLDE}, such that
	\begin{align}\label{lem:smooth-LD}
	\tilde{\mathcal{L}}_{j+1}^{n,2}(s)+\tilde{\mathcal{D}}_{j+1}^{n,2}(s)
	\leq &
	\mathrm{e}^{\hat{C}_{2}+\hat{C}_{2}(\mathfrak{a}\tau+\mathfrak{e}(\tau))}
	\big(\tilde{\mathcal{L}}_{j+1}^{n,2}(0)+\tilde{\mathcal{D}}_{j+1}^{n,2}(0)\big)\nonumber\\
	&+
	\mathrm{e}^{\hat{C}_{2}+\hat{C}_{2}(\mathfrak{a}\tau+\mathfrak{e}(\tau))}
	\big(\hat{C}_{2}+\hat{C}_{2}(\mathfrak{a}+\mathfrak{e}(\tau))\big)\int_{0}^{\tau}(|\mathfrak{g}_{j}^{n}(s)|^2+|\mathfrak{f}_{j}^{n}(s)|^2)\mathrm{d}s,
	\end{align}
	and
  	\begin{align}
  	|\mathcal{U}_{j+1}^{n,2}(\tau)|^{2}
  	\leq&
  	|\mathcal{U}_{j+1}^{n,2}(0)|^{2}
  	+
  	\hat{C}_{3}\mathrm{e}^{\hat{C}_{3}(\mathfrak{a}\tau+\mathfrak{e}(\tau))}
  	\big(\tilde{\mathcal{L}}_{j+1}^{n,2}(0)\tau+\tilde{\mathcal{D}}_{j+1}^{n,2}(0)\tau\big)
  	\nonumber\\
  	&
  	+\hat{C}_{3}\mathrm{e}^{\hat{C}_{3}(\mathfrak{a}\tau+\mathfrak{e}(\tau))}
  	\big(\hat{C}_{4}+\hat{C}_{4}(\mathfrak{a}\tau+\mathfrak{e}(\tau))\big)
  	\int_{0}^{\tau}(|\mathfrak{g}_{j}^{n}(s)|^2+|\mathfrak{f}_{j}^{n}(s)|^2)\mathrm{d}s\,
  	\label{eq:smooth-U}\\
  	|\mathcal{V}_{j+1}^{n,2}(\tau)|^{2}
  	\leq&|\mathcal{V}_{j+1}^{n,2}(0)|^{2}
  	+
  	\hat{C}_{3}\mathrm{e}^{\hat{C}_{3}(\mathfrak{a}\tau+\mathfrak{e}(\tau))}
  	\big(\tilde{\mathcal{L}}_{j+1}^{n,2}(0)\tau+\tilde{\mathcal{D}}_{j+1}^{n,2}(0)\tau\big)
  	\nonumber\\
  	&
  	+\hat{C}_{3}\mathrm{e}^{\hat{C}_{3}(\mathfrak{a}\tau+\mathfrak{e}(\tau))}
  	\big(\hat{C}_{4}+\hat{C}_{4}(\mathfrak{a}\tau+\mathfrak{e}(\tau))\big)
  	\int_{0}^{\tau}(|\mathfrak{g}_{j}^{n}(s)|^2+|\mathfrak{f}_{j}^{n}(s)|^2)\mathrm{d}s.\label{eq:smooth-V}
  \end{align}
  Here the constants $\hat{C}_2$, $\hat{C}_3$ and $\hat{C}_4$ are independent of $\mathcal{M}_0$, $\mathcal{M}$ and $\tau$.
\end{lemma}
\begin{proof}
First, using \eqref{eq:4.13-1} and carrying out similar arguments as in the proof of \eqref{2.2}, \eqref{2.3} and \eqref{2.1},  we deduce that
\begin{align}
	&\frac{\mathrm{d}|\hat{u}_{j+1}^{n,2}(s)|^{2}}{\mathrm{d}s}
	\leq
	(m+1)\big( |\hat{u}_{j+1}^{n,2}(s)|^{2}+|\hat{v}_{j+1}^{n,2}(s)|^{2}\big) +4|\beta||\hat{u}_{j+1}^{n,2}(s)|^{2}|\hat{v}_{j+1}^{n,2}(s)|^{2}+|\mathfrak{g}_{j}^{n}(s)|^2
	,\label{eq:4.36}\\
	&\frac{\mathrm{d}|\hat{v}_{j+1}^{n,2}(s)|^{2}}{\mathrm{d}s}
	\leq
	(m+1)\big( |\hat{u}_{j+1}^{n,2}(s)|^{2}+|\hat{v}_{j+1}^{n,2}(s)|^{2}\big) +4|\beta||\hat{u}_{j+1}^{n,2}(s)|^{2}|\hat{v}_{j+1}^{n,2}(s)|^{2}
	+|\mathfrak{f}_{j}^{n}(s)|^2,\label{eq:4.37}
\end{align}
and
\begin{equation}\label{eq:4.39}
	\frac{\mathrm{d}|\hat{u}_{j+1}^{n,2}(s)|^{2}}{\mathrm{d}s}+\frac{\mathrm{d}|\hat{v}_{j+1}^{n,2}(s)|^{2}}{\mathrm{d}s}\leq
	|\hat{u}_{j+1}^{n,2}(s)|^{2}+|\hat{v}_{j+1}^{n,2}(s)|^{2}
	+|\mathfrak{g}_{j}^{n}(s)|^2+|\mathfrak{f}_{j}^{n}(s)|^2.
\end{equation}
Therefore, for $s\in[0,\tau]$, we get
\begin{equation}\label{eq:4.38}
	|\hat{u}_{j+1}^{n,2}(s)|^{2}+|\hat{v}_{j+1}^{n,2}(s)|^{2}
	\leq \mathrm{e}^{\tau}\Big(
	|\hat{u}_{j+1}^{n,2}(0)|^{2}+|\hat{v}_{j+1}^{n,2}(0)|^{2}
	+\int_{0}^{\tau}(|\mathfrak{g}_{j}^{n}(s)|^2+|\mathfrak{f}_{j}^{n}(s)|^2)\mathrm{ds}\Big).
\end{equation}
Here and in the sequel, we let
\begin{equation*}
	\hat{\mathfrak{s}}_{j+1}^{n,2}(s)	
	=|\hat{u}_{j+1}^{n,2}(s)|^{2}+|\hat{v}_{j+1}^{n,2}(s)|^{2},\quad s\in[0,\tau].
\end{equation*}

Recalling the definition of $\mathfrak{s}_{j}^{n,2}(s)$ in Definition \ref{def:functional-uv} and using \eqref{2.1}, we have
		\begin{equation*}
		|u_{j+1}^{n,2}(s)|^{2}+|v_{j+1}^{n,2}(s)|^{2}
		=
		|u_{j+1}^{n,2}(0)|^{2}+|v_{j+1}^{n,2}(0)|^{2}
		=\mathfrak{s}_{j+1}^{n,2}(0),\quad s\in[0,\tau].
	\end{equation*}
Moreover, for $\tau\in(0,\tau_{\sharp})$, it follows from \eqref{3.16.5} that
\begin{equation*}
	\mathfrak{s}_{j+1}^{n,2}(0)\tau\leq c_0.
\end{equation*}

Since $\tau\in(0,\tau_{\sharp})\subset(0,1)$, by the above estimates, \eqref{2.2}--\eqref{2.3} and \eqref{eq:def-Lij}--\eqref{eq:smooth-dUV}, we have		
\begin{align*}
	&\frac{\mathrm{d}\big(\tilde{\mathcal{D}}_{j+1}^{n,2}(s)+\tilde{\mathcal{L}}_{j+1}^{n,2}(s)\big)}{ds}\\
	\leq &
	\Big(2\hat{C}_1+(\hat{C}_1+m+1)\Big(\mathfrak{s}_{j+1}^{n,2}(0)+\hat{\mathfrak{s}}_{j+1}^{n,2}(0)+\mathrm{e}\int_{0}^{\tau}(|\mathfrak{g}_{j}^{n}|^2+|\mathfrak{f}_{j}^{n}|^2)\mathrm{ds}\Big)
	+|\mathfrak{g}_{j}^{n}(s)|^2+|\mathfrak{f}_{j}^{n}(s)|^2
	\Big)
	\tilde{\mathcal{L}}_{j+1}^{n,2}(s)\\
	&+\Big(2\hat{C}_1+ (\hat{C}_1+4|\beta|)\Big(\mathfrak{s}_{j+1}^{n,2}(0)+\hat{\mathfrak{s}}_{j+1}^{n,2}(0)+\mathrm{e}\int_{0}^{\tau}(|\mathfrak{g}_{j}^{n}(s)|^2+|\mathfrak{f}_{j}^{n}(s)|^2)\mathrm{ds}\Big)\Big)
	\tilde{\mathcal{D}}_{j+1}^{n,2}(s)\\
	&+\Big(2\hat{C}_1+\hat{C}_1\Big(\mathfrak{s}_{j+1}^{n,2}(0)+\hat{\mathfrak{s}}_{j+1}^{n,2}(0)+\mathrm{e}\int_{0}^{\tau}(|\mathfrak{g}_{j}^{n}(s)|^2+|\mathfrak{f}_{j}^{n}(s)|^2)\mathrm{ds}\Big)\Big)
	(|\mathfrak{g}_{j}^{n}(s)|^2+|\mathfrak{f}_{j}^{n}(s)|^2).
\end{align*}	
Then using Gronwall's inequality and setting
$\hat{C}_2=2\hat{C}_1+(\hat{C}_1+m+4|\beta|+2)\mathrm{e}$, we obtain \eqref{lem:smooth-LD}.

Next, using \eqref{eq:smooth-dU} and \eqref{lem:smooth-LD}, we deduce that
		\begin{align*}
			\frac{\mathrm{d}|\mathcal{U}_{j+1}^{n}(s)|^{2}}{\mathrm{d}s}
			\leq&
			\hat{C}_1\big(\tilde{\mathcal{L}}_{j+1}^{n,2}(s)
			+\tilde{\mathcal{D}}_{j+1}^{n,2}(s)+|\mathfrak{g}_{j}^{n}(s)|^2\big)\\
	\leq&
\hat{C}_1
\mathrm{e}^{\hat{C}_{2}+\hat{C}_{2}(\mathfrak{a}\tau+\mathfrak{e}(\tau))}
\big(\tilde{\mathcal{L}}_{j+1}^{n,2}(0)+\tilde{\mathcal{D}}_{j+1}^{n,2}(0)\big)
+\hat{C}_1|\mathfrak{g}_{j}^{n}(s)|^2\\
&+
\hat{C}_1
\mathrm{e}^{\hat{C}_{2}+\hat{C}_{2}(\mathfrak{a}\tau+\mathfrak{e}(\tau))}
\Big(\big(\hat{C}_{2}+\hat{C}_{2}(\mathfrak{a}+\mathfrak{e}(\tau))\big)\int_{0}^{\tau}(|\mathfrak{g}_{j}^{n}(s)|^2+|\mathfrak{f}_{j}^{n}(s)|^2)\mathrm{d}s\Big).
		\end{align*}
Integrating the above inequality over $[0,\tau]$, we obtain \eqref{eq:smooth-U}
by setting
 $\hat{C}_{3}=\hat{C}_1\mathrm{e}^{\hat{C}_{2}}+\hat{C}_{2}$ and $\hat{C}_{4}=\hat{C}_{2}+1$.
Finally, \eqref{eq:smooth-V} follows as in \eqref{eq:smooth-U}, by using \eqref{eq:smooth-dV} in place of \eqref{eq:smooth-dU}.
	\end{proof}

\subsection{$L^2$ estimates and proof of main result}
This subsection is devoted to establishing the $L^2$ estimates for $(\mathcal{U},\mathcal{V})$ and proving the main theorem. Using the pointwise estimates established in Lemma \ref{lemma5.4}, we first aim to establish the local $L^2$ estimates for $(\mathcal{U},\mathcal{V})$.
To achieve this, we construct a Glimm-type functional $\tilde{\mathcal F}(n;\Delta)$ for each discrete characteristic triangle $\Delta=\Delta(j_1,n_1;n_0)$, see Fig.~\ref{fig1}.

We begin by introducing the following functionals in constructing $\tilde{\mathcal F}(n;\Delta)$.
\begin{definition}\label{definition6.3}
	For $j,j_{1}=0,\pm1,\ldots$ and $n,n_{1}=0,1,\ldots$ with $0\leq n_0\leq n\leq n_{1}$, define
	\begin{align*}
		&\tilde{\mathcal{L}}(n;\Delta)
		=\sum_{j=j_{1}-n_{1}+n}^{j_{1}+n_{1}-n}
		\left( |\mathcal{U}_{j}^{n}|^{2}
		+|\mathcal{V}_{j}^{n}|^{2}\right),
		\qquad\quad
		\tilde{\mathcal{E}}(n;\Delta)
		=\sum_{j=j_{1}-n_{1}+n}^{j_{1}+n_{1}-n-2}
		\left( |\mathfrak{g}_{j}^{n}|^{2}
		+|\mathfrak{f}_{j}^{n}|^{2}\right),
		\\
		&\tilde{\mathcal{D}}(n;\Delta)=
		\sum_{j=j_{1}-n_{1}+n}^{j_{1}+n_{1}-n-2}			
			\Big(
			|\mathcal{U}_{j}^{n}|^{2}
			\big(|v_{j+2}^{n}|^{2}+|\hat{v}_{j+2}^{n}|^{2}\big)
			+|\mathcal{V}_{j+2}^{n}|^{2}
			\big(|u_{j}^{n}|^{2}+|\hat{u}_{j}^{n}|^{2}\big)
			\Big),
	\end{align*}
		and
		\begin{equation*}\tilde{\mathcal{Q}}(n;\Delta)
		=\sum_{j_{1}-n_{1}+n\leq j< k\leq j_{1}+n_{1}-n }
		\Big(
		|\mathcal{U}_{j}^{n}|^{2}
		\big(|v_{k}^{n}|^{2}+|\hat{v}_{k}^{n}|^{2}\big)
		+|\mathcal{V}_{k}^{n}|^{2}
		\big(|u_{j}^{n}|^{2}+|\hat{u}_{j}^{n}|^{2}\big)
		\Big).\end{equation*}
	\end{definition}
	
Now we define a new Glimm-type functional on a discrete characteristic angle.
\begin{definition}[Glimm-type functional]\label{def:Glimm}
	For any integer $n=0,1,2,\ldots$ and any $\tau>0$, define
	\begin{equation*}
	\tilde{\mathcal{F}}(n;\Delta)
	=\tilde{\mathcal{L}}(n;\Delta)\tau+\tilde{\kappa }\tilde{\mathcal{Q}}(n;\Delta)\tau^{2},
	\end{equation*}
	where $\tilde{\kappa }$ is a large positive constant to be chosen later.
\end{definition}
To estimate the bound of $\tilde{\mathcal{F}}$ in $\Delta$, we define
\begin{equation*}
	\hat{\mathfrak{s}}(n;\Delta)
	=\sum_{j=j_{1}-n_{1}+n}^{j_{1}+n_{1}-n}|\hat{u}_{j}^{n}|^{2}|\hat{v}_{j}^{n}|^{2},	
	\qquad
	\hat{\mathfrak{q}}(n;\Delta)
	=\sum_{j=j_{1}-n_{1}+n}^{j_{1}+n_{1}-n-2}|\hat{u}_{j}^{n}|^{2}|\hat{v}_{j+2}^{n}|^{2}.	
\end{equation*}

With the coefficient $\tilde{\kappa }$ chosen properly, the following proposition shows that the Glimm-type functional $\tilde{\mathcal{F}}$ is uniformly bounded.
\begin{proposition}\label{lemma5.6}
There exist constants $\delta>0$ and $\hat{C}_7>0$, independent of $\mathcal{M}_0$, $\mathcal{M}$ and $\tau$, such that if $\mathfrak{s}(n_{0};\Delta)\tau\leq\delta$, $\sup\limits_{n_0\leq n\leq n_1-1}\hat{\mathfrak{s}}(n;\Delta)\tau\leq\delta$ and $\sup\limits_{n_0\leq n \leq n_1-1}\int_{0}^{\tau}\tilde{\mathcal{E}}(n;\Delta)\mathrm{d}s\leq\delta$, then there exists a constant $0<\hat{\tau}_{\sharp}<1$, depending on $\mathcal{M}_0$, $\mathcal{M}$ and the initial data $(\hat{u}(x,0),\hat{v}(x,0))$, such that for any $\tau\in(0,\hat{\tau}_{\sharp})$,
	\begin{equation}\label{eq:F-1-UNI}
		\tilde{\mathcal{F}}(n+1;\Delta)
	\leq
	\Big(1+\hat{C}_7\tau+\hat{C}_7
	\big(\mathfrak{q}(n;\Delta)\tau^2+\hat{\mathfrak{q}}(n;\Delta)\tau^2\big)\Big)
	\tilde{\mathcal{F}}(n;\Delta)
	+\hat{C}_7\tau\int_{0}^{\tau}\tilde{\mathcal{E}}(n;\Delta)\mathrm{d}s.
	\end{equation}
Moreover, let $\tilde{\mathcal{K}}(n,n_0)=\hat{C}_7(n-n_0)\tau
	+\hat{C}_7\sum_{p=n_{0}}^{n-1}
	\big(\mathfrak{q}(p;\Delta)\tau^2+\hat{\mathfrak{q}}(p;\Delta)\tau
	^2\big)$.
	Then for any $T>0$, there exists a constant $\hat{C}_8(T)>0$, independent of $\mathcal{M}_0$, $\mathcal{M}$ and $\tau$, such that for all integers $n$ satisfying $\:0\leq n_{0}\leq n\leq n_{1}$, $n\leq (T+1)/\tau$, with the range of $n$ shown in Fig.~\ref{fig-n}, it holds that
	\begin{equation}\label{eq:smooth-F}
		\tilde{\mathcal{F}}(n;\Delta)
		\leq
		\mathrm{e}^{\tilde{\mathcal{K}}(n,n_0)}
		\Big(\tilde{\mathcal{F}}(n_0;\Delta)+\hat{C}_7\tau\int_{0}^{\tau}\sum_{p=n_{0}}^{n-1}\tilde{\mathcal{E}}(p;\Delta)\mathrm{d}s\Big)
	\end{equation}
	and
	\begin{equation*}
		\tilde{\mathcal{K}}(n,n_0)\leq \hat{C}_8(T).
	\end{equation*}
\end{proposition}
\begin{proof}
We split the proof into three steps.

\Step \label{step:smooth-1}
For $\tilde{\mathcal{L}}(n;\Delta)$, we aim to show the following claim:
There exists a constant $0<\hat{\tau}_{\sharp}<1$, depending on $\mathcal{M}_0$, $\mathcal{M}$ and the initial data $(\hat{u}(x,0),\hat{u}(x,0))\in\mathrm{C}_{\mathrm{c}}^\infty(\mathbb{R})$, such that
for all $\tau\in(0,\tau_{\sharp})$, there exists a constant $\hat{C}_5>0$, independent of $\mathcal{M}_0$, $\mathcal{M}$ and $\tau$, such that
	\begin{align}\label{corollary6.1}
		\tilde{\mathcal{L}}(n+1;\Delta)-\tilde{\mathcal{L}}(n;\Delta)
		\leq&
		\hat{C}_5
		\Big(\tilde{\mathcal{L}}(n;\Delta)\tau+\tilde{\mathcal{D}}(n;\Delta)\tau+\int_{0}^{\tau}\tilde{\mathcal{E}}(n;\Delta)\mathrm{d}s\Big).
\end{align}
To prove \eqref{corollary6.1}, we deduce from Definition \ref{definition6.3}, \eqref{eq:U-CAL} and the definition of $\tilde{\mathcal L}_{j+1}^{n,2}$ in \eqref{eq:def-Lij} that
\begin{align}
	\tilde{\mathcal{L}}(n+1;\Delta)
	=&\sum_{j=j_{1}-n_{1}+n+1}^{j_{1}+n_{1}-n-1}
	\left( |\mathcal{U}_{j}^{n+1}|^{2}
	+|\mathcal{V}_{j}^{n+1}|^{2}\right)
	=
	\sum_{j=j_{1}-n_{1}+n}^{j_{1}+n_{1}-n-2}
	\left( |\mathcal{U}_{j+1}^{n+1}|^{2}
	+|\mathcal{V}_{j+1}^{n+1}|^{2}\right)\nonumber \\
	=&	
	\sum_{j=j_{1}-n_{1}+n}^{j_{1}+n_{1}-n-2}
	\left( |\mathcal{U}_{j+1}^{n,2}(\tau)|^{2}
	+|\mathcal{V}_{j+1}^{n,2}(\tau)|^{2}\right)
	=\sum_{j=j_{1}-n_{1}+n}^{j_{1}+n_{1}-n-2}
	\tilde{\mathcal{L}}_{j+1}^{n,2}(\tau),\label{eq:L-n}
\end{align}
and
\begin{equation}\label{eq:L-n-0}
	\sum_{j=j_{1}-n_{1}+n}^{j_{1}+n_{1}-n-2}
	\tilde{\mathcal{L}}_{j+1}^{n,2}(0)
	\leq
	\sum_{j=j_{1}-n_{1}+n}^{j_{1}+n_{1}-n}
	\left( |\mathcal{U}_{j}^{n}|^{2}
	+|\mathcal{V}_{j}^{n}|^{2}\right)
	=	\tilde{\mathcal{L}}(n;\Delta).
\end{equation}

Then, using \eqref{eq:smooth-dUV} and \eqref{lem:smooth-LD}, we have	
\begin{align*}
\frac{\mathrm{d}\tilde{\mathcal{L}}_{j+1}^{n,2}(s)}{\mathrm{d}s}
	\leq&
2\hat{C}_1
\mathrm{e}^{\hat{C}_{2}+\hat{C}_{2}(\mathfrak{a}\tau+\mathfrak{e}(\tau))}
\big(\tilde{\mathcal{L}}_{j+1}^{n,2}(0)+\tilde{\mathcal{D}}_{j+1}^{n,2}(0)\big)
+2\hat{C}_1\big(|\mathfrak{g}_{j}^{n}(s)|^2+|\mathfrak{f}_{j}^{n}(s)|^2\big)\\
&+
2\hat{C}_1
\mathrm{e}^{\hat{C}_{2}+\hat{C}_{2}(\mathfrak{a}\tau+\mathfrak{e}(\tau))}
\Big(\big(\hat{C}_{2}+\hat{C}_{2}(\mathfrak{a}+\mathfrak{e}(\tau))\big)\int_{0}^{\tau}(|\mathfrak{g}_{j}^{n}(s)|^2+|\mathfrak{f}_{j}^{n}(s)|^2)\mathrm{d}s\Big).
	\end{align*}
Integrating the above equation over $[0,\tau]$ with respect to $s$ and using $0<\tau<1$, we obtain
\begin{align}\label{6.30.5}
	\tilde{\mathcal{L}}_{j+1}^{n,2}(\tau)
	 \leq&
	 \tilde{\mathcal{L}}_{j+1}^{n,2}(0)
	 +2\hat{C}_1\mathrm{e}^{\hat{C}_{2}+\hat{C}_{2}(\mathfrak{a}\tau+\mathfrak{e}(\tau))}
	 \Big(\tilde{\mathcal{L}}_{j+1}^{n,2}(0)\tau+\tilde{\mathcal{D}}_{j+1}^{n,2}(0)\tau\Big)
	 \nonumber\\
	 &+2\hat{C}_1\mathrm{e}^{\hat{C}_{2}+\hat{C}_{2}(\mathfrak{a}\tau+\mathfrak{e}(\tau))}
	 \big(1+\hat{C}_{2}\tau+\hat{C}_{2}(\mathfrak{a}\tau+\mathfrak{e}(\tau))\big)\int_{0}^{\tau}(|\mathfrak{g}_{j}^{n}(s)|^2+|\mathfrak{f}_{j}^{n}(s)|^2)\mathrm{d}s.
	\end{align}
	
Next, we estimate $\mathfrak{a}\tau+\mathfrak{e}(\tau)=
	\mathfrak{s}_{j+1}^{n,2}(0)\tau+\hat{\mathfrak{s}}_{j+1}^{n,2}(0)\tau
+\int_{0}^{\tau}(|\mathfrak{g}_{j}^{n}(s)|^2+|\mathfrak{f}_{j}^{n}(s)|^2)\mathrm{ds}$, which is defined by \eqref{def:ae}.
By \eqref{5.1}, there exists a constant $\hat{c}_0>0$, independent of $\mathcal{M}_0$, $\mathcal{M}$ and $\tau$, such that
$\|\hat{u}(x,0)\|_{\mathrm{L}^{2}(\mathbb{R})} +
\|\hat{v}(x,0))\|_{\mathrm{L}^{2}(\mathbb{R})}\leq \hat{c}_0.$ Then
it follows from \cite[Lemma 2.1]{Y.ZhangandQ.Zhao1} that there exists a constant $0<\hat{\tau}_{\sharp}<1$, depending on $\mathcal{M}_0$, $\mathcal{M}$ and the initial data $(\hat{u}(x,0),\hat{u}(x,0))\in \mathrm{C}_{\mathrm{c}}^\infty(\mathbb{R})$, such that
\begin{align*}
	\hat{\mathfrak{s}}_{j+1}^{n,2}(0)\tau
	=&(|\hat{u}(j\tau,n\tau)|^{2}+|v((j+2)\tau,n\tau)|^{2})\tau\\
	=&\int_{j\tau}^{(j+1)\tau}2\big(8\mathcal{M}^2\tau^2
	+|\hat{u}(x,n\tau)|^{2}+|v(x,n\tau)|^{2}\big)\mathrm{d}x
	\leq 16\mathcal{M}^2\tau^2+\hat{c_0} \leq \delta+\hat{c_0},
\end{align*}
and
\begin{equation*}
	\mathfrak{e}(\tau)
	=\int_{0}^{\tau}(|\mathfrak{g}_{j}^{n}(s)|^2+|\mathfrak{f}_{j}^{n}(s)|^2)\mathrm{ds}
	\leq \delta.
\end{equation*}
where we have used Lemma \ref{lem:error}. Hence, recall that $\mathfrak{s}_{j+1}^{n,2}(0)\tau\leq c_0$, we conclude that
\begin{equation}\label{eq:exponen}
	\mathfrak{a}\tau+\mathfrak{e}(\tau)=
	\mathfrak{s}_{j+1}^{n,2}(0)\tau+\hat{\mathfrak{s}}_{j+1}^{n,2}(0)\tau
	+\int_{0}^{\tau}(|\mathfrak{g}_{j}^{n}(s)|^2+|\mathfrak{f}_{j}^{n}(s)|^2)\mathrm{ds}
	\leq 2\delta+\hat{c_0}+c_0.
\end{equation}
Thus, plugging \eqref{eq:exponen} into \eqref{6.30.5} yields
\begin{align}\label{6.30.5-1}
	\tilde{\mathcal{L}}_{j+1}^{n,2}(\tau)
	\leq&
	\tilde{\mathcal{L}}_{j+1}^{n,2}(0)
	+2\hat{C}_1\mathrm{e}^{\hat{C}_{2}+\hat{C}_{2}(2\delta+\hat{c_0}+c_0)}
	\Big(\tilde{\mathcal{L}}_{j+1}^{n,2}(0)\tau+\tilde{\mathcal{D}}_{j+1}^{n,2}(0)\tau\Big)
\\
	&+2\hat{C}_1\mathrm{e}^{\hat{C}_{2}+\hat{C}_{2}(2\delta+\hat{c_0}+c_0)}
	\big(1+\hat{C}_{2}\tau+\hat{C}_{2}(2\delta+\hat{c_0}+c_0)\big)\int_{0}^{\tau}(|\mathfrak{g}_{j}^{n}(s)|^2+|\mathfrak{f}_{j}^{n}(s)|^2)\mathrm{d}s.\nonumber
\end{align}

Furthermore, recalling the Definition \ref{definition6.3} of $\tilde{\mathcal{D}}(n;\Delta)$ and the definition  of $\tilde{\mathcal{D}}_{j+1}^{n,2}$ in \eqref{eq:def-Dij}, and using \eqref{5.3}, \eqref{eq:4.13-1} and \eqref{eq:U-CAL}, we get
\begin{equation*}
	\sum_{j=j_{1}-n_{1}+n}^{j_{1}+n_{1}-n-2}\tilde{\mathcal{D}}_{j+1}^{n,2}(0)
	=
	\sum_{j=j_{1}-n_{1}+n}^{j_{1}+n_{1}-n-2}
	\Big(
	|\mathcal{U}_{j}^{n}|^{2}
	\big(|v_{j+2}^{n}|^{2}+|\hat{v}_{j+2}^{n}|^{2}\big)
	+|\mathcal{V}_{j+2}^{n}|^{2}
	\big(|u_{j}^{n}|^{2}+|\hat{u}_{j}^{n}|^{2}\big)
	\Big)
	=\tilde{\mathcal{D}}(n;\Delta).
\end{equation*}
From this, summing \eqref{6.30.5-1} over all integers $j$ with $j_{1}-n_{1}+n\leq j\leq j_{1}+n_{1}-n-2$ and using \eqref{eq:L-n} and \eqref{eq:L-n-0}, we obtain \eqref{corollary6.1} by setting $\hat{C}_5=2\hat{C}_1\mathrm{e}^{\hat{C}_{2}+\hat{C}_{2}(2\delta+\hat{c_0}+c_0)}
\big(1+\hat{C}_{2}+\hat{C}_{2}(2\delta+\hat{c_0}+c_0)\big)$.

\Step \label{step:smooth-2}
We estimate  $\tilde{\mathcal{Q}}(n+1;\Delta)-\tilde{\mathcal{Q}}(n;\Delta)$.
To see this, we first establish pointwise estimates for $|\hat{u}_{j+1}^{n,2}|^{2}$ and $|\hat{v}_{j+1}^{n,2}(\tau)|^{2}$. Using \eqref{eq:4.36}--\eqref{eq:4.38} and \eqref{eq:exponen} and carrying out similar arguments as in the proof of \eqref{2.10} and \eqref{2.11}, we deduce that
\begin{align*}
	|\hat{u}_{j+1}^{n,2}(\tau)|^{2}
	\leq&|\hat{u}_{j+1}^{n,2}(0)|^{2}
	+\hat{m}_1\hat{\mathfrak{s}}_{j+1}^{n,2}(0)\tau
	+\hat{c}_{1}|\hat{u}_{j+1}^{n,2}(0)|^{2}|\hat{v}_{j+1}^{n,2}(0)|^{2}\tau
	+\hat{c}_{2}\int_{0}^{\tau}(|\mathfrak{g}_{j}^{n}(s)|^2+|\mathfrak{f}_{j}^{n}(s)|^2)\mathrm{ds},\\
	|\hat{v}_{j+1}^{n,2}(\tau)|^{2}
	\leq&|\hat{v}_{j+1}^{n,2}(0)|^{2}
	+\hat{m}_1\hat{\mathfrak{s}}_{j+1}^{n,2}(0)\tau
	+\hat{c}_{1}|\hat{u}_{j+1}^{n,2}(0)|^{2}|\hat{v}_{j+1}^{n,2}(0)|^{2}\tau
	+\hat{c}_{2}\int_{0}^{\tau}(|\mathfrak{g}_{j}^{n}(s)|^2+|\mathfrak{f}_{j}^{n}(s)|^2)\mathrm{ds},
\end{align*}
where
$\hat{c}_{1}=4|\beta|\mathrm{e}^{4|\beta|(2\delta+\hat{c}_0)}$, $\hat{m}_1=(m+1)(1+\hat{c}_{1}(2\delta+\hat{c}_0))$ and $\hat{c}_{2}=\hat{m}_1+\hat{m}_1(2\delta+\hat{c}_0)$.

Let
\begin{equation*}
	\mathfrak{a}(n;\delta)=\sum_{j=j_{1}-n_{1}+n}^{j_{1}+n_{1}-n}\mathfrak{a},
	\qquad\qquad
	\mathfrak{e}(n;\delta;\tau)
	=\sum_{j=j_{1}-n_{1}+n}^{j_{1}+n_{1}-n-2}\mathfrak{e}(\tau).
\end{equation*}
Then, by the assumptions of Proposition \ref{lemma5.6}, we have
\begin{equation}\label{eq:a}
	\mathfrak{a}(n;\delta)\tau
	=\mathfrak{s}(n;\Delta)\tau
	+\hat{\mathfrak{s}}(n;\Delta)\tau
	\leq
	\mathfrak{s}(n_{0};\Delta)\tau
	+\sup\limits_{n_0\leq n\leq n_1-1}\hat{\mathfrak{s}}(n;\Delta)\tau\leq 2\delta,
\end{equation}
and
\begin{equation}\label{eq:e}
	\mathfrak{e}(n;\delta;\tau)
	=\int_{0}^{\tau}\sum_{j=j_{1}-n_{1}+n}^{j_{1}+n_{1}-n-2}(|\mathfrak{g}_{j}^{n}(s)|^2+|\mathfrak{f}_{j}^{n}(s)|^2)\mathrm{ds}
	=
	\sup\limits_{n_0\leq n\leq n_1-1}\int_{0}^{\tau}	\tilde{\mathcal{E}}(n;\Delta) \mathrm{d}s
	\leq \delta.
\end{equation}

Next, using \eqref{2.10}--\eqref{2.11} and \eqref{eq:smooth-U}--\eqref{eq:smooth-V}, it follows from \eqref{eq:U-CAL} and \eqref{eq:exponen} that
	\begin{align}\label{eq:smooth-Q-DE}
		&\tilde{\mathcal{Q}}(n+1;\Delta)\nonumber\\
		=&
		\sum_{j_{1}-n_{1}+n\leq j< k\leq j_{1}+n_{1}-n-2 }
		\left(
		\big(|\mathcal{U}_{j+1}^{n+1}|^{2}
		|v_{k+1}^{n+1}|^{2}+|\hat{v}_{k+1}^{n+1}|^{2}\big)
		+|\mathcal{V}_{k+1}^{n+1}|^{2}
		\big(|u_{j+1}^{n+1}|^{2}|\hat{u}_{j+1}^{n+1}|^{2}\big)
		\right)\nonumber\\
		\leq &
		\sum_{j_{1}-n_{1}+n\leq j< k\leq j_{1}+n_{1}-n}
		\left(
		|\mathcal{U}_{j}^{n}|^{2}
		\big(|v_{k}^{n}|^{2}+|\hat{v}_{k}^{n}|^{2}\big)
		+|\mathcal{V}_{k}^{n}|^{2}
		\big(|u_{j}^{n}|^{2}+|\hat{u}_{j}^{n}|^{2}\big)
		\right)\nonumber\\
		&-\sum_{j_{1}-n_{1}+n\leq j< k\leq j_{1}+n_{1}-n-2}
		\left(
		|\mathcal{U}_{j}^{n}|^{2}
		\big(|v_{j+2}^{n}|^{2}+|\hat{v}_{j+2}^{n}|^{2}\big)
		+|\mathcal{V}_{j+2}^{n}|^{2}
		\big(|u_{j}^{n}|^{2}+|\hat{u}_{j}^{n}|^{2}\big)
		\right)\nonumber\\
		&+
		\hat{C}_6
		\big(\mathfrak{a}(n;\delta)\tau+\mathfrak{e}(n;\delta;\tau) +\mathfrak{q}(n;\Delta)\tau+\hat{\mathfrak{q}}(n;\Delta)\tau\big)
		\tilde{\mathcal{L}}(n;\Delta)\nonumber\\
		&+
		\hat{C}_6
		\big(\mathfrak{a}(n;\delta)\tau+\mathfrak{e}(n;\delta;\tau) +\mathfrak{q}(n;\Delta)\tau^2+\hat{\mathfrak{q}}(n;\Delta)\tau^2\big)
	   \tilde{\mathcal{D}}(n;\Delta)\nonumber\\
	    &+
	    \hat{C}_6
	    \big(\mathfrak{a}(n;\delta)+\mathfrak{e}(n;\delta;\tau) +\mathfrak{q}(n;\Delta)\tau+\hat{\mathfrak{q}}(n;\Delta)\tau\big)\int_{0}^{\tau}\tilde{\mathcal{E}}(n;\Delta)\mathrm{d}s\nonumber\\
	    \leq &
	    \tilde{\mathcal{Q}}(n;\Delta)-\tilde{\mathcal{D}}(n;\Delta)+
	    \hat{C}_6
	    \big(\mathfrak{a}(n;\delta)\tau+\mathfrak{e}(n;\delta;\tau) +\mathfrak{q}(n;\Delta)\tau^2+\hat{\mathfrak{q}}(n;\Delta)\tau^2\big)
	    \tilde{\mathcal{D}}(n;\Delta)\nonumber\\
	    &+
	    \hat{C}_6
	    \big(\mathfrak{a}(n;\delta)\tau+\mathfrak{e}(n;\delta;\tau) +\mathfrak{q}(n;\Delta)\tau+\hat{\mathfrak{q}}(n;\Delta)\tau\big)
	    \tilde{\mathcal{L}}(n;\Delta)\nonumber\\
	    &+
	    \hat{C}_6
	    \big(\mathfrak{a}(n;\delta)+\mathfrak{e}(n;\delta;\tau) +\mathfrak{q}(n;\Delta)\tau+\hat{\mathfrak{q}}(n;\Delta)\tau\big)
	    \int_{0}^{\tau}\tilde{\mathcal{E}}(n;\Delta)\mathrm{d}s,
		\end{align}
where $\hat{C}_6=4\hat{c}_3(m+\hat{m}_1+1+c_1+\hat{c}_1)+\hat{c}_2(1+\hat{c}_3)$,
$\hat{c}_3=\hat{C}_3\mathrm{e}^{\hat{C}_3(2\delta+\hat{c}_0)+c_0}\big(1+\hat{C}_4+\hat{C}_4(2\delta+\hat{c}_0)+c_0)\big)$.

Thus, substituting \eqref{eq:a} and \eqref{eq:e} into \eqref{eq:smooth-Q-DE}, and  choosing $\delta>0$ sufficiently small such that $-1+\hat{C}_6(2\delta+\delta+2\delta^2)<-\frac{1}{2}$, we obtain
\begin{align}\label{eq:smooth-Q}
	\tilde{\mathcal{Q}}(n+1;\Delta)\tau^2-\tilde{\mathcal{Q}}(n;\Delta)\tau^2\nonumber
	\leq &
	\big(-1+\hat{C}_6(2\delta+\delta+2\delta^2)\big)\tilde{\mathcal{D}}(n;\Delta)\tau^2
	\nonumber\\
	 &+
	\hat{C}_6\tau
	\big(\mathfrak{a}(n;\delta)\tau+\mathfrak{e}(n;\delta;\tau) +\mathfrak{q}(n;\Delta)\tau+\hat{\mathfrak{q}}(n;\Delta)\tau\big)
	\tilde{\mathcal{L}}(n;\Delta)\tau\nonumber\\
	&+
	\hat{C}_6\tau
	\big(\mathfrak{a}(n;\delta)\tau+\mathfrak{e}(n;\delta;\tau) +\mathfrak{q}(n;\Delta)\tau^2+\hat{\mathfrak{q}}(n;\Delta)\tau^2\big)
	\int_{0}^{\tau}\tilde{\mathcal{E}}(n;\Delta)\mathrm{d}s\nonumber\\
	\leq &-\frac{1}{2}\tilde{\mathcal{D}}(n;\Delta)\tau^2
	+\hat{C}_6
	\big(3\delta+2\delta^2\big)
	\tau\int_{0}^{\tau}\tilde{\mathcal{E}}(n;\Delta)\mathrm{d}s\nonumber\\
	&+\hat{C}_6
	\big(3\delta\tau +\mathfrak{q}(n;\Delta)\tau^2+\hat{\mathfrak{q}}(n;\Delta)\tau^2\big)
	\tilde{\mathcal{L}}(n;\Delta)\tau.
\end{align}

\Step \label{step:smooth-3}
Using \eqref{corollary6.1} and \eqref{eq:smooth-Q}, we have
	\begin{align*}
		&\tilde{\mathcal{F}}(n+1;\Delta)-\tilde{\mathcal{F}}(n;\Delta)\\
		\leq&
		\hat{C}_5\tau
		\Big(\tilde{\mathcal{L}}(n;\Delta)\tau+\tilde{\mathcal{D}}(n;\Delta)\tau+\int_{0}^{\tau}\tilde{\mathcal{E}}(n;\Delta)\mathrm{d}s\Big)
		-\frac{\tilde{\kappa }}{2}\tilde{\mathcal{D}}(n;\Delta)\tau^2
		+\hat{C}_6\tilde{\kappa }
		\big(3\delta+2\delta^2\big)
		\tau\int_{0}^{\tau}\tilde{\mathcal{E}}(n;\Delta)\mathrm{d}s\nonumber\\
		&+\hat{C}_6\tilde{\kappa }
		\big(3\delta\tau +\mathfrak{q}(n;\Delta)\tau^2+\hat{\mathfrak{q}}(n;\Delta)\tau^2\big)
		\tilde{\mathcal{L}}(n;\Delta)\tau.
	\end{align*}
Therefore, we can choose $\tilde{\kappa }>0$ large enough such that $-\frac{\tilde{\kappa }}{2}+\hat{C}_5<-1$. Then
\begin{align*}
	\tilde{\mathcal{F}}(n+1;\Delta)
	\leq &
	\big(1+(\hat{C}_5+3\delta\hat{C}_6\tilde{\kappa })\tau
	+\hat{C}_6\tilde{\kappa }
	\big(\mathfrak{q}(n;\Delta)\tau^2+\hat{\mathfrak{q}}(n;\Delta)\tau
	^2\big)\big)
\tilde{\mathcal{F}}(n;\Delta)\\
	&+\big(\hat{C}_5+\hat{C}_6\tilde{\kappa }
	\big(3\delta+2\delta^2\big)\big)\tau\int_{0}^{\tau}\tilde{\mathcal{E}}(n;\Delta)\mathrm{d}s,
\end{align*}
which gives \eqref{eq:F-1-UNI} with $\hat{C}_7=\hat{C}_5+3\delta\hat{C}_6\tilde{\kappa }+\hat{C}_6\tilde{\kappa }+\hat{C}_6\tilde{\kappa }
\big(3\delta+2\delta^2\big)$.

Since $1+x\leq \mathrm{e}^x$ and $n\leq (T+1)/\tau$, we deduce that
\begin{align*}
\tilde{\mathcal{F}}(n;\Delta)
\leq&
\mathrm{e}^{\hat{C}_7\tau+
\hat{C}_7 (\mathfrak{q}(n;\Delta)\tau^2+\hat{\mathfrak{q}}(n;\Delta)\tau
^2)}
\tilde{\mathcal{F}}(n-1;\Delta)
+\hat{C}_7\tau\int_{0}^{\tau}\tilde{\mathcal{E}}(n;\Delta)\mathrm{d}s\\
\leq&
\mathrm{e}^{\hat{C}_7(n-n_0)\tau
+\hat{C}_7\sum_{p=n_{0}}^{n-1}
\big(\mathfrak{q}(p;\Delta)\tau^2+\hat{\mathfrak{q}}(p;\Delta)\tau
^2\big)}
\Big(\tilde{\mathcal{F}}(n_0;\Delta)+\hat{C}_7\tau\int_{0}^{\tau}\sum_{p=n_{0}}^{n-1}\tilde{\mathcal{E}}(p;\Delta)\mathrm{d}s\Big),
\end{align*}
which together with Lemma~\ref{lem:nonlinear-estimates} gives \eqref{eq:smooth-F}.
The proof is complete.
\end{proof}

Now we consider the difference $(\mathcal{U},\mathcal{V})
=(\hat{u}-u^{(\tau)},\hat{v}-v^{(\tau)})$ between two consecutive time steps, namely, $t\in[n\tau,(n+1)\tau)$.
\begin{lemma}\label{lem:0-to-tau}
	For all $s\in[0,\tau]$, there holds that
		\begin{align*}
		&|\mathcal{U}(x,n\tau+s)|^2
		\leq
		\mathrm{e}^{2\tau}|\mathcal{U}(x-s,n\tau)|^2
		+\mathrm{e}^{2\tau}\big(m^2+(|\alpha|+4|\beta|)^2\big)\mathcal{M}_0^6(T)\tau,\\
		&|\mathcal{V}(x,n\tau+s)|^2
		\leq
		\mathrm{e}^{2\tau}|\mathcal{V}(x+s,n\tau)|^2
		+\mathrm{e}^{2\tau}\big(m^2+(|\alpha|+4|\beta|)^2\big)\mathcal{M}_0^6(T)\tau.
	\end{align*}
\end{lemma}
\begin{proof}
	By \eqref{eq:NLDE} and \eqref{eq:subproblem1}, we have
	\begin{equation*}
		\left\{
		\begin{aligned}
			&\frac{\mathrm{d}}{\mathrm{d}s}|\mathcal{U}(x+s,n\tau+s)|^2
			=2\mathrm{Re}\left\lbrace[im\hat{v}^{(\tau)}\overline{\mathcal{U}^{(\tau)}}+i\mathcal{N}_{1}(\hat{u}^{(\tau)},\hat{v}^{(\tau)})\overline{\mathcal{U}}](x+s,n\tau+s)
			\right\rbrace,\\
			&\frac{\mathrm{d}}{\mathrm{d}s}|\mathcal{V}(x-s,n\tau+s)|^2
			=2\mathrm{Re}\left\lbrace[im\hat{u}^{(\tau)}\overline{\mathcal{V}}+i\mathcal{N}_{2}(\hat{u}^{(\tau)},\hat{v}^{(\tau)})\overline{\mathcal{V}}](x-s,n\tau+s)\right\rbrace.
		\end{aligned}\right.
	\end{equation*}
Then for $s\in[0,\tau)$, using Young's inequality and Gronwall's inequality, we have
	\begin{align*}
		&|\mathcal{U}(x+s,n\tau+s)|^2
		\leq
		\mathrm{e}^{2s}\left( |\mathcal{U}(x,n\tau)|^2
		+\big(m^2\mathcal{M}_0^2+(|\alpha|+4|\beta|)^2\mathcal{M}_0^6\big)s\right),\\
		&|\mathcal{V}(x-s,n\tau+s)|^2
		\leq
		\mathrm{e}^{2s}\left( |\mathcal{V}(x,n\tau)|^2
		+\big(m^2\mathcal{M}_0^2+(|\alpha|+4|\beta|)^2\mathcal{M}_0^6\big)s\right),
	\end{align*}
	which completes the proof.
\end{proof}

To estimate $
(\mathcal{U},\mathcal{V})
=(\hat{u}-u^{(\tau)},\hat{v}-v^{(\tau)})
$ for any $t\in [0,T]$,
by Proposition \ref{lemma5.6} and Lemma \ref{lem:0-to-tau}, we have the following.
\begin{corollary}\label{coro:unique}
Under the assumptions of Proposition \ref{lemma5.6}, there exists a constant $\hat{C}(T)>0$,  which depends only on $T$, $\delta$, $c_0$ and the system \eqref{eq:NLDE}, but is independent of $\mathcal{M}_0$, $\mathcal{M}$ and $\tau$, such that for all $\tau\in(0,\hat{\tau}_{\sharp})$ and $t\in[n_0\tau,T]$, it holds that
	\begin{align*}
		&\int_{j_1\tau-n_1\tau+t}^{j_1\tau+n_1\tau-t}
		\left(|\mathcal{U}(x,t)|^2
		+|\mathcal{V}(x,t)|^2
		\right)\mathrm{d}x\\
		\leq&
		\hat{C}(T)
	    \tilde{\mathcal{F}}(n_0;\Delta)
		+\hat{C}(T)\tau\int_{0}^{\tau}\Big(\sum_{p=n_{0}}^{n-1}\sum_{j=-\infty}^{+\infty}\big( |\mathfrak{g}_{j}^{p}(s)|^{2}
		+|\mathfrak{f}_{j}^{p}(s)|^{2}\big)\Big)\mathrm{d}s
		+\hat{C}(T)\mathcal{M}_0^6\mathcal{M}^2(n_1-n_0)\tau^2.
	\end{align*}
\end{corollary}
\begin{proof}
For	$t\in(n\tau,(n+1)\tau)\cap[n_0\tau,T]$, Lemma \ref{lem:0-to-tau} implies
\begin{align}\label{eq:t}
	\hspace*{-0.3cm}&\int_{j_1\tau-n_1\tau+t}^{j_1\tau+n_1\tau-t}
		\left(|\mathcal{U}(x,t)|^2
		+|\mathcal{V}(x,t)|^2
		\right)\mathrm{d}x
	\nonumber\\
	\hspace*{-0.3cm}\leq&	\int_{j_1\tau-n_1\tau+n\tau}^{j_1\tau+n_1\tau-n\tau}
		\left(
		|\mathcal{U}(x,n\tau)|^2
		+ |\mathcal{V}(x,n\tau)|^2
		\right)
		\mathrm{d}x
		+4\mathrm{e}^{2\tau}\big(m^2+(|\alpha|+4|\beta|)^2\big)\mathcal{M}_0^6(n-n_0)\tau^2.
	\end{align}
	
On the other hand,  to estimate the first term on the right hand of \eqref{eq:t}, we consider the time step $t=n\tau$ for $n=0,1,2,\ldots$ with $0\leq n_{0}\leq n$ and  $n\leq (T+1)/\tau$. Then, by \eqref{eq:error1}, \eqref{eq:error2} and Proposition~\ref{lemma5.6}, we have
\begin{align}
	&\int_{j_1\tau-n_1\tau+n\tau}^{j_1\tau+n_1\tau-n\tau}
	\left( |\mathcal{U}(x,n\tau)|^2+ |\mathcal{V}(x,n\tau)|^2\right)
	\mathrm{d}x\nonumber\\
	\leq &2\sum_{j=j_1-n_1+n}^{j_1+n_1-n}
	\left( |\mathcal{U}_{j}^{n}|^2+\mathcal{M}^2\tau^2 +|\mathcal{V}_{j}^{n}|^2+\mathcal{M}^2\tau^2\right)\tau\nonumber\\
	\leq&
	2\mathrm{e}^{\hat{C}_8(T)}
	\Big(\tilde{\mathcal{F}}(n_0;\Delta)+\hat{C}_7\tau\int_{0}^{\tau}\sum_{p=n_{0}}^{n-1}\tilde{\mathcal{E}}(p;\Delta)\mathrm{d}s\Big)
	+8\mathcal{M}^2(n_1-n_0)\tau^3.\label{eq:n-tau}
\end{align}

Substituting \eqref{eq:n-tau} into \eqref{eq:t}, and setting $\hat{C}(T)=2\mathrm{e}^{2+\hat{C}_8(T)}(1+\hat{C}_7)+8\mathrm{e}^{2}+4\mathrm{e}^{2}\big(m^2+(|\alpha|+4|\beta|)^2\big)$, we complete the proof.
\end{proof}
	
The remainder of this section will be devoted to the proof of the main theorem.
\begin{proof}[\textbf{\textup{Proof of Theorem \ref{mainresult}}}]
To prove Theorem \ref{mainresult}, by Lemma \ref{lem:proposition4.1}, it suffices to prove that
$
(u_\flat,v_\flat)(x,t)=(\hat{u},\hat{v})(x,t)$ in  $\mathrm{L}^2(\mathbb{R}\times[0,T]).
$

First, for a fixed $\mathrm{k}$, we estimate the difference between the smooth solution
$(\hat{u}_{\mathrm{k}},\hat{v}_{\mathrm{k}})$ and the limit $(u_\flat,v_\flat)$ of any convergent subsequence of the time splitting solutions $\{(u^{(\tau_l)},v^{(\tau_l)})\}_{l=1}^{+\infty}$ with $\tau_l\to0$ as $l\to+ \infty$.
Here $\{(\hat{u}_{\mathrm{k}},\hat{v}_{\mathrm{k}})\}_{k=1}^{+\infty}$ is the sequence of solutions with \eqref{5.1} and \eqref{eq:uniqueness-s}.

Notice that $(\hat{u}_{\mathrm{k}}(x,0),\hat{v}_{\mathrm{k}}(x,0))\in \mathrm{C}_{\mathrm{c}}^\infty(\mathbb{R}\times [0,T])$ for any $k$. Without loss of generality, we assume that
\begin{equation*}
	\mathrm{supp}(\hat{u}_{\mathrm{k}}(\cdot,0))\subset[-\Theta_k,\Theta_k],\quad
	\mathrm{supp}(\hat{v}_{\mathrm{k}}(\cdot,0))\subset[-\Theta_k,\Theta_k],
\end{equation*}
where $\mathrm{supp}(\omega)$ denotes the support of $\omega$ and $\Theta_k\geq\Theta$.
It is shown in \cite[Lemma 3.1]{Y.ZhangandQ.Zhao1} that the smooth solutions $(\hat{u}_{\mathrm{k}},\hat{v}_{\mathrm{k}})$, $k=1,2,\ldots$, have compact support in $\mathbb{R}\times [0,T]$. More precisely,
\begin{equation*}
	\mathrm{supp}(\hat{u}_{\mathrm{k}}(\cdot,t))\subset[-\Theta_k-t,\Theta_k+t],\quad
	\mathrm{supp}(\hat{v}_{\mathrm{k}}(\cdot,t))\subset[-\Theta_k-t,\Theta_k+t],\quad \text{ for } t\geq 0.
\end{equation*}
Let
\begin{align*}
	\mathcal{M}_{0,\mathrm{k}}
	=\max_{\mathbb{R}\times[0,T]}
	(|\hat{u}_\mathrm{k}|+|\hat{v}_\mathrm{k}|+1),\quad
	\mathcal{M}_\mathrm{k}=\max_{\mathbb{R}\times[0,T]}(|\partial_t\hat{u}_\mathrm{k}|
	+|\partial_x\hat{u}_\mathrm{k}|
	+|\partial_t\hat{v_\mathrm{k}}|+|\partial_x\hat{v}_\mathrm{k}|+1).
	\end{align*}

Given a constant $\delta>0$, it is shown in \cite[Lemmas 5.2--5.3]{N.Li} that there exists a $\hat{\sigma}>0$, independent of $\mathrm{k}$ and $\tau_l$, such that for any $t\in[0,T]$ and any interval $(x_1,x_2)\subset [-\Theta_k-4T,\Theta_k+4T]$ with $|x_1-x_2|<4\hat{\sigma}$, we have
	\begin{equation}\label{eq:uni-small-1}
		\sup_{\mathrm{k}}\int_{x_1}^{x_2}
		\big(|\hat{u}_{\mathrm{k}}(x,t)|^2+|\hat{v}_{\mathrm{k}}(x,t)|^2\big)\mathrm{d}x
		\leq\frac{ \delta}{4}.
	\end{equation}
Then, using \eqref{eq:initial-data-convergent} and \eqref{eq:interval-any}, there exists a $\sigma\in(0,\min\{\hat{\tau}_\sharp,\hat{\sigma}\})$, independent of $\mathrm{k}$ and $\tau_l$, such that for any interval $[x_1,x_2]\subset [-\Theta_k-4T,\Theta_k+4T]$ with $|x_1-x_2|\leq 4\sigma$, there holds that
\begin{equation}\label{eq:uni-small}
	\int_{x_1}^{x_2}\big(|u^{(\tau_l)}|^{2}+|v^{(\tau_l)}|^{2}\big)(x,t)\mathrm{d}x
	\leq\delta,\qquad n=0,1,2,\ldots,[T/\sigma].
\end{equation}
Without loss of generality, we further assume that $\Theta_k+4T=2N\sigma$.

Next, we consider the time interval $t\in[n\sigma,(n+1)\sigma]$ for $n=0,1,2,\ldots,[T/\sigma]$.
By \eqref{5.1}, there exists a constant $\hat{c}_0>0$, independent of $\mathrm{k}$ and $\tau_l$, such that
$\|\hat{u}_{\mathrm{k}}(x,0)\|_{\mathrm{L}^{2}(\mathbb{R})} +
\|\hat{v}_{\mathrm{k}}(x,0))\|_{\mathrm{L}^{2}(\mathbb{R})}\leq \hat{c}_0.$
Since $|u_{\mathrm{k}}(x,t)-u_{\mathrm{k}}(j\tau_l,t)|\leq \mathcal{M}_k\tau_l$ for $x\in[j\tau_l,(j+1)\tau_l]$ and $t\in[0,T]$, it follows from \cite[Lemma 2.1]{Y.ZhangandQ.Zhao1} that
\begin{align*}
&\sum_{j=j_1-n_1+n}^{j=j_1+n_1-n} \big(|u_{\mathrm{k}}(j\tau_l,t)|^{2}+|v_{\mathrm{k}}(j\tau_l,t)|^{2}\big)\tau_l\\
\leq &
\sum_{j=j_1-n_1+n}^{j=j_1+n_1-n-1}\int_{j\tau_l}^{(j+1)\tau_l}
2\big(2\mathcal{M}_k^2\tau_l^2+|u_{\mathrm{k}}(x,t)|^{2}
+|v_{\mathrm{k}}(x,t)|^{2}\big)\mathrm{d}x
\leq  4\mathcal{M}_k^2\tau_l^2+\hat{c}_0.
\end{align*}
Thus, carrying out similar arguments as in the proof of \eqref{eq:F-L}, for any given constant $\delta>0$, there exists a constant $0<\hat{\tau}_*<\frac{\delta}{2\mathcal{M}_k}$, such that, for all $\tau\in(0,\hat{\tau}_*)$, we have
\begin{equation}\label{eq:4.52}
	\tilde{\mathcal{F}}(n\sigma;\Delta)
	=\tilde{\mathcal{L}}(n\sigma;\Delta)\tau_l
	+\tilde{\kappa}\tilde{\mathcal{Q}}(n\sigma;\Delta)\tau_l^{2}
	\leq (c_0\kappa+\hat{c}_0\tilde{\kappa}+\delta^2+1)\tilde{\mathcal{L}}(n\sigma;\Delta)\tau_l.
\end{equation}

Let $\hat{C}_1(T)=\hat{C}(T)(c_0\kappa+\hat{c}_0\tilde{\kappa}+\delta^2+1)$, where  $\hat{C}(T)$ is given by Corollary \ref{coro:unique}. Hence, the constant $\hat{C}_1(T)$ is independent of $\mathcal{M}_{0,\mathrm{k}}$, $\mathcal{M}_{\mathrm{k}}$ and $\tau_l$.
Using \eqref{eq:uni-small-1}--\eqref{eq:uni-small} and applying Corollary \ref{coro:unique} to $(\hat{u}_{\mathrm{k}},\hat{v}_{\mathrm{k}})$ and $(u^{(\tau_l)},v^{(\tau_l)})$ over the domain $\Lambda((j+2)\sigma,(n+2)\sigma;n\sigma)$, we deduce that, for all $\tau\in(0,\min\{\tau_\sharp,\hat{\tau},\hat{\tau}_\sharp,\hat{\tau}_*\})$,
	\begin{align*}
&\sup_{n\sigma \leq t\leq (n+2)\sigma}	
\int_{j\sigma+t-n\sigma}^{(j+4)\sigma-t+n\sigma}
	\left(|\hat{u}_{\mathrm{k}}(x,t)-u^{(\tau_l)}(x,t)|^2
	+|\hat{v}_{\mathrm{k}}(x,t)-v^{(\tau_l)}(x,t)|^2
	\right)\mathrm{d}x
	\\
\leq 	
&\hat{C}_1(T)
\int_{j\sigma}^{(j+4)\sigma}
\left(|\hat{u}_{\mathrm{k}}(x,n\sigma)-u^{(\tau_l)}(x,n\sigma)|^2
+|\hat{v}_{\mathrm{k}}(x,n\sigma)-v^{(\tau_l)}(x,n\sigma)|^2
\right)\mathrm{d}x\\
&+\hat{C}_1(T)\tau_l\int_{0}^{\tau_l}\Big(\sum_{p=n_{0}}^{n-1}\sum_{j=-\infty}^{+\infty}\big( |\mathfrak{g}_{j}^{p}|^{2}
+|\mathfrak{f}_{j}^{p}|^{2}\big)\Big)\mathrm{d}s
+4\hat{C}_1(T)\mathcal\mathcal{M}_{0,\mathrm{k}}^6\mathcal{M}_{\mathrm{k}}^2\sigma\tau_l.
\end{align*}
By Lemma \ref{lem:error} and \eqref{eq:limit-solution}, letting $l\to +\infty$, i.e. $\tau_l\to 0^+$, we obtain
\begin{align*}
	&\sup_{n\sigma \leq t\leq (n+2)\sigma}	
	\int_{j\sigma+t-n\sigma}^{(j+4)\sigma-t+n\sigma}
	\left(|\hat{u}_{\mathrm{k}}(x,t)-u_\flat(x,t)|^2
	+|\hat{v}_{\mathrm{k}}(x,t)-v_\flat(x,t)|^2
	\right)\mathrm{d}x\\
	\leq &
	\hat{C}_1(T)
	\int_{j\sigma}^{(j+4)\sigma}
	\left(|\hat{u}_{\mathrm{k}}(x,n\sigma)-u_\flat(x,n\sigma)|^2
	+|\hat{v}_{\mathrm{k}}(x,n\sigma)-v_\flat(x,n\sigma)|^2
	\right)\mathrm{d}x.
\end{align*}
Summing the above inequality over all integers $j$ with$-2N\leq j\leq 2N-4$ yields
\begin{align*}
	&\sup_{n\sigma \leq t\leq (n+2)\sigma}	
	\int_{-\Theta_k-4T+t}^{\Theta_k+4T-t}
	\left(|\hat{u}_{\mathrm{k}}(x,t)-u_\flat(x,t)|^2
	+|\hat{v}_{\mathrm{k}}(x,t)-v_\flat(x,t)|^2
	\right)\mathrm{d}x\\
	\leq &
	4\hat{C}_1(T)
	\int_{-\Theta_k-4T+n\sigma}^{\Theta_k+4T-n\sigma}
	\left(|\hat{u}_{\mathrm{k}}(x,n\sigma)-u_\flat(x,n\sigma)|^2
	+|\hat{v}_{\mathrm{k}}(x,n\sigma)-v_\flat(x,n\sigma)|^2
	\right)\mathrm{d}x.
\end{align*}
Since $\hat{u}_{\mathrm{k}}(x,t)=\hat{v}_{\mathrm{k}}(x,t)=0$ for $|x|> \Theta_k+t$, we deduce from \eqref{2.4} that
\begin{align} \label{eq:substi-2}
	&\lim\limits_{l\to +\infty}\sup_{n\sigma \leq t\leq (n+2)\sigma}	
	\int_{|x|> \Theta_k+t}
	\left(|u^{(\tau_l)}(x,t)|^2
	+|v^{(\tau_l)}(x,t)|^2
	\right)\mathrm{d}x\nonumber\\
	\leq &\lim\limits_{l\to +\infty}\int_{|x|> \Theta_k+n\sigma}
	\left(|u^{(\tau_l)}(x,n\sigma)|^2
	+|v^{(\tau_l)}(x,n\sigma)|^2
	\right)\mathrm{d}x\nonumber\\
	\leq&
    \int_{|x|> \Theta_k+n\sigma}
    \left(|\hat{u}_{\mathrm{k}}(x,n\sigma)-u_\flat(x,n\sigma)|^2
    +|\hat{v}_{\mathrm{k}}(x,n\sigma)-v_\flat(x,n\sigma)|^2
    \right)\mathrm{d}x\nonumber.
\end{align}
Consequently, we conclude that
\begin{align*}
	&\sup_{n\sigma \leq t\leq (n+1)\sigma}	
	\int_{\mathbb{R}}
	\left(|\hat{u}_{\mathrm{k}}(x,t)-u_\flat(x,t)|^2
	+|\hat{v}_{\mathrm{k}}(x,t)-v_\flat(x,t)|^2
	\right)\mathrm{d}x\\
	\leq &
	(4\hat{C}_1(T)+1)
	\int_{\mathbb{R}}
	\left(|\hat{u}_{\mathrm{k}}(x,n\sigma)-u_\flat(x,n\sigma)|^2
	+|\hat{v}_{\mathrm{k}}(x,n\sigma)-v_\flat(x,n\sigma)|^2
	\right)\mathrm{d}x.
\end{align*}
Let $\left[\frac{T}{\sigma}\right]$ denote the greatest integer less than or equal to $\frac{T}{\sigma}$. Since $0\leq n\leq \left[\frac{T}{\sigma}\right]$ and $\sigma>0$ is fixed in the sequel, we obtain
\begin{align*}
	&\sup_{0 \leq t\leq T}	
	\int_{\mathbb{R}}
	\left(|\hat{u}_{\mathrm{k}}(x,t)-u_\flat(x,t)|^2
	+|\hat{v}_{\mathrm{k}}(x,t)-v_\flat(x,t)|^2
	\right)\mathrm{d}x\\
	\leq &
	(4\hat{C}_1(T)+1)^{\left[\frac{T}{\sigma}\right]+1}
	\int_{\mathbb{R}}
	\left(|\hat{u}_{\mathrm{k}}(x,0)-u_\flat(x,0)|^2
	+|\hat{v}_{\mathrm{k}}(x,0)-v_\flat(x,0)|^2
	\right)\mathrm{d}x.
\end{align*}
Thus, integrating the above inequality over $[0,T]$ with respect to $t$, we get
\begin{equation}\label{eq:final}
\begin{aligned}
	&\int_{0}^{T}
	\int_{\mathbb{R}}
	\left(|\hat{u}_{\mathrm{k}}(x,t)-u_\flat(x,t)|^2
	+|\hat{v}_{\mathrm{k}}(x,t)-v_\flat(x,t)|^2
	\right)\mathrm{d}x\mathrm{d}t\\
	\leq &
	(4\hat{C}_1(T)+1)^{\left[\frac{T}{\sigma}\right]+1}
	T\int_{\mathbb{R}}
	\left(|\hat{u}_{\mathrm{k}}(x,0)-u_\flat(x,0)|^2
	+|\hat{v}_{\mathrm{k}}(x,0)-v_\flat(x,0)|^2
	\right)\mathrm{d}x.
\end{aligned}
\end{equation}

Finally, by \eqref{eq:initial-data-convergent} and \eqref{5.1}--\eqref{eq:limit-solution}, we obtain
\begin{equation}\label{eq:final-0}
	(u_\flat,v_\flat)(x,0)=(\hat{u},\hat{v})(x,0)\quad \text{ in } \:  \mathrm{L}^2(\mathbb{R}).
\end{equation}
Therefore, letting $\mathrm{k}\to+\infty$ in \eqref{eq:final} and using \eqref{eq:final-0}, we conclude that
\begin{equation*}
	(u_\flat,v_\flat)(x,t)=(\hat{u},\hat{v})(x,t)\qquad \text{ in }\:  \mathrm{L}^2(\mathbb{R}\times[0,T])
\end{equation*}
 for any $T>0$.
The proof of Theorem~\ref{mainresult} is thus complete.
\end{proof}

\appendix
\section{Some auxiliary lemmas}

In this appendix, we first present a compactness result, which is used to prove Lemma \ref{lem:proposition4.1}.

\begin{lemma}[\text{Ascoli characterization of compact sets in $\mathrm{C}([0,T]; B)$, see \cite[Lemma 1]{Simon1986}}]\label{lem:A2}
Let $B$ is a Banach space. A set $\mathcal{F}$ of $\mathrm{C}([0,T];B)$ is relatively compact if and only if:
\begin{enumerate}
    \item [\textup{(i)}]
$
\mathcal{F}(t)=\{f(t): f\in\mathcal{F}\} \ \text{is relatively compact in } B,\quad \forall\,0<t<T;
$
\item [\textup{(ii)}]
$
\mathcal{F} \ \text{is uniformly equicontinuous: } \forall \varepsilon>0,\ \exists \eta \ \text{such that}
$
\[
\|f(t_2)-f(t_1)\|_B \le \varepsilon,\quad \forall f\in\mathcal{F},\ \forall\,0<t_1\le t_2<T \ \text{such that } |t_2-t_1|\le \eta.
\]
\end{enumerate}
\end{lemma}

Next, we recall the following global existence result for strong solutions to the NLDE Cauchy problem \eqref{eq:NLDE}--\eqref{eq:NLDE-initail-data}.
\begin{lemma}[\text{Global existence of strong solution, see \cite[Theorem~1.1]{Y.ZhangandQ.Zhao1}}]\label{lem:A3}
	For $(u_0,v_0)\in\mathrm{L}^2(\mathbb{R})$, the NLDE Cauchy problem \eqref{eq:NLDE}--\eqref{eq:NLDE-initail-data} admits a unique global strong solution $(\hat{u},\hat{v})$ with $(\hat{u},\hat{v}) \in \mathrm{C}([0,\infty);\mathrm{L}^{2}(\mathbb{R}))$. Moreover, $|\hat{u}|^2|\hat{v}|^2\in \mathrm{L}^{2}(\mathbb{R}\times[0,T])$ for any $T>0$.
\end{lemma}

{\bf Acknowledgments.}
The research of Yongqian Zhang is supported in part by the NSFC Project (11421061, 12271507) and by Natural Science Foundation of Shanghai 15ZR1403900. The research of Qin Zhao is supported in part by the NSFC Project (12101471, 12272297).


\end{document}